\documentclass[11pt,reqno,hidelinks]{amsart}
\usepackage{amsthm, amsmath,amsfonts,amssymb,euscript,hyperref,graphics,color}
\usepackage{graphicx}
\usepackage[square, comma, sort&compress, numbers]{natbib}
\usepackage{subfigure}
\usepackage{comment}
\usepackage{import}
\usepackage{tikz}
\usepackage{latexsym}

\usepackage{ulem}
\usepackage{todonotes}
\usepackage{slashed}

\usepackage[numbers,sort&compress]{natbib}

\makeatother
\newtheorem{theorem}{Theorem}[section]
\newtheorem{lemma}[theorem]{Lemma}

\newtheorem{proposition}[theorem]{Proposition}

\theoremstyle{definition}
\newtheorem{definition}[theorem]{Definition}

\theoremstyle{remark}
\newtheorem{remark}[theorem]{Remark}

\newcommand{\nnb}{\nonumber}

\newcommand \del \partial
\newcommand \be {\begin{equation}}
\newcommand \ee {\end{equation}}
\newcommand \bes {\begin{equation*}}
\newcommand \ees {\end{equation*}}

\numberwithin{equation}{section}

\allowdisplaybreaks[4]

\def\ub{\underline{u}}
\def\Lb{\underline{L}}
\def\Cb{\underline{C}}
\def\Eb{\underline{E}}

\def\Qb{\underline{Q}}
\def\fb{\underline{f}}
\def\Pb{\underline{P}}

\def\po{\overline{\partial}}
\def\Eo{\overline{E}}
\def\io{\overline{i}}
\def\jo{\overline{j}}
\def\deo{\overline{\delta}}

\def\p{\partial}
\newcommand{\di}{\mathrm{d}} 

\newcommand \bei  {\begin{itemize}}
\newcommand \eei {\end{itemize}}
\allowdisplaybreaks[4]
\begin{document}
\title[Faddeev model with large data]{Global existence of smooth solution to evolutionary Faddeev model with short-pulse data}

\vspace{10mm}
\author[S. Luo]{Shaoying Luo}
\address{School of Statistics and Data Science, Ningbo University of Technoligy, Ningbo 315000, China.}\email{ luoshaoying@nbut.edu.cn}

\author[J. Wang]{Jinhua Wang}
\address{School of Mathematical Sciences, Xiamen University, Xiamen 361005, China.}\email{ wangjinhua@xmu.edu.cn}

\author[C. Wei]{Changhua Wei}
\address{Corresponding author: Department of Mathematics, Zhejiang Sci-Tech University, Hangzhou, 310018, China.\quad Nanbei Lake Institute for Artificial Intelligence in Medicine, Jiaxing, 314300, China.}\email{chwei@zstu.edu.cn}

\date{}

\begin{abstract}
This paper is concerned with the Cauchy problem of the evolutionary Faddeev model, a system that maps from the Minkowski space $\mathbb{R}^{1+3}$ to the unit sphere $\mathbb{S}^2$. The model is a system of nonlinear wave equations whose nonlinearities exhibit a null structure and include semilinear terms, quasilinear terms, and the unknowns themselves. By considering a class of large initial data (in energy norm) of the short pulse type, we prove that the evolutionary Faddeev model admits a globally smooth solution via energy estimates. The main result is achieved through the selection of appropriate multipliers that are specially adapted to the geometry of the system.
\end{abstract}

\maketitle
\vskip 3mm


\tableofcontents


\section{Introduction}\label{sec:1}

In quantum field theory, the Faddeev model \cite{Fad1, Fad2} was proposed to model the elementary heavy particles by topological solitons, and may be viewed as a constrained or refined Skyrme model \cite{C-M}. It can be identified with a map $\mathbf{n}$ from the Minkowski spacetime $(\mathbb R^{1+n}, g)$ to the unit sphere $\mathbb S^2$, where $g = diag(-1, 1, ... ,1)$ is the standard Minkowski metric.

Denote an arbitrary point in $\mathbb R^{1+n}$ by $(t,x) = (x^\alpha),  0\leq \alpha \leq n$. The Lagrangian density governing the evolution of the fields $\mathbf{n}=\mathbf{n}(x^{\alpha})$ is given by
\be\label{1-L}
\mathcal{L}(\mathbf{n})= -\frac{1}{2}\p_\alpha \mathbf{n} \cdot \p^\alpha \mathbf{n}- \frac{1}{4}\left(\p_\alpha \mathbf{n} \wedge \p_\beta \mathbf{n}\right) \cdot (\p^\alpha \mathbf{n} \wedge \p^\beta \mathbf{n}),
\ee
where the repeated indices means summation over their ranges in the whole paper, $\wedge$ denotes the exterior product.

The Faddeev model can be characterized variationally as critical points of the action integral
 \begin{equation}\label{Action}
 	\mathcal{A}(\mathbf{n})=\int_{\mathbb R\times \mathbb R^n}\mathcal{L(\mathbf{n})}dtdx.
 \end{equation}

 The associated Euler–Lagrange equations of motion are
\be\label{1-EL}
\mathbf{n} \wedge \p_\alpha\p^\alpha \mathbf{n}-\left(\p_\alpha [ \mathbf{n} \cdot ( \p^\alpha \mathbf{n} \wedge \p^\beta \mathbf{n})]\right) \p_\beta \mathbf{n} = 0.
\ee
See Faddeev \cite{Fad1, Fad2, Fad3} and Lin and Yang \cite{L-Y} and references therein.

The Faddeev model holds significantly importance in the area of quantum field theory and presents a multitude of intriguing and challenging mathematical problems. For instance, one can refer to the works cited in \cite{C-M, Esteban, M-1, M-2, Riv, Ryb, Skyrme, Ward}. Researchers such as Lin and Yang \cite{L-Y-1, L-Y-2, L-Y-3, L-Y, L-Y-4} and Faddeev \cite{Fad3} have delved into a wealth of fascinating results in the mathematical exploration of the static Faddeev model. However, when it comes to the evolutionary Faddeev model, which is characterized by its unusual quasi-linear wave equations that feature null and double null structures, along with semilinear and quasilinear terms, there are only a few findings. To the best of our knowledge, using Klainerman's vector field method, Lei, Lin, and Zhou \cite{Lei} were the first to demonstrate global well-posedness results in Sobolev spaces for the evolutionary Faddeev model, under the condition that the initial data are sufficiently small, smooth, and have compact support. Without relying on the assumption of compact support, Dong and Lei \cite{Dong-Lei} revisited the two-dimensional Faddeev model and managed to obtain global existence of smooth solutions through an enhanced $L^2$ estimate. Geba et al. \cite{Geba} showed that for the 2+1-dimensional equivariant Faddeev model, small initial data in critical Besov spaces evolve into global solutions that scatter. Creek \cite{Creek} and Geba and Grillakis \cite{Geba-2} established the global well-posedness of the (1+2)-dimensional equivariant Faddeev model for large data. Most recently, Zha, Liu, and Zhou \cite{Zha} investigated the global nonlinear stability of geodesic solutions of the evolutionary Faddeev model, which represents a class of nontrivial and large solutions.

In this paper, we focus on the evolutionary Faddeev model with a specific class of large initial data, characterized by the short pulse type in terms of energy norm. The concept of this kind of large initial data was pioneered by Christodoulou in his groundbreaking work \cite{Ch D}, where he established the global existence of solutions to the Einstein vacuum equations. This foundational work was subsequently expanded upon by Klainerman and Rodnianski \cite{K-R}, who introduced the concept of relaxed propagation estimates.
Within the domain of quasilinear wave equations, Christodoulou \cite{Ch D}, followed by Speck \cite{Speck} and Speck et al. \cite{Speck2}, delineated a geometric mechanism for shock formation through an in-depth examination of the corresponding Lorentzian metric. Miao and Yu \cite{M-Y} explored a quasilinear wave equation with cubic nonlinearity in the context of large data, which defies the double null condition in Minkowski geometry, and uncovered a similar mechanism for shock formation. In a recent development, when the equation exhibits a double null structure, Wang and Wei \cite{W-W} demonstrated the global existence of smooth solutions to the relativistic membrane equation with short-pulse initial data.
Furthermore, Ding, Xin, and Yin \cite{Ding1, Ding2, Ding3, Ding4, Ding5} investigated the global existence of various quasilinear wave equations in two, three, and four dimensions, providing further insights into the behavior of these equations under large initial data conditions.

\subsection{Main results}

Since $\mathbf{n}$: $\mathbb R^{1+n} \rightarrow \mathbb S^2$ is the map from the Minkowski space to the unit sphere in $\mathbb R^3$, we denote
\be\label{1-n}
\mathbf{n} = (\cos\theta \cos\phi, \cos\theta \sin\phi, \sin\theta ).
\ee
Substituting \eqref{1-n} into \eqref{1-L}, the Lagrangian density becomes \bes
\mathcal{L}(\theta, \phi) = -\frac{1}{2} Q(\theta, \theta) - \frac{1}{2} \cos^2 \theta \, Q(\phi, \phi) -\frac{1}{4} \cos^2 \theta \, Q^{\alpha\beta}(\theta, \phi) Q_{\alpha\beta}(\theta, \phi)
\ees
where the null forms take
\be\label{null Q}
Q(\xi, \chi) = g^{\alpha\beta} \p_\alpha \xi \p_\beta \chi = \p^\alpha \xi \p_\alpha \chi
\ee
and
\bes
Q_{\alpha\beta}(\xi, \chi) = \p_\alpha \xi \p_\beta \chi - \p_\beta \xi \p_\alpha \chi.
\ees

The Euler Lagrange equations are given by
\be\label{Eq}
\left\{
  \begin{aligned}
  &\Box \theta = -\frac{1}{2} \sin (2\theta) \, Q(\phi, \phi) + \frac{1}{4} \sin (2\theta) \,  Q^{\alpha\beta}(\theta, \phi) Q_{\alpha\beta}(\theta, \phi) + \frac{1}{2} \cos^2 \theta Q^{\alpha\beta} (\phi, Q_{\alpha\beta} (\theta, \phi) ), \\
  &\Box \phi = \sin^2 \theta \, \Box \phi + \sin (2\theta) \, Q(\theta, \phi) + \frac{1}{2} \cos^2 \theta Q^{\alpha\beta} (\theta, Q_{\alpha\beta} (\phi, \theta)),
  \end{aligned}
\right.
\ee
where $\Box=-\partial_t^2+\sum_{i=1}^3\partial_{x^{i}}^2$.

Let $r = \sqrt{(x^1)^2 + (x^2)^2+ (x^3)^2}$ and $\omega$ denote the usual radial and angular coordinates on $\mathbb R^3$. $\nabla$ stands for the spatial derivatives $\p_i, i = 1, 2, 3.$ Denote the null coordinates of the Minkowski metric $g$ by
\bes
u = \frac{1}{2} (t - r), \quad \ub = \frac{1}{2} (t + r).
\ees
We use $C_u$ and $\Cb_{\ub}$ to denote the level surface of $u$ and $\ub$, respectively. The intersection of $C_u$ and $\Cb_{\ub}$ is a 2-sphere with radius $\ub-u$, denoted by $S_{u, \ub}$.
In the context of Minkowski space, the double null foliation is characterized by the outgoing null foliation  $\{C_u|u \in \mathbb R\}$ and the incoming null foliation $\{\Cb_{\ub}|\ub \in \mathbb R\}$. The corresponding null vectors are
\bes
L = \p_{\ub} = \p_t + \p_r, \quad \Lb = \p_u = \p_t - \p_r.
\ees

Our initial data is similar to that in \cite{M-P, W-W}. We divide $\Sigma_1:=\{t=1\}$ into three parts (see Figure 1)
\begin{figure}
	\centering
	\includegraphics[width=6.0in]{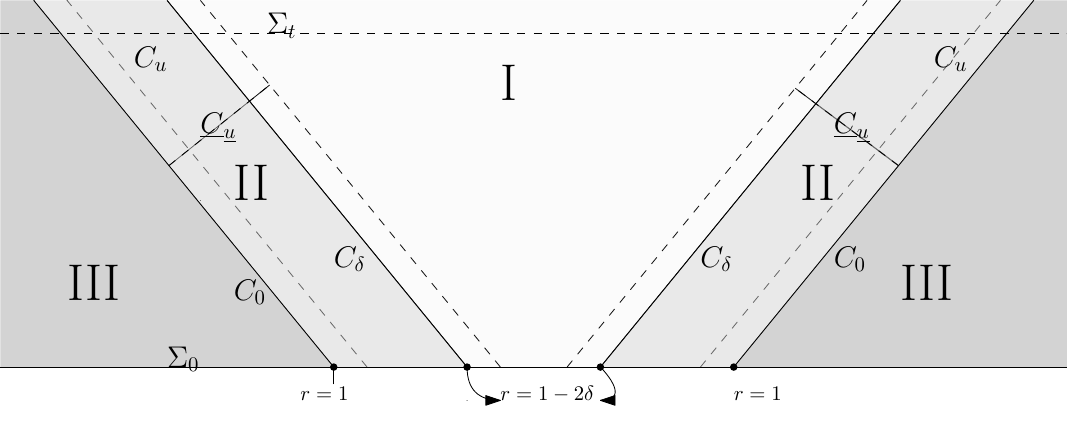}
	\caption{The foliation}
	\label{fig:foliation}
\end{figure}

\bes
\{t=1\} = B_{1-2\delta } \cup  (B_1 - B_{1-2\delta } ) \cup (\mathbb R^3 - B_1),
\ees
where $B_r$ denotes the ball centred at the origin with radius $r$ and $\delta$ a small positive constant which will be determined later.

Prescribed \eqref{Eq} with the following initial data
\bes
(\theta, \phi, \p_t \theta, \p_t \phi)|_{t=1} = (\theta_0, \phi_0, \theta_1, \phi_1),
\ees
which satisfies

On $B_{1-2\delta } \cup (\mathbb R^3 - B_1)$, $\theta_i =\phi_i = 0\ \ (i=0, 1)$.

On $B_1 - B_{1-2\delta }$, $\theta_i (r, \omega), \phi_i (r, \omega)\ \ ( i=0, 1)$ are smooth functions supported in $ r \in ( 1-2\delta, 1)$ and satisfy the constraints
\be\label{Con1}
\Vert \slashed{\nabla}^l (\delta \p_r)^k(\theta_1 +\p_r \theta_0)\Vert _{L^{\infty}(B_1 - B_{1-2\delta})} + \Vert \slashed{\nabla}^l (\delta \p_r)^k(\phi_1 +\p_r \phi_0)\Vert _{L^{\infty}(B_1 - B_{1-2\delta})} \leq C_{k,l}\delta,
\ee
\be\label{Con2}
\begin{split}
\Vert \slashed{\nabla}^l & \, (\delta \p_r)^k \theta_0 \Vert _{L^{\infty}(B_1 - B_{1-2\delta})} + \delta \Vert \slashed{\nabla}^l (\delta \p_r)^{k-1} \theta_1 \Vert _{L^{\infty}(B_1 - B_{1-2\delta})} \\
+ & \Vert \slashed{\nabla}^l (\delta \p_r)^k \phi_0 \Vert _{L^{\infty}(B_1 - B_{1-2\delta})} + \delta \Vert \slashed{\nabla}^l (\delta \p_r)^{k-1} \phi_1 \Vert _{L^{\infty}(B_1 - B_{1-2\delta})} \leq \hat{C}_{k,l}\delta,
\end{split}
\ee
\be\label{Con3}
\Vert \slashed{\nabla}^l (\p_t + \p_r)^k (\delta \p_r)^m \theta \Vert _{L^{\infty}(B_1 - B_{1-2\delta})} + \Vert \slashed{\nabla}^l (\p_t + \p_r)^k (\delta \p_r)^m \phi \Vert _{L^{\infty}(B_1 - B_{1-2\delta})} \leq C_{k,l,m}\delta,
\ee
where $\slashed{\nabla}$ denotes the covariant derivative (associated to the induced metric from the flat Lorentzian metric) on the sphere with constant $t$ and $r$, and $m, k, l$ are non-negative integers.

In the following, we always use $f \lesssim g$ to denote $f \leq Cg$ for some universal and positive constant $C$ and $f \sim g$ to denote $f \lesssim g$ and $g \lesssim f$.

\begin{remark}
During the propagation of our initial data, the hyperbolicity of \eqref{Eq} can be maintained since we have $| \slashed{\nabla} \theta | \lesssim \delta^{\frac{3}{4}} $, $| L \theta | \lesssim \delta$, $| \Lb \theta | \lesssim 1$ and $| \slashed{\nabla} \phi | \lesssim \delta^{\frac{3}{4}} $, $| L \phi | \lesssim \delta$, $| \Lb \phi | \lesssim 1$ according to Lemma \ref{Sobolev}.
\end{remark}

We now formulate the main result of our paper.
\begin{theorem}\label{main-theorem}
There exists a positive constant $\delta_0$, such that the Cauchy problem \eqref{Eq} with the short pulse initial data \eqref{Con1}-\eqref{Con3} admits a unique and globally smooth solution in $[1, +\infty ) \times \mathbb R^3$ providing that $0 < \delta < \delta_0$.
\end{theorem}

\subsection{Strategy of the Proof}

Due to the finite speed of propagation of wave equations, \eqref{Eq} has a trivial solution in the region III. Then our proof focus on the global existence of the large solution in Region II and small solution in Region I.




\subsubsection{Multipliers adapted to the Faddeev geometry}

In the large data region II, our principal strategy for proving global existence mainly depends on energy estimates. The selection of suitable multipliers is important, as it significantly simplifies our calculations. Specifically, in the context of the Faddeev model—a system of quasilinear wave equations—the multipliers $L$ and
$\Lb$ must be chosen with great care. Their role is to generate energies of a definite sign, which is essential for effectively managing the intrinsic complexities of the system’s dynamics.

Our choice of multipliers is inspired by \cite{W-W}. We rewrite the first equation of \eqref{Eq} as
\begin{align}\label{multi}
& \, \Box \theta  - \cos^2 \theta \big(\p^\alpha \phi \, \p^\beta \phi - g^{\alpha \beta} Q(\phi, \phi)\big)\p_\alpha \p_\beta \theta  -\cos^2 \theta \big( g^{\alpha \beta} Q(\theta, \phi)- \p^\alpha \theta \, \p^\beta \phi \big)  \p_\alpha \p_\beta  \phi \nnb \\
= & \, -\frac{1}{2} \sin (2\theta) \, Q(\phi, \phi) + \frac{1}{4} \sin (2\theta) \,  Q^{\alpha\beta}(\theta, \phi) Q_{\alpha\beta}(\theta, \phi) \nnb \\
&+ \, \cos^2 \theta \Big( (\,\p^\alpha \theta \, \p^\beta \p_\beta \theta \, - \p^\beta \theta \, \p^\alpha \p_\beta \theta ) \p_\alpha   \phi + (\,\p^\beta \theta \, \p^\alpha \p_\beta \phi \, - \p^\alpha \theta \, \p^\beta \p_\beta \phi ) \p_\alpha   \theta \Big).
\end{align}
 In order to find suitable multipliers, we can rewrite \eqref{multi} as a geometric wave equation
\bes
\Box_{g(\p \phi)} \theta=N(\theta, \phi)
\ees
where
\be
g^{\alpha\beta}(\p \phi) = g^{\alpha\beta} - \cos^2 \theta \big(\p^\alpha \phi \, \p^\beta \phi - g^{\alpha \beta} Q(\phi, \phi)\big).
\ee
In the subsequent discussion, we refer to this as the geometric aspect of the Faddeev model. The natural selection for the multipliers is $(-2D\ub, -2Du)$, which are null vectors and are defined by
\bes
-2D\ub = -2 g^{\alpha\beta}(\p \phi) \p_\alpha \ub \, \p_\beta, \quad -2Du = -2 g^{\alpha\beta}(\p \phi) \p_\alpha u \, \p_\beta.
\ees

However, the aforementioned multipliers introduce a multitude of additional higher-order terms, complicating the analysis. To streamline the computations, we opt to use the leading parts of the expressions $-2D\ub$ and $-2Du $ , which are $ L + (L \phi)^2 \Lb$ and $ \Lb + (\Lb \phi)^2 L$ as our multipliers. We can verify that these modified multipliers are causal with respect to $g(\p \phi)$, which is crucial for ensuring the positivity of the energies we aim to derive. This approach maintains the causal structure while simplifying the energy estimates.

By employing these leading parts as multipliers, we can perform a similar analysis on the equations that $\phi$ satisfies, leveraging the causality and positivity properties to our advantage. This methodological choice not only simplifies the algebraic complexity but also provides a clearer pathway to establishing the global well-posedness of the solutions in the context of the Faddeev model.

Unlike the scalar wave equation previously discussed, the Faddeev model consists of a system of coupled wave equations. The cross terms, such as $-\cos^2 \theta \big( g^{\alpha \beta} Q(\theta, \phi)- \p^\alpha \theta \, \p^\beta \phi \big)  \p_\alpha \p_\beta  \phi$ also contain the highest order derivatives, which pose anther challenge in the analysis. To achieve a certain symmetry that can help cancel out these highest order derivatives of order
$k+2$, we will employ the following two modified multiplier vector fields
\bes
\tilde{L} = L + ((L\theta)^2 + (L\phi)^2)\Lb, \quad \tilde{\Lb} = \Lb + ((\Lb\theta)^2 + (\Lb\phi)^2)L.
\ees

The selection of these multiplier vector fields is strategic, aiming to exploit the symmetries and structures within the Faddeev model to our advantage. By carefully constructing these multipliers, we can effectively manage the complexity introduced by the cross terms and the highest order derivatives, which is essential for proving global existence of smooth solutions in the short-pulse region II.

\subsubsection{A globally small solution in region I}

Based on the energy estimates in region II, we first demonstrate that the incoming energy is small. Leveraging the good derivatives of the solution, which exhibit smallness and better decay, we have $$ \big|\po (\p^l \Gamma^{k-l} (\theta,\phi))  \big|_{L^{\infty}(C_\delta)} \lesssim \delta^{\frac{3}{4}} |\ub|^{-2},\quad 0 \leq l \leq k,$$ and thanks to the zero initial data assumption on $S_{\delta,1-\delta}$,
the property of the bad derivatives of the solution can be improved to $$\big|\Lb (\p^l \Gamma^{k-l} (\theta,\phi))  \big|_{L^{\infty}(C_\delta)} \lesssim \delta^{\frac{3}{4}} |\ub|^{-1}, \quad 0 \leq l \leq k$$ by integrating the whole system in terms of the null coordinates along the direction $L$ on $C_{\delta}$.

Second, building on the favorable smallness and decay properties of the solution on $C_{\delta}$, we address the initial boundary value (Goursat) problem in region I. Given that the problem is inherently small, we can partition the space-time domain by hypersurfaces $\{t=const.\}$. By selecting $\partial_t$
as the multiplier and utilizing the Klainerman-Sobolev inequality, the null structure of the Faddeev model, and the smallness of the solution on the boundary, we are able to derive global energy estimates.

\subsection{Arrangement of the paper} Our paper is structured as follows: Section \ref{sec:2} presents the essential notations and lemmas required for our analysis. Section \ref{sec3} focuses on the proof of the global existence of large solutions in Region II. In Section \ref{sec4}, we elucidate the geometric interpretation of short-pulse data, which posits that the significant energy propagates outward while the energy along the incoming direction remains small, implying that the solution on $C_\delta$ is of small magnitude. Section \ref{sec5} is dedicated to demonstrating the global existence of the smooth solution emanates from the solution on $C_\delta$.

\section{Preliminaries}\label{sec:2}

\subsection{Notations}

Throughout this paper, we use Greek indices $\alpha, \beta, \ldots$ to denote values ranging from 0 to 3, and Latin letters $i, j, k, \ldots$ to denote values ranging from 1 to 3. Additionally, we employ Einstein's summation convention, which means that repeated upper and lower indices are summed over their ranges.
Introduce the Lorentz vector fields $\mathcal{Z} = \{\p, \Gamma \} = \{ \p_\mu, \Omega_{0i}, \Omega_{ij}, S\}$, where $\Omega_{0i} = t\p_i + x_i\p_t$ is the Lorentz boost, $ \Omega_{ij} = x_i\p_j - x_j\p_i$ is the rotation and $S = t\p_t + r \p_r$ is the scaling. If we denote $|\Omega \psi| = \sum |\Omega_{ij} \psi|$,  we have
\be\label{omega3}
|\Omega \psi| \sim r \,| \slashed{\nabla} \psi |.
\ee

Recall the two null vectors $L = \p_t + \p_r $ and $\Lb = \p_t - \p_r$. We will introduce a global frame induced by the projection of $\p_i (i=1,2,3)$ onto the tangent space of sphere $\mathbb S^2$, which is denoted by
\be\label{po}
\po_i = \p_i - \omega_i \omega^j \p_j =  \p_i - \omega_i \p_r, \quad \omega_i = \frac{x_i}{r}.
\ee
Noting that
\be\label{po-Omega}
\po_i = \frac{\omega^j}{r} \Omega_{ji}.
\ee
We define the ``good derivatives'' $\{\po \} = \{L, \po_1, \po_2,\po_3\}$, which spans the tangent space of the forward light cone $\{ t+r = constant\}$.

\subsection{Null forms}
In this subsection, we recall the definition of the null forms and their properties
\begin{definition}
A real-valued quadratic form $\mathcal{Q}$ defined on $\mathbb R^{1+n}$ is termed a null form if for any null vector $\xi \in \mathbb R^{1+n}$, it satisfies  $$\mathcal{Q}(\xi, \xi) = 0.$$
\end{definition}

 Denote the Lie bracket as
\bes
[Z, \mathcal{Q}] (\xi, \chi) = Z \mathcal{Q}(\xi, \chi) - \mathcal{Q} (Z\xi, \chi) - \mathcal{Q} (\xi, Z\chi),
\ees
where $Z \in \mathcal{Z}$ is an arbitrary Lorentz field.

It is easy to check that
\be\label{null1}
[\p, \mathcal{Q}] = 0, \quad [S, \mathcal{Q}] = -2\mathcal{Q},
\ee
and
\be\label{null2}
\begin{aligned}
& [\Omega_{0k}, Q] = 0, \quad [\Omega_{0k}, Q_{0i}] = Q_{ik}, \quad [\Omega_{0k}, Q_{ij}] = -\delta_{ik}Q_{0j} + \delta_{jk}Q_{0i},\\
& [\Omega_{kl}, Q] = 0, \quad [\Omega_{kl}, Q_{0i}] = -\delta_{ik}Q_{0l} + \delta_{il}Q_{0k}, \\
& [\Omega_{kl}, Q_{ij}] = -\delta_{ik}Q_{lj} + \delta_{il}Q_{kj} + \delta_{jk}Q_{li} - \delta_{jl}Q_{ki},
\end{aligned}
\ee
where $\delta_{ij}$ is the kronecker delta function.

Based on above relations, we can get
\be\label{null}
Z \mathcal{Q}(\xi, \chi) = \mathcal{Q} (Z\xi, \chi) + \mathcal{Q} (\xi, Z\chi) + \tilde{\mathcal{Q}}(\xi, \chi),
\ee
where $\tilde{\mathcal{Q}}$ is a linear combination of null forms.

The double null form in our equation takes $Q^{\alpha\beta} (\psi, Q_{\alpha\beta} (\xi, \chi))$. In fact, it contains cubic nonlinearities having better decay properties. The following lemma shows that the commutation of any Lorentz field with the double null form still preserves the double null structure.

\begin{lemma}
For any Lorentz field $Z \in \mathcal{Z}$,
\begin{align}
Z Q^{\alpha\beta} (\psi, Q_{\alpha\beta} (\xi, \chi)) = & \, Q^{\alpha\beta} (Z\psi, Q_{\alpha\beta} (\xi, \chi)) + Q^{\alpha\beta} (\psi, Q_{\alpha\beta} (Z\xi, \chi)) + Q^{\alpha\beta} (\psi, Q_{\alpha\beta} (\xi, Z\chi)) \nnb\\
& + c\, Q^{\alpha\beta} (\psi, Q_{\alpha\beta} (\xi, \chi)) \label{double},
\end{align}
where $c=-4$ when $Z=S$, and $c=0$ for other Lorentz fields.
\end{lemma}

\begin{proof}
According to \eqref{null}, for any $Z \in \mathcal{Z}$, we have
\begin{align*}
Z Q^{\alpha\beta} (\psi, Q_{\alpha\beta} (\xi, \chi))
= & \, Q^{\alpha\beta} (Z\psi, Q_{\alpha\beta} (\xi, \chi)) + Q^{\alpha\beta} (\psi, Z Q_{\alpha\beta} (\xi, \chi)) + [Z, Q^{\alpha\beta}] (\psi, Q_{\alpha\beta} (\xi, \chi))\\
= & \, Q^{\alpha\beta} (Z\psi, Q_{\alpha\beta} (\xi, \chi)) + Q^{\alpha\beta} (\psi, Q_{\alpha\beta} (Z \xi, \chi)) + Q^{\alpha\beta} (\psi, Q_{\alpha\beta} (\xi, Z \chi))\\
& + Q^{\alpha\beta} (\psi, [Z, Q_{\alpha\beta}] (\xi, \chi)) + [Z, Q^{\alpha\beta}](\psi, Q_{\alpha\beta} (\xi, \chi)).
\end{align*}

If $Z = \p$ or $S$, \eqref{null} follows directly from \eqref{null1}.

If $Z = \Omega_{0k}$, we get
\begin{align*}
& Q^{\alpha\beta} (\psi, [\Omega_{0k}, Q_{\alpha\beta}] (\xi, \chi)) + [\Omega_{0k}, Q^{\alpha\beta}](\psi, Q_{\alpha\beta} (\xi, \chi))\\
= &\, 2 Q^{0i}(\psi, [\Omega_{0k}, Q_{0i}](\xi, \chi)) + 2 [\Omega_{0k}, Q^{0i}](\psi, Q_{0i} (\xi, \chi)) + Q^{ij}(\psi, [\Omega_{0k}, Q_{ij}](\xi, \chi))\\
&\, + [\Omega_{0k}, Q^{ij}] (\psi, Q_{ij}(\xi, \chi))\\
= &\, 2 Q^{0i} (\psi, Q_{ik} (\xi, \chi)) - 2 Q^{ik}(\psi, Q_{0i} (\xi, \chi)) + Q^{ij} (\psi, (-\delta_{ik}Q_{0j} + \delta_{jk}Q_{0i}) (\xi, \chi)) \\
&\, + (\delta^i_{k}Q^{0j} - \delta^j_{k}Q^{0i})(\psi, Q_{ij} (\xi, \chi))\\
= & \, 2 Q^{0i} (\psi, Q_{ik} (\xi, \chi)) - 2 Q^{ik}(\psi, Q_{0i} (\xi, \chi)) - Q^{kj} (\psi, Q_{0j}(\xi, \chi)) + Q^{ik} (\psi, Q_{0i}(\xi, \chi)) \\
&\, + Q^{0j} (\psi, Q_{kj} (\xi, \chi)) - Q^{0i} (\psi, Q_{ik} (\xi, \chi)) \\
= & \, 2 Q^{0i} (\psi, Q_{ik} (\xi, \chi)) - 2 Q^{ik}(\psi, Q_{0i} (\xi, \chi)) + Q^{jk} (\psi, Q_{0j}(\xi, \chi)) + Q^{ik} (\psi, Q_{0i}(\xi, \chi)) \\
&\, - Q^{0j} (\psi, Q_{jk} (\xi, \chi)) - Q^{0i} (\psi, Q_{ik} (\xi, \chi))  \\ = & \;0.
\end{align*}
Here, we used the fact that $Q_{\alpha\beta}$ is anti-symmetric and the double null form $Q^{\alpha\beta} (\psi, Q_{\alpha\beta} (\xi, \chi))$ is symmetric with respect to $\alpha$, $\beta$.

If $Z = \Omega_{kl}$, we can obtain
\begin{align*}
&Q^{\alpha\beta} (\psi, [\Omega_{kl}, Q_{\alpha\beta}] (\xi, \chi)) + [\Omega_{kl}, Q^{\alpha\beta}](\psi, Q_{\alpha\beta} (\xi, \chi))\\
= & \, 2 Q^{0i} (\psi, [\Omega_{kl}, Q_{0i}] (\xi, \chi)) + 2 [\Omega_{kl}, Q^{0i}](\psi, Q_{0i} (\xi, \chi)) + Q^{ij} (\psi, [\Omega_{kl}, Q_{ij}] (\xi, \chi))\\
&\, + [\Omega_{kl}, Q^{ij}](\psi, Q_{ij} (\xi, \chi)) \\
= &\, 2 Q^{0i} (\psi, (-\delta_{ik}Q_{0l} + \delta_{il}Q_{0k}) (\xi, \chi)) + 2 (-\delta^i_{k}Q^{0l} + \delta^i_{l}Q^{0k})(\psi, Q_{0i} (\xi, \chi))\\
&\,  + Q^{ij} (\psi, (-\delta_{ik}Q_{lj} + \delta_{il}Q_{kj} + \delta_{jk}Q_{li} - \delta_{jl}Q_{ki}) (\xi, \chi)) \\
&\, + (-\delta^i_{k}Q^{lj} + \delta^i_{l}Q^{kj} + \delta^j_{k}Q^{li} - \delta^j_{l}Q^{ki})(\psi, Q_{ij} (\xi, \chi))
\\
= & -2 Q^{0k} (\psi, Q_{0l} (\xi, \chi)) + 2 Q^{0l} (\psi, Q_{0k} (\xi, \chi)) - 2 Q^{0l} (\psi, Q_{0k} (\xi, \chi)) + 2 Q^{0k} (\psi, Q_{0l} (\xi, \chi)) \\
& - Q^{kj} (\psi, Q_{lj}(\xi, \chi)) + Q^{lj} (\psi, Q_{kj}(\xi, \chi)) + Q^{ik} (\psi, Q_{li}(\xi, \chi)) - Q^{il} (\psi, Q_{ki}(\xi, \chi))\\
& -Q^{lj} (\psi, Q_{kj}(\xi, \chi)) + Q^{kj} (\psi, Q_{lj}(\xi, \chi)) + Q^{li} (\psi, Q_{ik}(\xi, \chi)) - Q^{ki} (\psi, Q_{il}(\xi, \chi))
\\ = & \;0.
\end{align*}
The lemma follows by combining all the above calculations.
\end{proof}

Similar calculation leads to the following lemma.

\begin{lemma}
For any Lorentz field $Z \in \mathcal{Z}$,
\begin{align}
 & Z \big( Q^{\alpha\beta} (\psi, \zeta) Q_{\alpha\beta} (\xi, \chi) \big) \nnb \\
= & \, Q^{\alpha\beta} (Z\psi, \zeta) Q_{\alpha\beta} (\xi, \chi) + Q^{\alpha\beta} (\psi, Z\zeta) Q_{\alpha\beta} (\xi, \chi) + Q^{\alpha\beta} (\psi, \zeta) Q_{\alpha\beta} (Z\xi, \chi) \nnb \\
& + Q^{\alpha\beta} (\psi, \zeta) Q_{\alpha\beta} (\xi, Z\chi) + c\, Q^{\alpha\beta} (\psi, \zeta) Q_{\alpha\beta} (\xi, \chi),\label{double pro}
\end{align}
where $c=-4$ when $Z=S$, and $c=0$ for other Lorentz fields.
\end{lemma}

The following lemma describes the good properties of null forms.
\begin{lemma}\label{null form}
For the null forms, we have
\be
| Q (\xi, \chi) | \lesssim | L \xi | \, |\Lb \chi | + | \Lb \xi | \, |L \chi | +  | \po \xi | \, |\po \chi |,
\ee
and
\be\label{Q_alpha}
| Q_{\alpha \beta} (\xi, \chi) | \lesssim | \po \xi | \, |\p \chi | + | \p \xi | \, |\po \chi |.
\ee
\end{lemma}

\begin{proof}
For the null form $Q$, it follows directly from \eqref{null Q} and \eqref{po}.

For null form $Q_{0i}$, we have
\begin{align*}
Q_{0i}(\xi, \chi) = & \,\p_t \xi \, \p_i \chi - \p_i \xi \,\p_t \chi   \\
= & \, \frac{1}{2}(L+\Lb) \xi \cdot (\po_i +\omega_i \p_r) \chi - (\po_i +\omega_i \p_r) \xi \cdot\frac{1}{2}(L+\Lb) \chi  \\
= & \, \frac{1}{2}(L+\Lb) \xi \, \po_i \chi - \frac{1}{2}\po_i \xi \, (L+\Lb) \chi + \frac{1}{4} \omega_i \Big( (L+\Lb) \xi \, (L- \Lb)\chi - (L- \Lb) \xi \, (L+\Lb) \chi \Big)  \nnb\\
= & \,  \frac{1}{2} \Big( Q_{L\io}(\xi, \chi) + Q_{\Lb \io}(\xi, \chi) \Big) + \frac{1}{2}  \omega_i Q_{\Lb L}(\xi, \chi),
\end{align*}
here $Q_{L\io}(\xi, \chi) = L \xi \, \po_i \chi - \po_i \xi \, L  \chi$ and $Q_{\Lb \io}(\xi, \chi) = \Lb \xi \, \po_i \chi - \po_i \xi \, \Lb  \chi$. For simplicity, we rewrite the above equality as
\be\label{Q_0i}
 Q_{0i} =  \frac{1}{2} \Big( Q_{L\io} + Q_{\Lb \io} + \omega_i Q_{\Lb L}\Big).
\ee
Similar calculation shows that
\be\label{Q_ij}
 Q_{ij} =   Q_{\io \, \jo} + \frac{1}{2} \omega_i  ( Q_{L \jo} - Q_{\Lb \jo} ) + \frac{1}{2} \omega_j  ( Q_{\Lb \io} - Q_{L \io}) .
\ee
The inequality \eqref{Q_alpha} directly follows from \eqref{Q_0i} and \eqref{Q_ij}.

\end{proof}

In the following, we transform the double null form in the frame of $\{L, \Lb, \po^{\, i}\}$. Denote $\p_a$, $\p_b$ to be any of the vectorfields $\{L, \Lb, \po^{\, i}\}$.

\begin{lemma} In terms of the frame $\{L, \Lb, \po^{\, i}\}$, we can get
\begin{align*}
Q^{\alpha\beta} (\psi, Q_{\alpha\beta} (\xi, \chi)) = & \, Q^{ab} (\psi, Q_{ab} (\xi, \chi)) - \frac{1}{r} ( L \psi + \Lb \psi ) \, Q_{\Lb L} (\xi, \chi)- \frac{1}{r} \po^{\,i}  \psi \big( Q_{L \io}(\xi, \chi) - Q_{\Lb \io}(\xi, \chi) \big).
\end{align*}
\end{lemma}

\begin{remark}
For simplicity, we denote $Q_{ab} (\xi, \chi)$ by $Q_{ab}$ and rewrite the above equality as
\bes
Q^{\alpha\beta} (\psi, Q_{\alpha\beta}) = Q^{ab} (\psi, Q_{ab} ) - \frac{1}{r} ( L \psi + \Lb \psi ) \, Q_{\Lb L} - \frac{1}{r} \po^{\,i}  \psi ( Q_{L \io} - Q_{\Lb \io} ).
\ees
\end{remark}

\begin{proof}
We raise the indices in $\po_i$ by $ \deo^{ij} = \delta^{ij} - \omega^i \omega^j$, then
\bes
\po^{\,i} = \deo^{ij} \po_j = (\delta^{ij} - \omega^i \omega^j) \po_j = \po_i,
\ees
noting that $\omega^j \po_j = \omega^j (\p_j - \omega_j \p_r ) = \p_r - \p_r = 0 $.

According to \eqref{Q_0i} , we have
\bes
 Q^{0i} (\psi, Q_{0i} )
=  Q^{\Lb \io} (\psi, Q_{0i} ) + Q^{L \io} (\psi, Q_{0i} ) + 2 \omega^i \,  Q^{\Lb L} (\psi, Q_{0i} ).
\ees
For the first term in the right hand side,
\begin{align*}
Q^{\Lb \io} (\psi, Q_{0i} )  = & \, \frac{1}{2} Q^{\Lb \io} (\psi, Q_{L\io} + Q_{\Lb \io} + \omega_i Q_{\Lb L} ) \\
= & \, \frac{1}{2} Q^{\Lb \io} (\psi, Q_{L\io} + Q_{\Lb \io} ) +  \frac{1}{2} \omega_i  Q^{\Lb \io} (\psi, Q_{\Lb L} ) + \frac{1}{2} \p^{\Lb} \psi ( \po^{\,i} \omega_i)  Q_{\Lb L} \\
= & \, \frac{1}{2} Q^{\Lb \io} (\psi, Q_{L\io} + Q_{\Lb \io} ) - \frac{1}{4} L \psi  \, \frac{1}{r} ( \delta^i_i - \omega_i \omega^i)  Q_{\Lb L}  \\
= & \, \frac{1}{2} Q^{\Lb \io} (\psi, Q_{L\io} + Q_{\Lb \io} ) - \frac{1}{2r} L \psi  \,  Q_{\Lb L}.
\end{align*}
Similarly, the second term
\bes
 Q^{L \io} (\psi, Q_{0i} ) =  \frac{1}{2} Q^{L \io} (\psi, Q_{L\io} + Q_{\Lb \io} )  - \frac{1}{2r} \Lb \psi  \,  Q_{\Lb L}.
\ees
As the third term,
\begin{align*}
2 \omega^i \,  Q^{\Lb L} (\psi, Q_{0i} ) = & \,\omega^i \,  Q^{\Lb L} (\psi, Q_{L\io} + Q_{\Lb \io} + \omega_i Q_{\Lb L} ) \\
= & \,  Q^{\Lb L} (\psi, \omega^i \, (Q_{L\io} + Q_{\Lb \io})  ) + \omega^i \omega_i \, Q^{\Lb L} (\psi,  Q_{\Lb L} ) \\
= & \,  Q^{\Lb L} (\psi,  Q_{\Lb L} ) ,
\end{align*}
where we use the fact $L \omega_i = \Lb  \omega_i = 0$.
Thus
\begin{align*}
 & \, Q^{0i} (\psi, Q_{0i} ) \\ = & \,\frac{1}{2} Q^{\Lb \io} (\psi, Q_{L\io} + Q_{\Lb \io} ) + \frac{1}{2} Q^{L \io} (\psi, Q_{L\io} + Q_{\Lb \io} ) +  Q^{\Lb L} (\psi,  Q_{\Lb L} ) - \frac{1}{2r} L \psi  \,  Q_{\Lb L} - \frac{1}{2r} \Lb \psi  \,  Q_{\Lb L}.
\end{align*}

A similar calculation yields
\bes
Q^{ij} (\psi, Q_{ij} )  = Q^{\io \, \jo} (\psi, Q_{\io \, \jo} ) + Q^{L \io} (\psi, Q_{L\io} - Q_{\Lb \io} ) -  Q^{\Lb \io} (\psi, Q_{L\io} - Q_{\Lb \io} ) - \frac{1}{r} \po^{\, i} \psi  ( Q_{L\io} - Q_{\Lb \io} ).
\ees

Therefore,
\begin{align*}
Q^{\alpha\beta} (\psi, Q_{\alpha\beta} ) = & \, 2 Q^{0i} (\psi, Q_{0i} ) + Q^{ij} (\psi, Q_{ij} ) \\
 = & \, Q^{ab} (\psi, Q_{ab} ) - \frac{1}{r} ( L \psi + \Lb \psi ) \, Q_{\Lb L}- \frac{1}{r} \po^{\,i}  \psi ( Q_{L \io} - Q_{\Lb \io} ).
\end{align*}
\end{proof}

 The following lemma reveals the analogous properties of double null forms as Lemma \ref{null form}.
\begin{lemma}\label{double null}
For the double null forms, we have
\begin{align}
& \, |  Q^{\alpha \beta} (\psi, Q_{\alpha \beta} (\xi, \chi)) | \nnb \\
 \lesssim & \, | L \psi | \Big( |\Lb^2 \xi | \, | L \chi | + | \Lb \xi | \, |\Lb L \chi | + |\Lb L \xi | \, | \Lb  \chi | + | L \xi | \, | \Lb^2 \chi | \Big) \nnb \\
 & \,+ | \Lb \psi | \Big( |L^2 \xi | \, | \Lb \chi | + | L \xi | \, |L \Lb \chi | + |L \Lb \xi | \, |L  \chi | + | \Lb \xi | \, | L^2 \chi | \Big) + \frac{1}{r} | \Lb \psi | \Big( |\Lb \xi | \, | L \chi | + |L  \xi | \, | \Lb \chi | \Big)
 \nnb \\
 & \, + | \po \psi | \Big(  |\p \po \xi | \, | \po \chi | +  | \po \xi | \, |\p \po \chi | +  |\po^2 \xi | \, | \p \chi | + |\p \xi | \, | \po^2 \chi | \Big)  + | \Lb \psi |  \Big( |\po^2 \xi | \, | \po \chi | + |\po \xi | \, | \po^2  \chi | \Big) \nnb \\
 & \, + \frac{1}{r} | \po \psi | \Big( |\p  \xi | \, | \po \chi | + |\po  \xi | \, | \p \chi | \Big) .\label{double Q}
\end{align}
\end{lemma}
We can see that each term in the first and second lines of \eqref{double Q} includes two $\Lb$, while that in the third and fourth lines includes only one $\Lb$. The good derivatives or the extra $r^{-1}$ will provide enough decay in three dimensional case.

\subsection{Energy identity}
We briefly recall the vector field method here. One can refer to Alinhac \cite{Alinhac} or Wang and Yu \cite{W-Y} for details.

Let $\psi$ be a solution for the following non-homogenous wave equation $$\Box \psi = \Psi,$$ $X, Y, Z$ be the vector fields defined before. The energy momentum tensor associated to $\psi$ is defined by
\bes
T(X, Y) [\psi] = X \psi \cdot Y \psi -\frac{1}{2} \left< X, Y \right> \, |\nabla  \psi|^2.
\ees
In local coordinates
\be\label{momentum}
T_{\alpha \beta}  [\psi] = \p_\alpha \psi \, \p_\beta \psi -\frac{1}{2} \, g_{\alpha \beta} \, |\nabla  \psi|^2.
\ee
The deformation tensor of a given vector field $X$ is defined by
\bes
^{(X)}\pi(Y, Z) = \left< D_Y  X, Z \right> + \left< D_Z  X, Y \right>.
\ees
In local coordinates
\bes
^{(X)}\pi_{\alpha \beta} = D_\alpha  X_\beta + D_\beta  X_\alpha.
\ees
For $f \in C^1$ being an arbitrary function, we have
\be\label{deformation2}
^{(fX)}\pi_{\alpha \beta} = f \, ^{(X)}\pi_{\alpha \beta} + X_\alpha \cdot \p_\beta f + X_\beta \cdot \p_\alpha f.
\ee
The energy currents are defined as follows
\bes
P_\alpha = T_{\alpha \beta} [\psi] \cdot X^\beta,
\ees
and
\be\label{KX}
K^X  [\psi] = \frac{1}{2} \, T^{\alpha \beta}  [\psi] \cdot {^{(X)}\pi_{\alpha \beta} },
\ee
then we have a key divergence form
\bes
div P  = D_\alpha P^\alpha = \Box \psi \cdot X \psi+ K^X  [\psi].
\ees
Integrating $\Box \psi \cdot X \psi$ in some domain $D$, and using the
divergence theorem, we obtain
\be\label{div}
\int_{\p D} \left<P, N \right> \, \di v = \int_{\p D} T(X, N) \,  \di v = \iint_{D} \big( \Box \psi \cdot X \psi+ K^X [\psi] \big) \, \di V,
\ee
where $N$ is the unit outward normal vector to $\p D$ and $\left<\cdot, \cdot \right>$ is the inner product.

In the null frame $(e_1, e_2, e_3 = L, e_4 = \Lb)$, the deformation tensors and currents are computed as
\be\label{deformation}
^{(L)}\pi = \frac{2}{r} \slashed{g}, \quad ^{(\Lb)}\pi = -\frac{2}{r} \slashed{g}, \quad K^L [\psi] = \frac{1}{r} L \psi \Lb \psi, \quad K^{\Lb} [\psi] = -\frac{1}{r} L \psi \Lb \psi,
\ee
here $\slashed{g}$ is the restriction of the Minkowski metric $g$ on $S_{u, \ub}$.

\section{Globally smooth solution in region II}\label{sec3}

In this section, we will prove the global existence of smooth solution in the short pulse region II.

For $0\leq l \leq k$, denote
\bes
\theta_{k,l} = \delta^l \p^l \Gamma^{k-l} \theta, \quad \phi_{k,l} = \delta^l \p^l \Gamma^{k-l} \phi,
\ees
\bes
| \p \theta_k | = \sum_{l \leq k} | \p \theta_{k,l} | , \quad | \p  \phi_{k} | = \sum_{l \leq k} | \p \phi_{k,l} |,
\ees
and
\bes
\Vert \p \theta_k \Vert^2_{L^2(\Omega)} = \sum_{l \leq k} \Vert \p \theta_{k,l} \Vert^2_{L^2(\Omega)} , \quad \Vert \p  \phi_{k} \Vert^2_{L^2(\Omega)} = \sum_{l \leq k} \Vert \p \phi_{k,l} \Vert^2_{L^2(\Omega)}.
\ees

We first define the energy norms defined on $C_u$ and $\Cb_{\ub}$. Denote $D_{u, \ub}$ to be the closed domain bounded by $C_u$, $\Cb_{\ub}$, $C_0$ and the initial hypersurface $\Sigma_1$. Let $C_u^{\ub'}$ be the outgoing null boundary of $D_{u, \ub'}$ on $C_u$ and   $\Cb_{\ub}^{u'}$ be the incoming null boundary of $D_{u', \ub}$ on $\Cb_{\ub}$. Given any function $f$, define
\bes
\begin{aligned}
\Vert f \Vert_{L^2(C_u^{\ub})} = \Big( \int_{1-u}^{\ub} \int_{S_{u, \ub'}} |f|^2 \, r^2 \, \di \ub' \, \di \sigma_{S^2} \Big)^{\frac{1}{2}}, \\
\Vert f \Vert_{L^2(\Cb_{\ub}^u)} = \Big( \int_{max\{1-\ub,0\}}^{u} \int_{S_{u', \ub}} |f|^2 \, r^2 \, \di u' \, \di \sigma_{S^2} \Big)^{\frac{1}{2}}.
\end{aligned}
\ees
Define the $k$-th order homogeneous energies as follows
\be\label{E}
\delta^2 E^2_k(u, \ub) = \Vert \slashed{\nabla} \theta_{k} \Vert^2_{L^2(\Cb_{\ub}^u)} + \Vert \slashed{\nabla} \phi_{k} \Vert^2_{L^2(\Cb_{\ub}^u)} + \Vert L \theta_{k} \Vert^2_{L^2(C_u^{\ub})} + \Vert L \phi_{k} \Vert^2_{L^2(C_u^{\ub})}
\ee
and
\be\label{Eb}
\delta \Eb^2_k(u, \ub) = \Vert \slashed{\nabla} \theta_{k} \Vert^2_{L^2(C_u^{\ub})} +  \Vert \slashed{\nabla} \phi_{k} \Vert^2_{L^2(C_u^{\ub})} + \Vert \Lb \theta_{k} \Vert^2_{L^2(\Cb_{\ub}^u)} + \Vert \Lb \phi_{k} \Vert^2_{L^2(\Cb_{\ub}^u)}.
\ee

The inhomogeneous energies are defined as
\be\label{sum}
E^2_{\leq k}(u, \ub) = \displaystyle \sum_{0 \leq j \leq k} E^2_j (u, \ub), \quad \Eb^2_{\leq k}(u, \ub) = \displaystyle \sum_{0 \leq j \leq k} \Eb^2_j (u, \ub).
\ee

Given any subregion $D_{u^*, \ub^*}$ in the short pulse region II, we have

\begin{theorem}
  There exists a positive constant $\delta_0$ such that for $\delta \in (0, \delta_0)$, $ (u,\ub) \in D_{u^*, \ub^*}$, we have for $N \in \mathbb N$ and $N \geq 6$
  \be\label{boostrap}
  E_{\leq N}(u, \ub) +  \Eb_{\leq N}(u, \ub) \lesssim I_N(\theta_0, \theta_1, \phi_0, \phi_1),
  \ee
  where $I_N(\theta_0, \theta_1, \phi_0, \phi_1)$ is a positive constant depending only on the initial data up to $N+1$ order of derivatives.
\end{theorem}

The proof is based on a bootstrap argument. Assume that there exists a large constant $M$, which is to be determined later and only depends on the initial data, such that
\be\label{assump}
E_{\leq N}(u, \ub) +  \Eb_{\leq N}(u, \ub) \leq M,
\ee
for all $(u,\ub) \in D_{u^*, \ub^*}$.

Under the bootstrap assumption \eqref{assump}, we have the following decay estimates.

\subsection{Sobolev inequalities and decay estimates}

The following Sobolev inequalities was proved in \cite{W-W}. These estimates can help us to guarantee the hyperbolicity of \eqref{Eq} and the positivity of the energy.

\begin{lemma}\label{Sobolev}
Under the assumption \eqref{assump}, we have the following Sobolev inequalities, for $0 \leq l \leq k \leq N - 2$
\begin{align*}
& \Vert \slashed{\nabla} \theta_{k,l} \Vert_{L^{\infty}(S_{u, \ub})} \lesssim \delta^{\frac{3}{4}} |\ub|^{-\frac{3}{2}} M, && \Vert \slashed{\nabla} \phi_{k,l} \Vert_{L^{\infty}(S_{u, \ub})} \lesssim \delta^{\frac{3}{4}} |\ub|^{-\frac{3}{2}} M, \\
& \Vert L \theta_{k,l} \Vert_{L^{\infty}(S_{u, \ub})} \lesssim \delta |\ub|^{-\frac{3}{2}} M, && \Vert L \phi_{k,l} \Vert_{L^{\infty}(S_{u, \ub})} \lesssim \delta |\ub|^{-\frac{3}{2}} M, \\
& \Vert \Lb \theta_{k,l} \Vert_{L^{\infty}(S_{u, \ub})} \lesssim |\ub|^{-1} M, && \Vert \Lb \phi_{k,l} \Vert_{L^{\infty}(S_{u, \ub})} \lesssim |\ub|^{-1} M.
\end{align*}
\end{lemma}

Due to the finite speed of propagation of wave equations, we have $\theta_{k,l} = 0, \phi_{k,l} = 0$ on $C_0$. Integrating $\Lb \theta_{k,l}$ in the $u= \frac{1}{2} (t-r)$ direction from  $C_0$ or $\Sigma_1$, we can get the following lemma in the short pulse region II.

\begin{lemma}\label{Sob-L}
For $0 \leq l \leq k \leq N - 2$, it holds that
\be
\Vert \theta_{k,l} \Vert_{L^{\infty}(S_{u, \ub})} \lesssim \delta |\ub|^{-1} M, \quad \Vert \phi_{k,l} \Vert_{L^{\infty}(S_{u, \ub})} \lesssim \delta |\ub|^{-1} M.
\ee
\end{lemma}

\begin{proof}
For $1- \delta \leq \ub \leq 1$, integrating $\Lb \theta_{k,l}$ along $u= \frac{1}{2} (t-r)$ direction from initial hyperplane $\Sigma_1$ and noting the constraints \eqref{Con1}-\eqref{Con3} on initial data, we can get
\bes
 \theta_{k,l} (u, \ub, v, w) = \theta_{k,l} (u_0 + \ub =1, v, w) + \int_{u_0}^u \Lb \theta_{k,l} \di u,
\ees
then
\bes
  \Vert \theta_{k,l} \Vert_{L^{\infty}(S_{u, \ub})} \lesssim \delta + \delta |\ub|^{-1} M \lesssim \delta |\ub|^{-1} M.
\ees

For $\ub > 1$, integrating $\Lb \theta_{k,l}$ along $u= \frac{1}{2} (t-r)$ direction starting from  $C_0$, we have
\bes
 \theta_{k,l} (u, \ub, v, w)  = \theta_{k,l} (0,  \ub, v, w) + \int_0^u \Lb \theta_{k,l} \di u.
\ees
Then
\bes
 \Vert \theta_{k,l} \Vert_{L^{\infty}(S_{u, \ub})} \lesssim \delta |\ub|^{-1} M.
\ees

The estimate of $\phi_{k,l}$ follows from the same calculations.
\end{proof}

For the derivatives with at least one good derivative $\bar{\partial} \in \{L,\slashed{\nabla}\}$, we have the following better point-wise decay estimates than those in Lemma \ref{Sobolev}.

\begin{lemma}\label{decay}
For $0 \leq l \leq k \leq N - 3$, we have
\begin{align*}
& | L^2 \theta_{k,l} | \lesssim \delta |\ub|^{-\frac{5}{2}} M,  \quad  &&| \Lb^2 \theta_{k,l} | \lesssim \delta^{-1} |\ub|^{-1} M, \quad && | L \Lb \theta_{k,l} | = | \Lb L \theta_{k,l} | \lesssim  |\ub|^{-2} M, \\
 & | L  \slashed{\nabla} \theta_{k,l} | \lesssim \delta^{\frac{3}{4}} |\ub|^{-\frac{5}{2}} M, \quad &&| \slashed{\nabla} L \theta_{k,l} | \lesssim \delta^{\frac{3}{4}} |\ub|^{-\frac{5}{2}} M, \quad && | \Lb \slashed{\nabla} \theta_{k,l} | \lesssim |\ub|^{-2} M,\\
 & | \slashed{\nabla} \Lb \theta_{k,l} | \lesssim |\ub|^{-2} M, \quad && |\slashed{\nabla}^2 \theta_{k,l} | \lesssim \delta^{\frac{3}{4}} |\ub|^{-\frac{5}{2}} M. &&
\end{align*}
The same estimates holds for $\phi_{k,l}$.
\end{lemma}

\begin{proof}
Since
\be\label{L}
L = \frac{S+\omega^j\Omega_{0j}}{t+r},
\ee
we have
\bes
\begin{aligned}
 L^2 \theta_{k,l} & = L \Big( \frac{S+\omega^j\Omega_{0j}}{t+r} \Big) \theta_{k,l} = \Big( L \big( \frac{1}{t+r} \big) \Big) (S+\omega^j\Omega_{0j}) \theta_{k,l} +  \frac{1}{t+r} L (S+\omega^j\Omega_{0j}) \theta_{k,l}\\
 & = -\frac{2}{(t+r)^2} (S+\omega^j\Omega_{0j}) \theta_{k,l} + \frac{1}{t+r} (L S+\omega^j L \Omega_{0j}) \theta_{k,l}\\
 & = -\frac{2}{t+r} L \theta_{k,l} + \frac{1}{t+r} (L S+\omega^j L \Omega_{0j}) \theta_{k,l},
\end{aligned}
\ees
here we use the fact $L \omega^j = 0$ in the second line, then
\be\label{LL}
| L^2 \theta_{k,l} | \lesssim |\ub|^{-1} | L \theta_{k,l} | + |\ub|^{-1} | L \theta_{k+1,l} | \lesssim \delta |\ub|^{-\frac{5}{2}} M.
\ee
The same derivation leads to the third estimate for $L \Lb \theta_{k,l} $. Noting that $[L, \Lb] = 0$, the estimate holds for $\Lb L \theta_{k,l} $.

For the $\Lb^2$ derivative, we have
\be\label{LbLb}
 | \Lb^2 \theta_{k,l} | = | \Lb \big(\p_t - \omega_i \p_i) \theta_{k,l} | \lesssim | \Lb \p \theta_{k,l} |  \lesssim \delta^{-1}  | \Lb  \theta_{k+1, l+1} | \lesssim \delta^{-1} |\ub|^{-1} M.
\ee

For $L\slashed{\nabla}\theta_{k,l}$, we get
\begin{align*}
|L \slashed{\nabla} \theta_{k,l} | \lesssim & \, \sum_i \Big| L \big( \frac{\omega^j}{r} \Omega_{ij} \big) \theta_{k,l}\Big|
 \lesssim  \sum_{i,j} \Big( \Big| \frac{1}{r} L \Omega_{ij} \theta_{k,l} \Big| + \Big| \frac{1}{r^2} \Omega_{ij} \theta_{k,l} \Big| \Big)
   \\
\lesssim & \,|\ub|^{-1}  \big(| L \theta_{k+1,l} | + |\slashed{\nabla} \theta_{k,l} | \big) \lesssim \delta^{\frac{3}{4}} |\ub|^{-\frac{5}{2}} M,
\end{align*}
here we have used \eqref{omega3} and $r \sim |\ub|$ in the short pulse region II.

On the other hand, we have $[\slashed{\nabla}, L] = \frac{1}{r} \slashed{\nabla}$, thus
\bes
|\slashed{\nabla} L \theta_{k,l} | \lesssim  | L \slashed{\nabla} \theta_{k,l} | + \frac{1}{r} | \slashed{\nabla} \theta_{k,l} | \lesssim \delta^{\frac{3}{4}} |\ub|^{-\frac{5}{2}} M.
\ees
Similar calculation gives the estimates for $\Lb \slashed{\nabla}$ and $\slashed{\nabla} \Lb$ since $[\slashed{\nabla}, \Lb] = - \frac{1}{r} \slashed{\nabla}$.

For the last one, we have
\begin{align}\label{popo}
  |\slashed{\nabla}^2 \theta_{k,l} | \lesssim \sum_i  \Big| \slashed{\nabla} \big( \frac{\omega^j}{r} \Omega_{ij} \big) \theta_{k,l}\Big|   \lesssim \Big| \frac{1}{r} \slashed{\nabla} \theta_{k,l} \Big| +   \Big| \frac{1}{r} \slashed{\nabla} \theta_{k+1,l} \Big| \lesssim  \delta^{\frac{3}{4}} |\ub|^{-\frac{5}{2}} M.
\end{align}

\end{proof}

\subsection{Higher order equations}

We are turning to deduce the higher order equations of \eqref{Eq}. From Lemma \ref{Sob-L}, we know that $\cos^2\theta  \sim 1$ if $\delta$ is sufficiently small. For simplification of calculations, we can rewrite \eqref{Eq} as
\be\label{Eqnew}
\left\{
  \begin{aligned}
  &\Box \theta = -\frac{1}{2} \sin (2\theta) \, Q(\phi, \phi) + \frac{1}{4} \sin (2\theta) \,  Q^{\alpha\beta}(\theta, \phi) Q_{\alpha\beta}(\theta, \phi) + \frac{1}{2} \cos^2 \theta Q^{\alpha\beta} (\phi, Q_{\alpha\beta} (\theta, \phi) ), \\
  &\Box \phi =  \frac{\sin (2\theta)}{\cos^2 \theta} \, Q(\theta, \phi) + \frac{1}{2}  Q^{\alpha\beta} (\theta, Q_{\alpha\beta} (\phi, \theta)).
  \end{aligned}
\right.
\ee

Applying $Z^k$ on both sides of \eqref{Eqnew}, we can get the following lemma directly without proof.

\begin{lemma}
Commuting \eqref{Eq} with the Lorentz vector field $Z \in \mathcal{Z}$ for $ k \leq N$ times, we have
  \begin{align}\label{Eq1}
  \Box Z^k \theta & = \sum_{\substack{|i_1| + \cdots + |i_m| \leq |k| \\ 2 \leq m \leq |k|+2}} a_{i_1, \cdots, i_m} \, (\sin (2\theta))^{(m-2)} Q(Z^{i_1} \phi, Z^{i_2} \phi) Z^{i_3} \theta \cdots Z^{i_m} \theta  \nnb\\
  &  + \sum_{\substack{|i_1| + \cdots + |i_m| \leq |k| \\ 4 \leq m \leq |N|+4}} b_{i_1, \cdots, i_m}  \, (\sin (2\theta))^{(m-4)} Q^{\alpha\beta}(Z^{i_1} \theta,  Z^{i_2} \phi) Q_{\alpha\beta}( Z^{i_3} \theta,  Z^{i_4} \phi) Z^{i_5} \theta \cdots Z^{i_m} \theta \nnb\\
  &  + \sum_{\substack{|i_1| + \cdots + |i_m| \leq |k| \\ |i_2|, |i_3| < |k| \\ 3 \leq m \leq |k|+3}} c_{i_1, \cdots, i_m} \, (\cos^2 \theta)^{(m-3)} Q^{\alpha\beta} (Z^{i_1}\phi, Q_{\alpha\beta} (Z^{i_2}\theta,  Z^{i_3} \phi)) Z^{i_4} \theta \cdots Z^{i_m} \theta,\nnb \\
   & + \frac{1}{2} \cos^2 \theta Q^{\alpha\beta} (\phi, Q_{\alpha\beta} ( Z^k \theta, \phi) ) +  \frac{1}{2} \cos^2 \theta Q^{\alpha\beta} (\phi, Q_{\alpha\beta} ( \theta, Z^k \phi) ) \nnb \\
    & \, \nnb\\
    \text{and}\nnb\\
  \Box Z^k \phi & =  \sum_{\substack{|i_1| + \cdots + |i_m| \leq |k| \\ 2 \leq m \leq |k|+2}} d_{i_1, \cdots, i_m} \, \Big(\frac{2 \sin (2\theta)}{1+ \cos (2 \theta)} \Big)^{(m-2)} \, Q(Z^{i_1}\theta, Z^{i_2} \phi) Z^{i_3} \theta \cdots Z^{i_m} \theta \nnb\\
  &  + \sum_{\substack{|i_1| + |i_2| + |i_3| \leq |k| \\ |i_2|, |i_3| < |k| }} e_{i_1, i_2, i_3} \,Q^{\alpha\beta} (Z^{i_1}\theta, Q_{\alpha\beta} (Z^{i_2}\phi,  Z^{i_3} \theta)) \nnb\\
  & + \frac{1}{2}  Q^{\alpha\beta} (\theta, Q_{\alpha\beta} (Z^k \phi, \theta)) + \frac{1}{2}  Q^{\alpha\beta} (\theta, Q_{\alpha\beta} (\phi, Z^k \theta)),
  \end{align}
where $a_{i_1, \cdots, i_m}, b_{i_1, \cdots, i_m}, c_{i_1, \cdots, i_m}, d_{i_1, \cdots, i_m}, e_{i_1, i_2, i_3}$ are constants.
\end{lemma}

\subsection{Energy scheme}\label{3.3}

The system \eqref{Eq1} is not symmetric, in order to close the energy estimates, we multiply $\cos^2 \theta$ on both sides of the second equation of \eqref{Eq1} as follows
\begin{align}
	& \Box Z^k \theta - g_1^{\alpha \beta} \p_\alpha  \p_\beta Z^k \theta  -h_1^{\alpha\beta} \p_\alpha \p_\beta Z^k \phi \nnb \\
	= &\sum_{\substack{|i_1| + \cdots + |i_m| \leq |k| \\ 2 \leq m \leq |k|+2}} a_{i_1, \cdots, i_m} \, (\sin (2\theta))^{(m-2)} Q(Z^{i_1} \phi, Z^{i_2} \phi) Z^{i_3} \theta \cdots Z^{i_m} \theta \nnb  \\
	& + \sum_{\substack{|i_1| + \cdots + |i_m| \leq |k| \\ 4 \leq m \leq |k|+4}} b_{i_1, \cdots, i_m} \, (\sin (2\theta))^{(m-4)} Q^{\alpha \beta}(Z^{i_1} \theta,  Z^{i_2} \phi) Q_{\alpha \beta}( Z^{i_3} \theta,  Z^{i_4} \phi) Z^{i_5} \theta \cdots Z^{i_m} \theta \nnb \\
	&+ \sum_{\substack{|i_1| + \cdots + |i_m| \leq |k| \\ |i_2|, |i_3| < |k| \\ 3 \leq m \leq |k|+3} } c_{i_1, \cdots, i_m} \, (\cos^2 \theta)^{(m-3)} Q^{\alpha \beta} (Z^{i_1}\phi, Q_{\alpha \beta} (Z^{i_2}\theta,  Z^{i_3} \phi)) Z^{i_4} \theta \cdots Z^{i_m} \theta \nnb  \\
	&+ \, \cos^2 \theta \Big( Q^{\alpha \beta} (\phi , \p_\beta \phi) \p_\alpha Z^k \theta - Q^{\alpha \beta} (\phi , \p_\beta \theta) \p_\alpha Z^k \phi \Big), \label{Eq2-1} \\
	\text{and}\nnb\\
	& \cos^2 \theta \Box Z^k \phi - g_2^{\alpha \beta} \p_\alpha \p_\beta Z^k \phi -h_2^{\alpha \beta} \p_\alpha \p_\beta Z^k \theta \nnb  \\
	= &\sum_{\substack{|i_1| + \cdots + |i_m| \leq |k| \\ 2 \leq m \leq |k|+2}} d_{i_1, \cdots, i_m} \cos^2 \theta \, \Big(\frac{2 \sin (2\theta)}{1+ \cos (2 \theta)} \Big)^{(m-2)} \, Q(Z^{i_1}\theta, Z^{i_2} \phi) Z^{i_3} \theta \cdots Z^{i_m} \theta \nnb\\
	&  + \sum_{\substack{|i_1| + |i_2| + |i_3| \leq |k| \\ |i_2|, |i_3| < |k| }} e_{i_1, i_2, i_3} \cos^2 \theta \,\Big( Q^{\alpha \beta} (Z^{i_1}\theta, Q_{\alpha \beta} (Z^{i_2}\phi,  Z^{i_3} \theta)) \nnb + F(Z^{i_1}\theta, Z^{i_2}\phi, Z^{i_3} \theta) \Big) \nnb \\
	& + \, \cos^2 \theta \Big( Q^{\alpha \beta} (\theta , \p_\beta \theta) \p_\alpha Z^k \phi - Q^{\alpha \beta} (\theta , \p_\beta \phi) \p_\alpha Z^k \theta \Big), \label{Eq2-2}
\end{align}
where
\be\label{metric}
\begin{aligned}
	&g_1^{\alpha \beta} = \cos^2 \theta \big(\p^\alpha \phi \, \p^\beta \phi - g^{\alpha \beta} Q(\phi, \phi)\big),\,\,&& h_1^{\alpha \beta} = \cos^2 \theta \big( g^{\alpha \beta} Q(\theta, \phi)- \p^\alpha \theta \, \p^\beta \phi \big) , \\
	& g_2^{\alpha \beta} = \cos^2 \theta \big(  \p^\alpha \theta \, \p^\beta \theta - g^{\alpha \beta} Q(\theta, \theta)\big),\,\,
	&& h_2^{\alpha \beta} = \cos^2 \theta \big( g^{\alpha \beta} Q(\theta, \phi)- \p^\alpha \phi \, \p^\beta \theta \big).
\end{aligned}
\ee

\subsubsection{Energy scheme for $\tilde{L}$}
Firstly, we multiply both sides of the equations \eqref{Eq2-1} and \eqref{Eq2-2} by $\delta^l$ and derive the energy for $\theta_{k,l}$ and $\phi_{k,l}$ corresponding to vector field $\tilde{L} = L + ((L\theta)^2 + (L\phi)^2) \Lb$. The left-hand sides of \eqref{Eq2-1} and \eqref{Eq2-2} become
\be\label{eq1}
\big(\Box \theta_{k,l} - g_1^{\alpha \beta} \p_\alpha  \p_\beta \theta_{k,l}  -h_1^{\alpha\beta} \p_\alpha \p_\beta \phi_{k,l} \big) \tilde{L} \theta_{k,l},
\ee
and
\be\label{eq2}
 \big(\cos^2 \theta \, \Box \phi_{k,l} - g_2^{\alpha\beta} \p_\alpha  \p_\beta \phi_{k,l}  -h_2^{\alpha\beta} \p_\alpha \p_\beta \theta_{k,l} \big) \tilde{L} \phi_{k,l}.
\ee

For the terms $\Box \theta_{k,l} \cdot \tilde{L} \theta_{k,l}$ and $\cos^2 \theta \, \Box \phi_{k,l} \cdot \tilde{L} \phi_{k,l}$ in \eqref{eq1} and \eqref{eq2}, we use the vector field method and integrate in the domain $D_{u, \ub}$. Since
\be\label{Box}
\Box \theta_{k,l} \cdot \tilde{L} \theta_{k,l} = \Box \theta_{k,l} \cdot L  \theta_{k,l} + \Box \theta_{k,l} \cdot ((L\theta)^2 + (L\phi)^2) \Lb \theta_{k,l},
\ee
then according to \eqref{momentum}, \eqref{div} and \eqref{deformation}, the integration of $\Box \theta_{k,l} \cdot L  \theta_{k,l}$ becomes
\begin{align*}
& \iint_{D_{u, \ub}} \Box \theta_{k,l} \cdot L  \theta_{k,l} \\
= & \, \int_{\p D_{u, \ub}} T(L, N) [\theta_{k,l}] - \iint_{D_{u, \ub}} K^L [\theta_{k,l}] \\
= & \, \int_{C_u} T(L, L) [\theta_{k,l}] + \int_{\Cb_{\ub}} T(L, \Lb)  [\theta_{k,l}]- \int_{\Sigma_1} T\left(L, \frac{1}{2}(L + \Lb)\right) [\theta_{k,l}] - \iint_{D_{u, \ub}} K^L [\theta_{k,l}] \\
= & \, \int_{C_u} |L \theta_{k,l}|^2  +  \int_{\Cb_{\ub}} | \slashed{\nabla} \theta_{k,l} |^2 - \frac{1}{2} \int_{\Sigma_1} | L \theta_{k,l}|^2 + |\slashed{\nabla} \theta_{k,l} |^2 - \iint_{D_{u, \ub}} K^L [\theta_{k,l}] ,
\end{align*}
here we use the fact that the data on $C_0$ is zero and
\be\label{KL}
K^L [\theta_{k,l}] = \frac{1}{r} L \theta_{k,l}\, \Lb \theta_{k,l}
\ee
according to the third equation in \eqref{deformation}.

 For the second term of the right-hand side of \eqref{Box}, we take $X = \Lb$ and
\be\label{f}
f = (L\theta)^2 + (L\phi)^2,
\ee
 in \eqref{deformation2}, then
\begin{align*}
& \iint_{D_{u, \ub}} \Box \theta_{k,l} \cdot ((L\theta)^2 + (L\phi)^2) \Lb \theta_{k,l}  \\
= & \, \int_{\p D_{u, \ub}} T(f\Lb, N) [\theta_{k,l}] - \iint_{D_{u, \ub}} K^{f\Lb}  [\theta_{k,l}]\\
= & \, \int_{C_u} f \, T(\Lb, L)  [\theta_{k,l}]+ \int_{\Cb_{\ub}} f \, T(\Lb, \Lb) [\theta_{k,l}] - \int_{\Sigma_1} f \, T\left(\Lb, \frac{1}{2}(L + \Lb)\right) [\theta_{k,l}] - \iint_{D_{u, \ub}} K^{f\Lb}  [\theta_{k,l}]\\
= & \, \int_{C_u} f \, |\slashed{\nabla} \theta_{k,l} |^2 + \int_{\Cb_{\ub}} f \,  |\Lb \theta_{k,l}|^2 - \frac{1}{2} \int_{\Sigma_1} f \, \Big( |\slashed{\nabla} \theta_{k,l} |^2 + | \Lb \theta_{k,l}|^2 \Big) - \iint_{D_{u, \ub}} K^{f\Lb}  [\theta_{k,l}] ,
\end{align*}
where
\be\label{fLb}
 K^{f\Lb}  [\theta_{k,l}] = -\frac{1}{r} f \, L \theta_{k,l}\, \Lb \theta_{k,l} + Q( f ,  \theta_{k,l}) \, \Lb \theta_{k,l}- \frac{1}{2} \Lb f  \, |\nabla \theta_{k,l} |^2
\ee
according to \eqref{momentum} - \eqref{KX}.

Similarly, we compute the term $\cos^2 \theta \, \Box \phi_{k,l} \, \tilde{L} \phi_{k,l}$ as follows,
\begin{align*}
& \iint_{D_{u, \ub}}\cos^2 \theta \, \Box \phi_{k,l} \cdot \tilde{L} \phi_{k,l} \\
= & \,  \iint_{D_{u, \ub}} \Box \phi_{k,l} \cdot \cos^2 \theta \, L \phi_{k,l} + \Box \phi_{k,l} \cdot f \cos^2 \theta \, \Lb \phi_{k,l} \\
= & \, \int_{\p D_{u, \ub}} T(\cos^2 \theta L, N)  [\phi_{k,l}]+ T( f \cos^2 \theta \,  \Lb, N)[\phi_{k,l}] - \iint_{D_{u, \ub}} K^{\cos^2 \theta L}[\phi_{k,l}] + K^{f \cos^2 \theta   \Lb} [\phi_{k,l}]\\
= & \, \int_{C_u} \cos^2 \theta \big( |L \phi_{k,l}|^2 + f \, | \slashed{\nabla} \phi_{k,l} |^2 \big) +  \int_{\Cb_{\ub}} \cos^2 \theta \big( | \slashed{\nabla} \phi_{k,l} |^2 + f | \Lb \phi_{k,l}|^2 \big) \\
& - \frac{1}{2} \int_{\Sigma_1} \cos^2 \theta \big( | L \phi_{k,l}|^2 + |\slashed{\nabla} \phi_{k,l} |^2 + f \, (|\slashed{\nabla} \phi_{k,l} |^2 + | \Lb \phi_{k,l}|^2) \big) \\
&  - \iint_{D_{u, \ub}} K^{\cos^2 \theta L}[\phi_{k,l}] + K^{f \cos^2 \theta   \Lb} [\phi_{k,l}],
\end{align*}
where $f$ is defined in \eqref{f} and
\be\label{cosL}
K^{\cos^2 \theta L}[\phi_{k,l}] =  \frac{1}{r} \cos^2 \theta  L \phi_{k,l}\, \Lb \phi_{k,l} + Q(\cos^2  \theta, \phi_{k,l}) L \phi_{k,l} - \frac{1}{2}  L (\cos^2 \theta)  \, |\nabla \phi_{k,l}|^2,
\ee
\be\label{fcosLb}
K^{f \cos^2 \theta   \Lb} [\phi_{k,l}] = -\frac{1}{r} f \cos^2 \theta  L \phi_{k,l}\, \Lb \phi_{k,l} + Q(f \cos^2  \theta, \phi_{k,l}) \Lb \phi_{k,l} - \frac{1}{2}  \Lb (f \cos^2 \theta)  \, |\nabla \phi_{k,l}|^2
\ee
according to \eqref{momentum} - \eqref{KX}.

To address the remaining terms in equations \eqref{eq1} and \eqref{eq2}, we will follow a two-step approach as outlined below. We will now detail the calculation of the term $g_1^{\alpha \beta} \, \p_\alpha \p_\beta \theta_{k,l} \, \tilde{L} \theta_{k,l}$, analogous methods can be applied to derive similar results for the remaining terms.

\textbf{Step 1.} Establishing the differential identities.

\begin{align}
& g_1^{\alpha \beta} \, \p_\alpha \p_\beta \theta_{k,l} \, \tilde{L} \theta_{k,l} \nnb \\
 = & \, \p_\alpha ( g_1^{ \alpha \beta} \, \p_\beta \theta_{k,l} \,\tilde{L} \theta_{k,l}) - g_1^{\alpha\beta} \, \p_\beta \theta_{k,l} \, \p_\alpha \tilde{L} \theta_{k,l} - \p_\alpha g_1^{\alpha\beta} \, \p_\beta \theta_{k,l} \,\tilde{L} \theta_{k,l} \nnb \\
= & \, \p_\alpha ( g_1^{\alpha\beta} \, \p_\beta \theta_{k,l} \,\tilde{L} \theta_{k,l}) - g_1^{\alpha\beta} \, \p_\beta \theta_{k,l} \, \big( L \p_\alpha \theta_{k,l} + [ \p_\alpha, L] \theta_{k,l} \big) - f \,g_1^{\alpha\beta}\p_\beta \theta_{k,l} \big( \Lb \p_\alpha \theta_{k,l} + [ \p_\alpha, \Lb] \theta_{k,l} \big) \nnb \\
& -  g_1^{\alpha\beta} \p_\alpha f \, \p_\beta \theta_{k,l} \, \Lb \theta_{k,l}  - \p_\alpha g_1^{\alpha\beta} \, \p_\beta \theta_{k,l} \,\tilde{L} \theta_{k,l} \nnb \\
 = & \, \p_\alpha ( g_1^{\alpha\beta} \, \p_\beta \theta_{k,l} \,\tilde{L} \theta_{k,l})
  - \frac{1}{2} L ( g_1^{\alpha\beta} \, \p_\alpha \theta_{k,l} \, \p_\beta \theta_{k,l}) + \frac{1}{2} L g_1^{\alpha\beta} \, \p_\alpha \theta_{k,l} \, \p_\beta \theta_{k,l} - \frac{1}{r} g_1^{i \beta} \, \po_i \theta_{k,l} \,\p_\beta \theta_{k,l} \nnb \\
  &
  - \frac{1}{2} \Lb \Big( f \, g_1^{\alpha\beta} \p_\alpha \theta_{k,l} \, \p_\beta \theta_{k,l} \Big) + \frac{1}{2} \Lb \big( f \, g_1^{\alpha\beta} \big) \p_\alpha \theta_{k,l} \, \p_\beta \theta_{k,l} + \frac{1}{r} f g_1^{i \beta} \, \po_i \theta_{k,l} \,\p_\beta \theta_{k,l} \nnb \\
 & -  g_1^{\alpha\beta} \p_\alpha f \, \p_\beta \theta_{k,l} \, \Lb \theta_{k,l}  - \p_\alpha g_1^{\alpha\beta} \, \p_\beta \theta_{k,l} \,\tilde{L} \theta_{k,l} \nnb \\
   = & \, \p_\alpha ( g_1^{\alpha\beta} \, \p_\beta \theta_{k,l} \,\tilde{L} \theta_{k,l})
  - \frac{1}{2} L ( g_1^{\alpha\beta} \, \p_\alpha \theta_{k,l} \, \p_\beta \theta_{k,l}) - \frac{1}{2} \Lb \Big( f \, g_1^{\alpha\beta} \p_\alpha \theta_{k,l} \, \p_\beta \theta_{k,l} \Big) + Q^{(1)}, \label{g1L}
\end{align}
here we have used the symmetry of $g_{\alpha \beta}$ and
\begin{align}\label{Q1}
Q^{(1)} = & \, \frac{1}{2} L g_1^{\alpha\beta} \, \p_\alpha \theta_{k,l} \, \p_\beta \theta_{k,l}  + \frac{1}{2} \Lb \big( f \, g_1^{\alpha\beta} \big) \p_\alpha \theta_{k,l} \, \p_\beta \theta_{k,l} - \p_\alpha g_1^{\alpha\beta} \, \p_\beta \theta_{k,l} \,\tilde{L} \theta_{k,l} \nnb \\
   & \, -  g_1^{\alpha\beta} \p_\alpha f \, \p_\beta \theta_{k,l} \, \Lb \theta_{k,l}- \frac{1}{r} (1-f) g_1^{i \beta} \, \po_i \theta_{k,l} \,\p_\beta \theta_{k,l}.
\end{align}

\textbf{Step 2}. Integration by parts in the domain $D_{u, \ub}$.

We apply Stokes' theorem to integrate both sides of equation \eqref{g1L} over the domain $D_{u, \ub}$, and get
\begin{align*}
 \iint_{D_{u, \ub}} g_1^{\alpha \beta} \, \p_\alpha \p_\beta \theta_{k,l} \, \tilde{L} \theta_{k,l}
=  \int_{C_u} P^{(1)}_1 + \int_{\Cb_{\ub}}  P^{(1)}_2 - \frac{1}{2} \int_{\Sigma_1} (P^{(1)}_1  + P^{(1)}_2) + \iint_{D_{u, \ub}}  Q^{(1)},
\end{align*}
where
\begin{align}
  P^{(1)}_1 = & \, - g_1^{0 \beta} \, \p_\beta \theta_{k,l} \,\tilde{L} \theta_{k,l} + \omega_i g_1^{i \beta} \, \p_\beta \theta_{k,l} \,\tilde{L} \theta_{k,l} + f \, g_1^{\alpha\beta} \p_\alpha \theta_{k,l} \, \p_\beta \theta_{k,l}
  \nnb \\
  = & \, \cos^2 \theta \Big(L \phi \, Q(\phi,  \theta_{k,l}) -  Q (\phi, \phi) L \theta_{k,l} \Big) \tilde{L} \theta_{k,l} + f \, \cos^2 \theta  \Big( Q( \phi, \theta_{k,l})^2 -  Q (\phi, \phi)  Q(\theta_{k,l} ,  \theta_{k,l}) \Big) , \nnb \\
   P^{(1)}_2 = & \, - g_1^{0 \beta} \, \p_\beta \theta_{k,l} \,\tilde{L} \theta_{k,l} - \omega_i g_1^{i \beta} \, \p_\beta \theta_{k,l} \,\tilde{L} \theta_{k,l} + g_1^{\alpha\beta} \p_\alpha \theta_{k,l} \, \p_\beta \theta_{k,l}
  \nnb \\ =  & \,\cos^2 \theta \Big(\Lb \phi \, Q (\phi, \theta_{k,l}) -  Q (\phi, \phi) \Lb \theta_{k,l} \Big) \tilde{L} \theta_{k,l} + \cos^2 \theta  \Big( Q( \phi, \theta_{k,l})^2 -  Q (\phi, \phi)  Q(\theta_{k,l} ,  \theta_{k,l}) \Big), \label{PP1}
\end{align}
and $Q^{(1)}$ is defined in \eqref{Q1}.

Following analogous calculations, we obtain
\begin{align*}
& h_1^{\alpha \beta} \p_\alpha \p_\beta \phi_{k,l} \tilde{L} \theta_{k,l}\\
  = & \, \, \p_\alpha ( h_1^{\alpha \beta} \, \p_\beta \phi_{k,l} \,\tilde{L} \theta_{k,l})
  -  L ( h_1^{\alpha \beta} \, \p_\alpha \theta_{k,l} \, \p_\beta \phi_{k,l}) - \Lb \big( f \, h_1^{\alpha \beta}  \p_\alpha \theta_{k,l} \, \p_\beta \phi_{k,l} \big) \\
   &\,  +  h_1^{\alpha \beta} \, \p_\alpha \theta_{k,l} \, \tilde{L} \p_\beta \phi_{k,l} + Q^{(2)},
\end{align*}
where
\begin{align}\label{Q2}
Q^{(2)} = & \, L h_1^{\alpha \beta} \, \p_\alpha \theta_{k,l} \, \p_\beta \phi_{k,l} +  \Lb \big(f \, h_1^{\alpha \beta} \, \big) \p_\alpha \theta_{k,l} \, \p_\beta \phi_{k,l} - \p_\alpha h_1^{\alpha \beta} \, \p_\beta \phi_{k,l} \,\tilde{L} \theta_{k,l} \nnb \\
  & \,- h_1^{\alpha \beta} \p_\alpha f \, \p_\beta \phi_{k,l} \, \Lb \theta_{k,l} - \frac{1}{r} (1-f) h_1^{i \beta} \, \po_i \theta_{k,l} \,\p_\beta \phi_{k,l}.
\end{align}
After integrating by parts over the domain
$D_{u,\ub}$, we find
\begin{align}
\iint_{D_{u, \ub}} h_1^{\alpha \beta} \p_\alpha \p_\beta \phi_{k,l} \tilde{L} \theta_{k,l}
= & \, \int_{C_u} P^{(2)}_1 + \int_{\Cb_{\ub}}  P^{(2)}_2 - \frac{1}{2} \int_{\Sigma_1} \big(P^{(2)}_1 +  P^{(2)}_2 \big) + \iint_{D_{u, \ub}}  Q^{(2)}  \nnb \\
& + \iint_{D_{u, \ub}} h_1^{\alpha \beta} \, \p_\alpha \theta_{k,l} \, \tilde{L} \p_\beta \phi_{k,l}, \label{h1L}
\end{align}
where
\begin{align}
 P^{(2)}_1 = & \, \cos^2 \theta \Big(Q (\theta, \phi) L \phi_{k,l} - L \theta \, Q(\phi,  \phi_{k,l}) \Big) \tilde{L} \theta_{k,l} \nnb \\
 & \, + 2f \, \cos^2 \theta  \Big( Q( \theta, \phi)  Q(\theta_{k,l} ,  \phi_{k,l}) - Q(\theta ,  \theta_{k,l}) Q (\phi, \phi_{k,l}) \Big) , \nnb \\
 P^{(2)}_2 = & \, \cos^2 \theta \Big(Q (\theta, \phi) \Lb \phi_{k,l} - \Lb \theta \, Q(\phi,  \phi_{k,l}) \Big) \tilde{L} \theta_{k,l}  \nnb \\
  & \, + 2 \cos^2 \theta  \Big( Q( \theta, \phi)  Q(\theta_{k,l} ,  \phi_{k,l}) - Q(\theta ,  \theta_{k,l}) Q (\phi, \phi_{k,l}) \Big) . \label{PP2}
\end{align}

The remaining terms in equation \eqref{eq2} can be handled in the following manner
\begin{align*}
& g_2^{\alpha \beta} \, \p_\alpha \p_\beta \phi_{k,l} \, \tilde{L} \phi_{k,l} \\
   = & \, \p_\alpha ( g_2^{\alpha\beta} \, \p_\beta \phi_{k,l} \,\tilde{L} \phi_{k,l})
  - \frac{1}{2} L ( g_2^{\alpha\beta} \, \p_\alpha \phi_{k,l} \, \p_\beta \phi_{k,l}) - \frac{1}{2} \Lb \Big( f \, g_2^{\alpha\beta} \p_\alpha \phi_{k,l} \, \p_\beta \phi_{k,l} \Big) + Q^{(3)},
\end{align*}
where
\begin{align}\label{Q3}
Q^{(3)} = & \, \frac{1}{2} L g_2^{\alpha\beta} \, \p_\alpha \phi_{k,l} \, \p_\beta \phi_{k,l}  + \frac{1}{2} \Lb \big( f \, g_2^{\alpha\beta} \big) \p_\alpha \phi_{k,l} \, \p_\beta \phi_{k,l} - \p_\alpha g_2^{\alpha\beta} \, \p_\beta \phi_{k,l} \,\tilde{L} \phi_{k,l} \nnb \\
   & -  g_2^{\alpha\beta} \p_\alpha f \, \p_\beta \phi_{k,l} \, \Lb \phi_{k,l} - \frac{1}{r} (1-f) g_2^{i \beta} \, \po_i \phi_{k,l} \,\p_\beta \phi_{k,l},
\end{align}
and
\bes
 h_2^{\alpha \beta} \p_\alpha \p_\beta \theta_{k,l} \tilde{L} \phi_{k,l}
  = \p_\alpha ( h_2^{\alpha \beta} \, \p_\beta \theta_{k,l} \,\tilde{L} \phi_{k,l})  - h_2^{\alpha \beta} \, \p_\beta \theta_{k,l} \, \tilde{L}  \p_\alpha \phi_{k,l} + Q^{(4)},
\ees
where
\be\label{Q4}
Q^{(4)} = - \p_\alpha h_2^{\alpha \beta} \, \p_\beta \theta_{k,l} \,\tilde{L} \phi_{k,l} - h_2^{\alpha \beta} \p_\alpha f \, \p_\beta \theta_{k,l} \, \Lb \phi_{k,l}
  - \frac{1}{r} (1-f) h_2^{i \beta} \, \po_i \phi_{k,l} \,\p_\beta \theta_{k,l}.
\ee

Integrating by parts, we arrive at
\be
\iint_{D_{u, \ub}} g_2^{\alpha \beta} \, \p_\alpha \p_\beta \phi_{k,l} \, \tilde{L} \phi_{k,l}
= \int_{C_u} P^{(3)}_1 + \int_{\Cb_{\ub}}  P^{(3)}_2 - \frac{1}{2} \int_{\Sigma_1} (P^{(3)}_1  + P^{(3)}_2) + \iint_{D_{u, \ub}}  Q^{(3)},
\ee
and
\begin{align}
\iint_{D_{u, \ub}} h_2^{\alpha \beta} \p_\alpha \p_\beta \theta_{k,l} \tilde{L} \phi_{k,l}
= & \, \int_{C_u} P^{(4)}_1 + \int_{\Cb_{\ub}}  P^{(4)}_2 - \frac{1}{2} \int_{\Sigma_1} (P^{(4)}_1  + P^{(4)}_2) + \iint_{D_{u, \ub}}  Q^{(4)}  \nnb \\
& - \iint_{D_{u, \ub}} h_2^{\alpha \beta} \, \p_\beta \theta_{k,l} \, \tilde{L}  \p_\alpha \phi_{k,l}, \label{h2L}
\end{align}
where
\begin{align}
 & P^{(3)}_1 = \cos^2 \theta \Big(L \theta \, Q(\theta,  \phi_{k,l}) -  Q (\theta, \theta) L \phi_{k,l} \Big) \tilde{L} \phi_{k,l} + f \, \cos^2 \theta  \Big( Q( \theta, \phi_{k,l})^2 -  Q (\theta, \theta)  Q(\phi_{k,l} ,  \phi_{k,l}) \Big) , \nnb \\
  & P^{(3)}_2 = \cos^2 \theta \Big(\Lb \theta \, Q (\theta, \phi_{k,l}) -  Q (\theta, \theta) \Lb \phi_{k,l} \Big) \tilde{L} \phi_{k,l} + \cos^2 \theta  \Big( Q( \theta, \phi_{k,l})^2 -  Q (\theta, \theta)  Q(\phi_{k,l} ,  \phi_{k,l}) \Big),\label{PP3}
\end{align}
and
\begin{align}
   & P^{(4)}_1 = \cos^2 \theta \Big(Q (\theta, \phi) L \theta_{k,l} - L \phi \, Q(\theta,  \theta_{k,l}) \Big) \tilde{L} \phi_{k,l},  \nnb \\
  & P^{(4)}_2 = \cos^2 \theta \Big(Q (\theta, \phi) \Lb \theta_{k,l} - \Lb \phi \, Q(\theta,  \theta_{k,l}) \Big) \tilde{L} \phi_{k,l}. \label{PP4}
\end{align}

We proceed to derive the energy terms associated with the vector field $\tilde{L}$. Given that $h_1^{\alpha \beta} = h_2^{\beta \alpha}$ according to \eqref{metric}, the highest order derivatives of $k+2$ in the expressions $h_1^{\alpha \beta} \, \p_\alpha \theta_{k,l} \, \tilde{L} \p_\beta \phi_{k,l}$ from equation \eqref{h1L}, as well as
$- h_2^{\alpha \beta} \, \p_\beta \theta_{k,l} \, \tilde{L}  \p_\alpha \phi_{k,l}$ from \eqref{h2L} cancel out. Combining the aforementioned integral equations, we obtain
\begin{align}
&\iint_{D_{u, \ub}} \big(\Box \theta_{k,l} - g_1^{\alpha \beta} \p_\alpha  \p_\beta \theta_{k,l}  -h_1^{\alpha \beta} \p_\alpha \p_\beta \phi_{k,l} \big) \tilde{L} \theta_{k,l} \nnb \\
& \qquad \quad + \big(\cos^2 \theta \, \Box \phi_{k,l} - g_2^{\alpha \beta} \p_\alpha \p_\beta \phi_{k,l}  -h_2^{\alpha \beta} \p_\alpha \p_\beta \theta_{k,l} \big) \tilde{L} \phi_{k,l} \nnb \\
= & \, E^2_{k,l}(\theta, \phi, \tilde{L}) +  \Eb^{2}_{k,l}(\theta, \phi, \tilde{L}) -  E^{2}_{k,l}\big|_{t=1}(\theta, \phi, \tilde{L})  \nnb \\
&-\iint_{D_{u, \ub}} K^L [\theta_{k,l}]+ K^{f \Lb} [\theta_{k,l}]+ K^{\cos^2 \theta L} [\phi_{k,l}]+ K^{f \cos^2 \theta \Lb}[\phi_{k,l}] - Q^{(1)} - Q^{(2)} - Q^{(3)} - Q^{(4)} , \label{energy1}
\end{align}
where  $K^L [\theta_{k,l}]$, $ K^{f \Lb} [\theta_{k,l}]$, $ K^{\cos^2 \theta L} [\phi_{k,l}]$, $K^{f \cos^2 \theta \Lb}[\phi_{k,l}]$ are defined in \eqref{KL}, \eqref{fLb}, \eqref{cosL}, \eqref{fcosLb}, $ Q^{(1)} $, $ Q^{(2)} $, $ Q^{(3)} $, $ Q^{(4)} $ are defined in \eqref{Q1}, \eqref{Q2}, \eqref{Q3}, \eqref{Q4},
\begin{align}
 E^2_{k,l}(\theta, \phi, \tilde{L})
 = & \int_{C_{u}} |L \theta_{k,l}|^2 + f \, |\slashed{\nabla} \theta_{k,l} |^2 + \cos^2 \theta \big( |L \phi_{k,l}|^2 + f \, | \slashed{\nabla} \phi_{k,l} |^2 \big)  \nnb \\
 & \qquad -\big(  P^{(1)}_1 +  P^{(2)}_1 + P^{(3)}_1 +  P^{(4)}_1 \big), \label{Energy1}
\end{align}
\begin{align}
  \Eb^{2}_{k,l}(\theta, \phi, \tilde{L})
   = & \int_{\Cb_{\ub}}  | \slashed{\nabla} \theta_{k,l} |^2 +  f \,  |\Lb \theta_{k,l}|^2 + \cos^2 \theta \big( | \slashed{\nabla} \phi_{k,l} |^2 + f | \Lb \phi_{k,l}|^2 \big)  \nnb \\
   & \qquad - \big( P^{(1)}_2 +  P^{(2)}_2 + P^{(3)}_2 +  P^{(4)}_2 \big), \label{Energy2}
\end{align}
and
\begin{align*}
   E^{2}_{k,l} \big|_{t=1}(\theta, \phi, \tilde{L}) = & \, \frac{1}{2} \int_{\Sigma_1} \Big(  | L \theta_{k,l}|^2 + |\slashed{\nabla} \theta_{k,l} |^2 +
  f \, \big( |\slashed{\nabla} \theta_{k,l} |^2 + | \Lb \theta_{k,l}|^2 \big) \\
   & \qquad + \cos^2 \theta \big( | L \phi_{k,l}|^2 + |\slashed{\nabla} \phi_{k,l} |^2 + f \, (|\slashed{\nabla} \phi_{k,l} |^2 + | \Lb \phi_{k,l}|^2) \big) \\
   & \qquad - \big( P^{(1)}_1 +  P^{(2)}_1 + P^{(3)}_1 +  P^{(4)}_1 + P^{(1)}_2 +  P^{(2)}_2 + P^{(3)}_2 +  P^{(4)}_2 \big) \Big),
\end{align*}
here, $P^{(j)}_i$ are defined in \eqref{PP1}, \eqref{PP2}, \eqref{PP3} and \eqref{PP4}.

\begin{lemma}\label{lem-en1}
With the bootstrap assumption \eqref{assump}, we have
\begin{align}
  E^2_{k,l}(\theta, \phi, \tilde{L}) \sim \int_{C_{u}} ( L \theta_{k, l} )^2 + ( L \phi_{k, l} )^2 + \big( (L\theta)^2 + (L\phi)^2 \big)& \Big( | \slashed{\nabla} \theta_{k, l} |^2 + | \slashed{\nabla} \phi_{k, l} |^2 \nnb  \\
 & + \big( L \phi \Lb \theta_{k, l} - L \theta \Lb \phi_{k, l}\big)^2 \Big), \label{Energy11}
\end{align}
and
\begin{align}
  \Eb^{2}_{k,l}(\theta, \phi, \tilde{L})  \sim \int_{\Cb_{\ub}} |\slashed{\nabla} \theta_{k,l}|^2 + |\slashed{\nabla} \phi_{k,l}|^2 & \, + \big( (L\theta)^2 + (L\phi)^2 \big) \big( (\Lb \theta_{k,l})^2 + (\Lb \phi_{k,l})^2 \big)  \nnb \\
 & + (\Lb \phi \, L \theta_{k,l} - \Lb \theta \, L \phi_{k,l} )^2. \label{Energy12}
\end{align}
\end{lemma}

\begin{proof}
We begin with addressing the dominant quadratic terms associated with $L$ and $\slashed{\nabla}$. Given that $\delta$ is sufficiently small and $\cos^2\theta  \sim 1$. Consequently, for the initial term on the right-hand side of equation \eqref{Energy1}, we find that
\begin{align}
& |L \theta_{k,l}|^2 + f \, |\slashed{\nabla} \theta_{k,l} |^2 + \cos^2 \theta \big( |L \phi_{k,l}|^2 + f \, | \slashed{\nabla} \phi_{k,l} |^2 \big) \nnb \\
\sim & \,  ( L \theta_{k, l} )^2 + ( L \phi_{k, l} )^2 + \big( (L\theta)^2 + (L\phi)^2 \big) \big( | \slashed{\nabla} \theta_{k, l} |^2 + | \slashed{\nabla} \phi_{k, l} |^2 \big). \label{quadratic}
\end{align}
We will further assume that $\delta^{\frac{1}{2}} M < 1$.
The second line in the right-hand side of \eqref{Energy1} becomes
\begin{align}
 & \, - \Big(P^{(1)}_1 +  P^{(2)}_1 + P^{(3)}_1 +  P^{(4)}_1 \Big) \nnb \\
 = & \, \cos^2 \theta \Big( Q (\phi, \phi) L \theta_{k,l} - L \phi \, Q(\phi,  \theta_{k,l}) \Big) L \theta_{k,l} + \cos^2 \theta \Big( L \theta \, Q(\phi,  \phi_{k,l}) - Q (\theta, \phi) L \phi_{k,l} \Big) L \theta_{k,l} \nnb \\
 &  +  \cos^2 \theta \Big(Q (\theta, \theta) L \phi_{k,l} - L \theta \, Q(\theta,  \phi_{k,l}) \big) L \phi_{k,l} +  \cos^2 \theta \Big( L \phi \, Q(\theta,  \theta_{k,l}) - Q (\theta, \phi) L \theta_{k,l} \Big) L \phi_{k,l} \nnb \\
 & + f  \cos^2 \theta \Big( Q (\phi, \phi) L \theta_{k,l} - L \phi \, Q(\phi,  \theta_{k,l}) \Big) \Lb \theta_{k,l} + f  \cos^2 \theta \Big( L \theta \, Q(\phi,  \phi_{k,l}) - Q (\theta, \phi) L \phi_{k,l} \Big) \Lb \theta_{k,l} \nnb \\
 & + f  \cos^2 \theta \Big(Q (\theta, \theta) L \phi_{k,l} - L \theta \, Q(\theta,  \phi_{k,l}) \big) \Lb \phi_{k,l} + f \cos^2 \theta \Big( L \phi \, Q(\theta,  \theta_{k,l}) - Q (\theta, \phi) L \theta_{k,l} \Big) \Lb \phi_{k,l} \nnb \\
 &  - f  \cos^2 \theta \Big( Q( \phi, \theta_{k,l})^2 -  Q (\phi, \phi)  Q(\theta_{k,l} ,  \theta_{k,l}) \Big)
 - f  \cos^2 \theta  \Big( Q( \theta, \phi_{k,l})^2 -  Q (\theta, \theta)  Q(\phi_{k,l} ,  \phi_{k,l}) \Big) \nnb \\
 & - 2f  \cos^2 \theta  \Big( Q( \theta, \phi)  Q(\theta_{k,l} ,  \phi_{k,l}) - Q(\theta ,  \theta_{k,l}) Q (\phi, \phi_{k,l}) \Big)\nnb \\
 = & \, \cos^2 \theta \Big( Q (\phi, \phi) L \theta_{k,l} - L \phi \, Q(\phi,  \theta_{k,l}) \Big) L \theta_{k,l} + \cos^2 \theta \Big( L \theta \, Q(\phi,  \phi_{k,l}) - Q (\theta, \phi) L \phi_{k,l} \Big) L \theta_{k,l} \nnb \\
 &  +  \cos^2 \theta \Big(Q (\theta, \theta) L \phi_{k,l} - L \theta \, Q(\theta,  \phi_{k,l}) \big) L \phi_{k,l} +  \cos^2 \theta \Big( L \phi \, Q(\theta,  \theta_{k,l}) - Q (\theta, \phi) L \theta_{k,l} \Big) L \phi_{k,l} \nnb \\
 & + f \cos^2 \theta \Big( Q(\phi , \phi) | \slashed{\nabla} \theta_{k, l} |^2 + Q(\theta , \theta) | \slashed{\nabla} \phi_{k, l} |^2 - 2Q(\theta, \phi) \po^i \theta_{k, l} \, \po_i \phi_{k, l} - (\frac{1}{2} \Lb \phi L \theta_{k, l} - \po^i \phi \, \po_i \theta_{k, l} )^2 \nnb \\
 & \qquad - (\frac{1}{2} \Lb \theta L \phi_{k, l} - \po^i \theta \, \po_i \phi_{k, l} )^2  + 2 (-\frac{1}{2} \Lb \phi L \phi_{k, l} + \po^i \phi \, \po_i \phi_{k, l} ) (-\frac{1}{2} \Lb \theta L \theta_{k, l} + \po^i \theta \, \po_i \theta_{k, l} ) \Big) \nnb \\
 & + \frac{1}{4} f \cos^2 \theta \big(L \phi  \Lb \theta_{k, l} - L \theta  \Lb \phi_{k, l} \big)^2
 . \label{P4}
\end{align}

First, we demonstrate that the terms involving two good derivatives acting on $\theta_{k,l}$ or $\phi_{k,l}$ are dominated by the quadratic terms presented in \eqref{quadratic}.

For the first two lines in \eqref{P4}, we only need to deal with the first term, and the rest can be treated in a similar way.
\begin{align}
& \, \big(Q (\phi, \phi) L \theta_{k,l} - L \phi \, Q(\phi,  \theta_{k,l}) \big) L \theta_{k,l}  \nnb \\
= & \, \big( - \frac{1}{2} L \phi \Lb \phi +  \po^i \phi \, \po_i \phi \big)( L \theta_{k,l} )^2 -  L \phi \, \po^i \phi \, \po_i \theta_{k,l}  \, L \theta_{k,l}  + \frac{1}{2} (L \phi)^2 \, L \theta_{k,l} \, \Lb \theta_{k,l} .\label{P4-1}
\end{align}
Utilizing the $L^{\infty}$ estimates from Lemma \ref{Sobolev}, we can handle the first line in \eqref{P4-1} as follows
\begin{align}
 & \, \Big| \big( - \frac{1}{2} L \phi \Lb \phi +  \po^i \phi \, \po_i \phi \big)( L \theta_{k,l} )^2 -  L \phi \, \po^i \phi \, \po_i \theta_{k,l}  \, L \theta_{k,l}\Big| \nnb \\
 \lesssim & \, (\delta |\ub|^{-\frac{5}{2}} M^2 + \delta^{\frac{3}{2}} |\ub|^{-3} M^2) ( L \theta_{k,l} )^2 + \delta^{\frac{3}{4}} |\ub|^{-\frac{3}{2}} M \big( (L \phi )^2 |\slashed{\nabla} \theta_{k,l}|^2 + ( L \theta_{k,l} )^2 \big)  \nnb\\
 \lesssim & \, \delta^{\frac{1}{2}}  M \big( ( L \theta_{k,l} )^2 + (L \phi )^2 |\slashed{\nabla} \theta_{k,l}|^2 \big), \label{P4-11}
\end{align}
where the term $ \delta^{\frac{1}{2}}  M \, ( L \theta_{k,l} )^2$ can be absorbed by the leading term $ ( L \theta_{k,l} )^2$, and the second term
$\delta^{\frac{1}{2}}  M \, (L \phi )^2 |\slashed{\nabla} \theta_{k,l}|^2$ can be absorbed by $\big( (L\theta)^2 + (L\phi)^2 \big) | \slashed{\nabla} \theta_{k, l} |^2  $.

For the first term in the third line of \eqref{P4}, we have
\begin{align}
&  \Big| f \cos^2 \theta \, Q(\phi , \phi) | \slashed{\nabla} \theta_{k, l} |^2 \Big|  \nnb \\
\lesssim & \, \big( (L\theta)^2 + (L\phi)^2 \big) \, \Big| -L \phi \, \Lb \phi +  |\slashed{\nabla} \phi |^2 \Big| \, | \slashed{\nabla} \theta_{k, l} |^2 \nnb \\
\lesssim  & \, \delta  M^2 \big( (L\theta)^2 + (L\phi)^2 \big)  | \slashed{\nabla} \theta_{k, l} |^2, \label{P4-2}
\end{align}
which can also be absorbed by $\big( (L\theta)^2 + (L\phi)^2 \big) | \slashed{\nabla} \theta_{k, l} |^2  $.

Similarly, the terms $ f \cos^2 \theta \, Q(\theta , \theta) | \slashed{\nabla} \phi_{k, l} |^2 $ and $ f \cos^2 \theta \, \big( - 2Q(\theta, \phi) \po^i \theta_{k, l} \, \po_i \phi_{k, l} \big) $ can be controlled by $\big( (L\theta)^2 + (L\phi)^2 \big) | \slashed{\nabla} \phi_{k, l} |^2  $ and $\big( (L\theta)^2 + (L\phi)^2 \big) | \slashed{\nabla} \theta_{k, l} |^2  $ in \eqref{quadratic}.

For the last term in the third line of \eqref{P4}, we have
\begin{align*}
& \, f \cos^2 \theta \,\left(\frac{1}{2} \Lb \phi L \theta_{k, l} - \po^i \phi \, \po_i \theta_{k, l} \right)^2 \nnb \\
 \lesssim & \, \big( (L\theta)^2 + (L\phi)^2 \big) \, |\Lb \phi |^2 |L \theta_{k, l}|^2 + | \slashed{\nabla} \phi |^2 \, \big( (L\theta)^2 + (L\phi)^2 \big) \, | \slashed{\nabla} \theta_{k, l} |^2 \nnb \\
 \lesssim & \, \delta^2  M^4 \, |L \theta_{k, l}|^2  + \delta^{\frac{3}{2}}  M^2 \, \big( (L\theta)^2 + (L\phi)^2 \big) \, | \slashed{\nabla} \theta_{k, l} |^2,
\end{align*}
which can be absorbed by the leading term $ ( L \theta_{k,l} )^2$ and $\big( (L\theta)^2 + (L\phi)^2 \big) | \slashed{\nabla} \theta_{k, l} |^2 $.

The terms in the fourth line of \eqref{P4} are similarly controlled by the quadratic terms in \eqref{quadratic}, following the same calculation as previously discussed.

Now, we address the terms in equation \eqref{P4} that include at least one undesirable derivative ${\Lb }$ of $\theta_{k,l}$ or $\phi_{k,l}$.  Following the same calculations as \eqref{P4-1} and combining the last line of \eqref{P4}, we derive these terms as
\begin{align}
 &  \, \frac{1}{2} \cos^2 \theta \Big( (L \phi)^2 \, L \theta_{k,l} \, \Lb \theta_{k,l} + (L \theta)^2 \, L \phi_{k,l} \, \Lb \phi_{k,l} - L \theta \, L \phi \, L \theta_{k,l} \, \Lb \phi_{k,l} - L \theta \, L \phi \, L \phi_{k,l} \, \Lb \theta_{k,l} \Big) \nnb \\
 & + \frac{1}{4} f \cos^2 \theta \big(L \phi  \Lb \theta_{k, l} - L \theta  \Lb \phi_{k, l} \big)^2  \nnb \\
  = & \, \frac{1}{4} f \cos^2 \theta \big(L \phi  \Lb \theta_{k, l} - L \theta  \Lb \phi_{k, l} \big)^2 + \frac{1}{2} \cos^2 \theta \big(L \phi  \Lb \theta_{k, l} - L \theta  \Lb \phi_{k, l} \big) \big(L \phi  L \theta_{k, l} - L \theta  L \phi_{k, l} \big).  \label{bad1}
\end{align}

For the second term, we apply the Cauchy inequality
\begin{align*}
&  \, \Big| \frac{1}{2} \cos^2 \theta \big( L \phi \, \Lb \theta_{k,l} - L \theta \, \Lb \phi_{k,l} \big) \big( L \phi \, L \theta_{k,l} - L \theta \, L \phi_{k,l} \big) \Big|  \\
\leq & \,  \frac{9}{64}f \big( L \phi \, \Lb \theta_{k,l} - L \theta \, \Lb \phi_{k,l} \big)^2 + \frac{4}{9f}  \big( L \phi \, L \theta_{k,l} - L \theta \, L \phi_{k,l} \big)^2 \\
\leq & \,  \frac{9}{64} f \big( L \phi \, \Lb \theta_{k,l} - L \theta \, \Lb \phi_{k,l} \big)^2 + \frac{8}{9f}  \Big((L\phi)^2 \, (L \theta_{k,l})^2 + (L\theta)^2 \, (L \phi_{k,l})^2 \Big)  \\
\leq & \,  \frac{9}{64}f \big( L \phi \, \Lb \theta_{k,l} - L \theta \, \Lb \phi_{k,l} \big)^2 + \frac{8}{9} \big( (L \theta_{k,l})^2 + (L \phi_{k,l})^2 \big),
\end{align*}
which can be absorbed by the first term in the last line of \eqref{bad1} and the leading quadratic terms of $(L \theta_{k,l})^2$ and $(L \phi_{k,l})^2 $ in \eqref{quadratic}.

Therefore, we arrive at
\begin{align*}
 E^2_{k,l}(\theta, \phi, \tilde{L}) \sim \int_{C_{u}} ( L \theta_{k, l} )^2 + ( L \phi_{k, l} )^2 + \big( (L\theta)^2 + (L\phi)^2 \big)& \Big( | \slashed{\nabla} \theta_{k, l} |^2 + | \slashed{\nabla} \phi_{k, l} |^2 \\
 & + \big( L \phi \Lb \theta_{k, l} - L \theta\Lb \phi_{k, l}\big)^2 \Big).
\end{align*}

For \eqref{Energy12}, by the definition of $\Eb^{2}_{k,l}$ in \eqref{Energy2}, the leadingly quadratic terms corresponding to $\Lb$ and $\slashed{\nabla}$ in the first line of the right-hand side of
\eqref{Energy2} are
\begin{align}
 & \, | \slashed{\nabla} \theta_{k,l} |^2 +  f \,  |\Lb \theta_{k,l}|^2 + \cos^2 \theta \big( | \slashed{\nabla} \phi_{k,l} |^2 + f | \Lb \phi_{k,l}|^2 \big) \nnb  \\
  \sim & \,|\slashed{\nabla} \theta_{k,l}|^2 + |\slashed{\nabla} \phi_{k,l}|^2  \, + \big( (L\theta)^2 + (L\phi)^2 \big) \big( (\Lb \theta_{k,l})^2 + (\Lb \phi_{k,l})^2 \big). \label{quadratic1}
\end{align}
The second line of the right-hand side of \eqref{Energy2} becomes
\begin{align}
& \,- \big( P^{(1)}_2 +  P^{(2)}_2 + P^{(3)}_2 +  P^{(4)}_2 \big)  \nnb \\
= & \, \cos^2 \theta \Big( Q(\phi , \phi) | \slashed{\nabla} \theta_{k, l} |^2 + Q(\theta , \theta) | \slashed{\nabla} \phi_{k, l} |^2 - 2Q(\theta, \phi) \po^i \theta_{k, l} \, \po_i \phi_{k, l} - (\frac{1}{2} L \phi \Lb \theta_{k, l} - \po^i \phi \, \po_i \theta_{k, l} )^2 \nnb \\
 & \qquad - (\frac{1}{2} L \theta \Lb \phi_{k, l} - \po^i \theta \, \po_i \phi_{k, l} )^2  + 2 (-\frac{1}{2} L \theta \Lb \theta_{k, l} + \po^i \theta \, \po_i \theta_{k, l} ) (-\frac{1}{2} L \phi \Lb \phi_{k, l} + \po^i \phi \, \po_i \phi_{k, l} ) \Big) \nnb  \\
 & \, + f \cos^2 \theta \Big( Q (\phi, \phi) \Lb \theta_{k,l} - \Lb \phi \, Q(\phi,  \theta_{k,l}) \Big) \Lb \theta_{k,l} + f \cos^2 \theta \Big( \Lb \theta \, Q(\phi,  \phi_{k,l}) - Q (\theta, \phi) \Lb \phi_{k,l} \Big) \Lb \theta_{k,l} \nnb \\
 & \, + f \cos^2 \theta \Big(Q (\theta, \theta) \Lb \phi_{k,l} - \Lb \theta \, Q(\theta,  \phi_{k,l}) \big) \Lb \phi_{k,l} + f \cos^2 \theta \Big( \Lb \phi \, Q(\theta,  \theta_{k,l}) - Q (\theta, \phi) \Lb \theta_{k,l} \Big) \Lb \phi_{k,l} \nnb \\
 &  \, + \frac{1}{4} \cos^2 \theta \big(\Lb \phi  L \theta_{k, l} - \Lb \theta  L \phi_{k, l} \big)^2 \nnb \\
= & \, \cos^2 \theta \Big( Q(\phi , \phi) | \slashed{\nabla} \theta_{k, l} |^2 + Q(\theta , \theta) | \slashed{\nabla} \phi_{k, l} |^2 - 2Q(\theta, \phi) \po^i \theta_{k, l} \, \po_i \phi_{k, l} - (\po^i \phi \, \po_i \theta_{k, l} )^2 - (\po^i \theta \, \po_i \phi_{k, l} )^2 \nnb \\
 &  \quad +2 (\po^i \theta \, \po_i \theta_{k, l}) ( \po^i \phi \, \po_i \phi_{k, l} ) \Big) \nnb \\
 & \, + \cos^2 \theta \Big( L \phi \, \po^i \phi \, \po_i \theta_{k, l} \, \Lb \theta_{k, l} +  L \theta \, \po^i \theta \, \po_i \phi_{k, l} \, \Lb \phi_{k, l} -  L \phi \, \po^i \theta \, \po_i \theta_{k, l} \, \Lb \phi_{k, l}-  L \theta \, \po^i \phi \, \po_i \phi_{k, l} \, \Lb \theta_{k, l} \nnb \\
 &  \quad + f \big( - \Lb \phi \, \po^i \phi \, \po_i \theta_{k, l} \, \Lb \theta_{k, l} + \Lb \theta \, \po^i \phi \, \po_i \phi_{k, l} \, \Lb \theta_{k, l} - \Lb \theta \, \po^i \theta \, \po_i \phi_{k, l} \, \Lb \phi_{k, l} + \Lb \phi \, \po^i \theta \, \po_i \theta_{k, l} \, \Lb \phi_{k, l} \big) \Big) \nnb \\
 & \, + f \cos^2 \theta \Big( \big( Q (\phi, \phi) + \frac{1}{2} \Lb \phi \, L \phi \big) (\Lb \theta_{k,l})^2 +  \big( -\frac{1}{2} \Lb \theta \, L \phi - Q (\theta, \phi)  \big) \Lb \phi_{k,l} \, \Lb \theta_{k,l} \nnb \\
 & \quad  + \big(Q (\theta, \theta)  + \frac{1}{2} \Lb \theta \,  L \theta  \big) (\Lb \phi_{k,l} )^2 + \big( -\frac{1}{2} \Lb \phi \, L \theta - Q (\theta, \phi) \big) \Lb \theta_{k,l}\, \Lb \phi_{k,l} \Big) \nnb \\
 &  \, - \frac{1}{4} \cos^2 \theta \big(L \phi  \Lb \theta_{k, l} - L \theta  \Lb \phi_{k, l} \big)^2 \nnb \\
 & \, + \frac{1}{2} f \cos^2 \theta \big( \Lb \phi  \Lb \theta_{k, l} -  \Lb \theta  \Lb \phi_{k, l} \big)  \big( \Lb \phi L \theta_{k, l} -  \Lb \theta  L \phi_{k, l} \big)  + \frac{1}{4} \cos^2 \theta \big(\Lb \phi  L \theta_{k, l} - \Lb \theta  L \phi_{k, l} \big)^2 .\label{P3}
\end{align}

By Combining the expansions discussed, we split the terms in equation \eqref{P3} into three distinct groups: (i) those containing at least one $\slashed{\nabla}$ derivative in the first four lines; (ii) those containing two $\Lb$ derivatives; and (iii) those containing at least one $L$
derivative in the last line. The terms belonging to category (i) are subsumed by the dominant quadratic terms in \eqref{quadratic1}, with the proof following a similar argument as presented in the proof of \eqref{Energy11}. Terms in category (ii) can be further split into two subsets. One subset can be managed as follows:
\begin{align*}
& \, f \cos^2 \theta \Big| \big( Q (\phi, \phi) + \frac{1}{2} \Lb \phi \, L \phi \big) (\Lb \theta_{k,l})^2 +  \big( -\frac{1}{2} \Lb \theta \, L \phi - Q (\theta, \phi)  \big) \Lb \phi_{k,l} \, \Lb \theta_{k,l} \nnb \\
 & \qquad \qquad  + \big(Q (\theta, \theta)  + \frac{1}{2} \Lb \theta \,  L \theta  \big) (\Lb \phi_{k,l} )^2 + \big( -\frac{1}{2} \Lb \phi \, L \theta - Q (\theta, \phi) \big) \Lb \theta_{k,l}\, \Lb \phi_{k,l} \Big|  \nnb \\
 \lesssim & \, \delta M^2 \big( (L\theta)^2 + (L\phi)^2 \big) \big( (\Lb \theta_{k,l})^2 + (\Lb \phi_{k,l})^2 \big),
\end{align*}
and the other becomes
\begin{align*}
& \, \Big|- \frac{1}{4} \cos^2 \theta \big(L \phi  \Lb \theta_{k, l} - L \theta  \Lb \phi_{k, l} \big)^2 \Big| \\
\leq & \, \frac{1}{4}  \Big| (L \phi)^2 (\Lb\theta_{k,l})^2 \, + (L \theta)^2 (\Lb\phi_{k,l})^2 - 2 L \phi L \theta \Lb\theta_{k,l} \phi_{k,l} \Big| \\
 \leq &\,\frac{1}{4} \big( (L \phi)^2 (\Lb\theta_{k,l})^2 + (L \theta)^2 (\Lb\phi_{k,l})^2 + (L \theta)^2 (\Lb\theta_{k,l})^2 + (L \phi)^2 (\Lb\phi_{k,l})^2 \big)\\
 \leq &\,\frac{1}{4} \big( (L\theta)^2 + (L\phi)^2 \big) \big( (\Lb \theta_{k,l})^2 + (\Lb \phi_{k,l})^2 \big).
\end{align*}
Both of these groups can be controlled by the leadingly quadratic terms in \eqref{quadratic1}.

We now turn our attention to the terms classified as type (iii).
The first term of last line in \eqref{P3} can be dealt as
\begin{align}
& \Big|\frac{1}{2} f \cos^2 \theta \, (\Lb \phi \, \Lb \theta_{k,l} - \Lb \theta \, \Lb \phi_{k,l} ) (\Lb \phi \, L \theta_{k,l} - \Lb \theta \, L \phi_{k,l}) \Big| \nnb  \\
 \leq & \, \frac{9}{64} \big(\Lb \phi \, L \theta_{k,l} - \Lb \theta \, L \phi_{k,l} \big)^2 + \frac{4}{9} f^2 \big(\Lb \phi \, \Lb \theta_{k,l} - \Lb \theta \, \Lb \phi_{k,l} \big)^2  \nnb \\
\leq & \, \frac{9}{64} \big(\Lb \phi \, L \theta_{k,l} - \Lb \theta \, L \phi_{k,l} \big)^2 + \frac{8}{9} C \delta^2 M^4 \big( (L\theta)^2 + (L\phi)^2 \big) \, \big( (\Lb \theta_{k,l} )^2 + (\Lb \phi_{k,l})^2 \big),\label{P3-3}
\end{align}
here we have used the decay estimations for $f$ and $\Lb \phi $, $\Lb \theta $ in Lemma \ref{Sobolev} for the last line,
which can be controlled by the last term of \eqref{P3} and the leading quadratic terms in \eqref{quadratic1}. Noting that $C$ in the last line of \eqref{P3-3} is a constant.
\end{proof}

\subsubsection{Energy scheme for $\tilde{\Lb}$}
Now we derive the energy of $\theta_{k,l}$ and $\phi_{k,l} $ corresponding to vector field $\tilde{\Lb} = \Lb + \big((\Lb\theta)^2 + (\Lb\phi)^2 \big) L$.

Define
\be\label{fb}
\fb = (\Lb\theta)^2 + (\Lb\phi)^2,
\ee
and
\begin{align}
  \Pb^{(1)}_1 = & \, \cos^2 \theta \Big(L \phi \, Q(\phi,  \theta_{k,l}) -  Q (\phi, \phi) L \theta_{k,l} \Big) \tilde{\Lb} \theta_{k,l} +  \cos^2 \theta  \Big( Q( \phi, \theta_{k,l})^2 -  Q (\phi, \phi)  Q(\theta_{k,l} ,  \theta_{k,l}) \Big) , \nnb \\
   \Pb^{(1)}_2 = & \, \cos^2 \theta \Big(\Lb \phi \, Q (\phi, \theta_{k,l}) -  Q (\phi, \phi) \Lb \theta_{k,l} \Big) \tilde{\Lb} \theta_{k,l} + \fb \cos^2 \theta  \Big( Q( \phi, \theta_{k,l})^2 -  Q (\phi, \phi)  Q(\theta_{k,l} ,  \theta_{k,l}) \Big), \nnb
\end{align}
\begin{align}
   \Pb^{(2)}_1 = & \,   \cos^2 \theta \Big(Q (\theta, \phi) L \phi_{k,l} - L \theta \, Q(\phi,  \phi_{k,l}) \Big) \tilde{\Lb} \theta_{k,l}  \nnb \\
   & \qquad + 2 \cos^2 \theta  \Big( Q( \theta, \phi)  Q(\theta_{k,l} ,  \phi_{k,l}) - Q(\theta ,  \theta_{k,l}) Q (\phi, \phi_{k,l}) \Big) , \nnb \\
  \Pb^{(2)}_2 = & \,  \cos^2 \theta \Big(Q (\theta, \phi) \Lb \phi_{k,l} - \Lb \theta \, Q(\phi,  \phi_{k,l}) \Big) \tilde{\Lb} \theta_{k,l}  \nnb \\
  & \,  \qquad + 2 f \cos^2 \theta  \Big( Q( \theta, \phi)  Q(\theta_{k,l} ,  \phi_{k,l}) - Q(\theta ,  \theta_{k,l}) Q (\phi, \phi_{k,l}) \Big), \nnb
\end{align}
\begin{align}
  \Pb^{(3)}_1 =  & \, \cos^2 \theta \Big(L \theta \, Q(\theta,  \phi_{k,l}) -  Q (\theta, \theta) L \phi_{k,l} \Big) \tilde{\Lb} \phi_{k,l} +  \cos^2 \theta  \Big( Q( \theta, \phi_{k,l})^2 -  Q (\theta, \theta)  Q(\phi_{k,l} ,  \phi_{k,l}) \Big) , \nnb \\
   \Pb^{(3)}_2 =  & \, \cos^2 \theta \Big(\Lb \theta \, Q (\theta, \phi_{k,l}) -  Q (\theta, \theta) \Lb \phi_{k,l} \Big) \tilde{\Lb} \phi_{k,l} + \fb \cos^2 \theta  \Big( Q( \theta, \phi_{k,l})^2 -  Q (\theta, \theta)  Q(\phi_{k,l} ,  \phi_{k,l}) \Big), \nnb
\end{align}
\begin{align}
  \Pb^{(4)}_1 = & \,  \cos^2 \theta \Big(Q (\theta, \phi) L \theta_{k,l} - L \phi \, Q(\theta,  \theta_{k,l}) \Big) \tilde{\Lb} \phi_{k,l},  \nnb \\
   \Pb^{(4)}_2 = & \,  \cos^2 \theta \Big(Q (\theta, \phi) \Lb \theta_{k,l} - \Lb \phi \, Q(\theta,  \theta_{k,l}) \Big) \tilde{\Lb} \phi_{k,l}.
\end{align}

Similar derivation as above subsection corresponding to vector field $\tilde{L}$ gives
\begin{align}
&\iint_{D_{u, \ub}} \big(\Box \theta_{k,l} - g_1^{\alpha \beta} \p_\alpha  \p_\beta \theta_{k,l}  -h_1^{\alpha \beta} \p_\alpha \p_\beta \phi_{k,l} \big) \tilde{\Lb} \theta_{k,l} \nnb \\
& \qquad \qquad + \big(\cos^2 \theta \, \Box \phi_{k,l} - g_2^{\alpha \beta} \p_\alpha  \p_\beta \phi_{k,l}  -h_2^{\alpha \beta} \p_\alpha \p_\beta \theta_{k,l} \big) \tilde{\Lb} \phi_{k,l} \nnb \\
= & \, E^2_{k,l}(\theta, \phi, \tilde{\Lb}) +  \Eb^{2}_{k,l}(\theta, \phi, \tilde{\Lb}) -  E^{2}_{k,l, t=1}(\theta, \phi, \tilde{\Lb})  \nnb \\
- & \iint_{D_{u, \ub}} K^{\Lb} [\theta_{k,l}]+ K^{\fb L} [\theta_{k,l}]+ K^{\cos^2 \theta \Lb} [\phi_{k,l}] + K^{\fb \cos^2 \theta L} [\phi_{k,l}] - \Qb^{(1)} - \Qb^{(2)} - \Qb^{(3)} - \Qb^{(4)},\label{energy2}
\end{align}
where
\begin{align}
 E^2_{k,l}(\theta, \phi, \tilde{\Lb})
 = & \int_{C_{u}} |\slashed{\nabla} \theta_{k,l} |^2 + \fb \,|L \theta_{k,l}|^2+ \cos^2 \theta \big( | \slashed{\nabla} \phi_{k,l} |^2 + \fb \,|L \phi_{k,l}|^2 \big)  \nnb \\
 & \qquad - \big( \Pb^{(1)}_1 +  \Pb^{(2)}_1 + \Pb^{(3)}_1 +  \Pb^{(4)}_1 \big),\label{Energy3}
\end{align}
\begin{align}
  \Eb^{2}_{k,l}(\theta, \phi, \tilde{\Lb})
   = & \int_{\Cb_{\ub}} |\Lb \theta_{k,l}|^2 +  \fb \,| \slashed{\nabla} \theta_{k,l} |^2 + \cos^2 \theta \big( | \Lb \phi_{k,l}|^2 + \fb | \slashed{\nabla} \phi_{k,l} |^2 \big)  \nnb \\
   & \qquad -  \big( \Pb^{(1)}_2 + \Pb^{(2)}_2 + \Pb^{(3)}_2 +  \Pb^{(4)}_2 \big),\label{Energy4}
\end{align}
\begin{align*}
  E^{2}_{k,l, t=1}(\theta, \phi, \tilde{\Lb}) = & \frac{1}{2} \int_{\Sigma_1} \Big(  | \slashed{\nabla} \theta_{k,l}|^2 + |\Lb \theta_{k,l} |^2 +
  \fb \, \big( |L \theta_{k,l} |^2 + | \slashed{\nabla} \theta_{k,l}|^2 \big) \\
   & \qquad + \cos^2 \theta \big( | \slashed{\nabla} \phi_{k,l}|^2 + |\Lb \phi_{k,l} |^2 + \fb \, (|L \phi_{k,l} |^2 + | \slashed{\nabla} \phi_{k,l}|^2) \big) \Big) \\
   & \qquad -\big( \Pb^{(1)}_1 +  \Pb^{(2)}_1 + \Pb^{(3)}_1 +  \Pb^{(4)}_1 + \Pb^{(1)}_2 + \Pb^{(2)}_2 + \Pb^{(3)}_2 +  \Pb^{(4)}_2 \big),
\end{align*}
\be\label{K1}
 K^{\Lb} [\theta_{k,l}] = -\frac{1}{r} L \theta_{k,l} \Lb \theta_{k,l},
\ee
\be\label{K2}
K^{\fb L} [\theta_{k,l}] = \frac{1}{r} \fb \, L \theta_{k,l}\, \Lb \theta_{k,l} + Q( \fb, \theta_{k,l})  L \theta_{k,l} -\frac{1}{2}  L \fb \,
 |\nabla \theta_{k,l} |^2,
\ee
\begin{align}
K^{\cos^2 \theta \Lb} [\phi_{k,l}] = & \, - \frac{1}{r} \cos^2 \theta L \phi_{k,l}\, \Lb \phi_{k,l} + Q ( \cos^2 \theta, \phi_{k,l}) \, \Lb \phi_{k,l} - \frac{1}{2}  \Lb (\cos^2 \theta) \, |\nabla \phi_{k,l} |^2 ,\label{K3}
\end{align}
\begin{align}
 K^{\fb \cos^2 \theta L} [\phi_{k,l}] = &\, \frac{1}{r} (\fb \cos^2 \theta) L \phi_{k,l}\, \Lb \phi_{k,l} + Q( \fb  \cos^2 \theta , \phi_{k,l})  L \phi_{k,l} -\frac{1}{2}  L (\fb  \cos^2 \theta)
 |\nabla \phi_{k,l} |^2,\label{K4}
\end{align}
\begin{align}\label{QB1}
\Qb^{(1)} = & \frac{1}{2} \Lb g_1^{\alpha\beta} \, \p_\alpha \theta_{k,l} \, \p_\beta \theta_{k,l}  + \frac{1}{2} L \big( \fb \, g_1^{\alpha\beta} \big) \p_\alpha \theta_{k,l} \, \p_\beta \theta_{k,l} - \p_\alpha g_1^{\alpha\beta} \, \p_\beta \theta_{k,l} \,\tilde{\Lb} \theta_{k,l} -  g_1^{\alpha\beta} \p_\alpha \fb \, \p_\beta \theta_{k,l} \, L \theta_{k,l} \nnb \\
   & + \frac{1}{r} (1-\fb) g_1^{i \beta} \, \po_i \theta_{k,l} \,\p_\beta \theta_{k,l} ,
\end{align}
\begin{align}\label{QB2}
\Qb^{(2)}  = &  \Lb h_1^{\alpha \beta} \, \p_\alpha \theta_{k,l} \, \p_\beta \phi_{k,l} +  L \big(\fb \, h_1^{\alpha \beta} \, \big) \p_\alpha \theta_{k,l} \, \p_\beta \phi_{k,l} - \p_\alpha h_1^{\alpha \beta} \, \p_\beta \phi_{k,l} \,\tilde{\Lb} \theta_{k,l} - h_1^{\alpha \beta} \p_\alpha \fb \, \p_\beta \phi_{k,l} \, L \theta_{k,l} \nnb \\
  & \, + \frac{1}{r} (1-\fb) h_1^{i \beta} \, \po_i \theta_{k,l} \,\p_\beta \phi_{k,l} ,
\end{align}
\begin{align}\label{QB3}
\Qb^{(3)} = & \frac{1}{2} \Lb g_2^{\alpha\beta} \, \p_\alpha \phi_{k,l} \, \p_\beta \phi_{k,l}  + \frac{1}{2} L \big( \fb \, g_2^{\alpha\beta} \big) \p_\alpha \phi_{k,l} \, \p_\beta \phi_{k,l} - \p_\alpha g_2^{\alpha\beta} \, \p_\beta \phi_{k,l} \,\tilde{\Lb} \phi_{k,l} -  g_2^{\alpha\beta} \p_\alpha \fb \, \p_\beta \phi_{k,l} \, L \phi_{k,l} \nnb  \\
   & + \frac{1}{r} (1-\fb) g_2^{i \beta} \, \po_i \phi_{k,l} \,\p_\beta \phi_{k,l} ,
\end{align}
\begin{align}\label{QB4}
 \Qb^{(4)} =  - \p_\alpha h_2^{\alpha \beta} \, \p_\beta \theta_{k,l} \,\tilde{\Lb} \phi_{k,l} - h_2^{\alpha \beta} \p_\alpha \fb \, \p_\beta \theta_{k,l} \, L \phi_{k,l} + \frac{1}{r} (1-\fb) h_2^{i \beta} \, \po_i \phi_{k,l} \,\p_\beta \theta_{k,l} .
\end{align}

Similar to Lemma \ref{lem-en1}, we can conclude that
\begin{lemma}\label{lem-en2}
With the bootstrap assumption \eqref{assump}, we have
\begin{align}
 E^2_{k,l}(\theta, \phi, \tilde{\Lb}) \sim  \int_{C_{u}} |\slashed{\nabla} \theta_{k,l}|^2 + |\slashed{\nabla} \phi_{k,l}|^2 & \, + \big( (\Lb\theta)^2 + (\Lb\phi)^2 \big) \big( (L \theta_{k,l})^2 + (L \phi_{k,l})^2 \big)  \nnb \\
 & + (L \phi \, \Lb \theta_{k,l} - L \theta \, \Lb \phi_{k,l} )^2, \label{Energy21}
\end{align}
and
\begin{align}
 \Eb^{2}_{k,l}(\theta, \phi, \tilde{\Lb}) \sim \int_{\Cb_{\ub}} ( \Lb \theta_{k, l} )^2 + ( \Lb \phi_{k, l} )^2& + \big( (\Lb\theta)^2 + (\Lb\phi)^2 \big) \Big( | \slashed{\nabla} \theta_{k, l} |^2 + | \slashed{\nabla} \phi_{k, l} |^2  \nnb \\
 & + \big( \Lb \phi L \theta_{k, l} - \Lb \theta L \phi_{k, l}\big)^2 \Big). \label{Energy22}
\end{align}
\end{lemma}


\subsection{Energy estimates}

There exists a hierarchical structure for the energy estimates, the computation of $E_{\leq N}$ necessitates prior knowledge of the estimates for $\Eb_{\leq N}$. Consequently, the estimation for $\Eb_{\leq N}$ should be conducted before proceeding with $E_{\leq N}$.

For simplicity, we use $\Vert \cdot \Vert_{L^2(C_{u})} $ to denote $\Vert \cdot \Vert_{L^2 (C_{u}^{\ub})}$ and $\Vert \cdot \Vert_{L^2(\Cb_{\ub})} $ to denote $\Vert \cdot \Vert_{L^2(\Cb_{\ub}^u)}$ in the following estimates.

\subsubsection{Estimates for $\Eb_{\leq N}$}

Recall the energy identity \eqref{energy2}. By Lemma \ref{lem-en2}, \eqref{Eb} and \eqref{sum}, $\delta \Eb_{\leq N}^2  ( u, \ub)$ can be bounded by the summation of the energy terms of the above identity for $0 \leq l \leq k \leq N$, then we need to deal with the following energy inequality
\begin{align} \label{First-E}
\delta \Eb^2_{\leq N} ( u, \ub) \lesssim & \, \delta I^2_N(\theta_0, \theta_1, \phi_0, \phi_1) +  \sum_{0\leq l \leq k \leq N} \iint_{D_{u, \ub}} \Big( \Big| \Qb^{(1)} + \Qb^{(2)} + \Qb^{(3)} + \Qb^{(4)} \Big|  \nnb \\
& + \Big|  K^{\Lb} [\theta_{k,l}]+ K^{\fb L} [\theta_{k,l}] + K^{\cos^2 \theta \Lb} [\phi_{k,l}] + K^{\fb \cos^2 \theta L} [\phi_{k,l}]  \Big|  \nnb\\
&   + \Big| \big(\Box \theta_{k,l} - g_1^{\alpha \beta} \p_\alpha  \p_\beta \theta_{k,l}  -h_1^{\alpha \beta} \p_\alpha \p_\beta \phi_{k,l} \big) \tilde{\Lb} \theta_{k,l} \Big| \nnb  \\
& + \Big| \big(\cos^2 \theta \, \Box \phi_{k,l} - g_2^{\alpha \beta} \p_\alpha  \p_\beta \phi_{k,l}  -h_2^{\alpha \beta} \p_\alpha \p_\beta \theta_{k,l} \big) \tilde{\Lb} \phi_{k,l} \Big|\Big),
\end{align}
where $\Qb^{(1)}$, $\Qb^{(2)}$, $\Qb^{(3)}$, $\Qb^{(4)}$ are defined in \eqref{QB1} - \eqref{QB4} and $K^{\Lb} [\theta_{k,l}]$, $K^{\fb L} [\theta_{k,l}]$, $K^{\cos^2 \theta \Lb} [\phi_{k,l}]$, $K^{\fb \cos^2 \theta L} [\phi_{k,l}]$ are defined in \eqref{K1} - \eqref{K4}.

We begin by addressing the terms associated with $\Qb$. In order to obtain
the decay estimates for the derivatives of the functions
$g_1, g_2, h_1$ and $h_2$, we introduce the following lemma which is originated from \cite{Lindblad}.
\begin{lemma}\label{symmetric}
  If $k^{\alpha \beta}$ is a symmetric tensor, then
  \begin{align*}
    k^{\alpha \beta} \p_\alpha \p_\beta =  & \, k_{\Lb \, \Lb} L^2 + 2 k_{L \Lb} L \Lb +  k_{L L } \Lb^2 + k^{ij} \po_i \po_j + k_{\Lb}^{\,\,\, i}  L  \po_i + k_{L}^{\,\,\, i}  \Lb \po_i \\
     & \, + \frac{1}{2r} \overline{tr} k (L - \Lb) + \frac{1}{r} k^{ij} \omega_j \po_i,
  \end{align*}
  where
  \begin{align*}
   & k_{\Lb \, \Lb} = \frac{1}{4} k^{00} + \frac{1}{2} k^{0i}  \omega_i + \frac{1}{4} k^{ij}  \omega_i  \omega_j,  \quad k_{L \Lb} = \frac{1}{4} ( k^{00} -  k^{ij}  \omega_i  \omega_j),  \\
   & k_{L L } = \frac{1}{4} k^{00} - \frac{1}{2} k^{0i}  \omega_i + \frac{1}{4} k^{ij}  \omega_i  \omega_j, \quad k_{\Lb}^{\,\,\, i} = k^{0i} + k^{ij} \omega_j,  \\
   & k_{L}^{\,\,\, i} = k^{0i} - k^{ij} \omega_j, \quad \overline{tr} k = \deo_{ij} k^{ij}.
  \end{align*}
\end{lemma}

The same deduction with Lemma \ref{symmetric} yields the following lemma.
\begin{lemma} \label{pp}
  For any smooth functions $\xi$ and $\chi $,
  \begin{align}\label{p1}
  \p^\alpha \xi \, \p_\alpha \p^\beta \chi  \, \p_\beta = & \, \frac{1}{4} \big( \Lb \xi  \, L\Lb \chi -2 \po^i \xi \, \Lb \po_i \chi + L \xi  \, \Lb^2 \chi \Big) L + \frac{1}{4} \big( L \xi  \, L\Lb \chi -2 \po^i \xi \, L \po_i \chi + \Lb \xi  \, L^2 \chi \Big) \Lb \nnb \\
  & \, + \Big( \po^j \xi \, \po_j \po^i \chi - \frac{1}{2} L \xi \, \Lb \po^i \chi - \frac{1}{2} \Lb \xi \, L \po^i \chi \Big) \, \po_i + \frac{1}{2r} \po^i \xi (L-\Lb)\chi \, \po_i.
  \end{align}
\end{lemma}

We now turn to the estimate of $\Qb$. Recall that the definition of $\Qb^{(1)}$ in \eqref{QB1}.
According to \eqref{metric}, the first term in \eqref{QB1} becomes
\begin{align}
 & \, \frac{1}{2} \Lb g_1^{\alpha\beta} \, \p_\alpha \theta_{k,l} \, \p_\beta \theta_{k,l}\nnb \\
 = & \, \frac{1}{2} \Lb \Big( \cos^2 \theta \big(\p^\alpha \phi \, \p^\beta \phi - g^{\alpha \beta} Q(\phi, \phi)\big) \Big)\p_\alpha \theta_{k,l} \, \p_\beta \theta_{k,l} \nnb \\
 = & \, -  \sin \theta \cos \theta \Lb \theta \big(\p^\alpha \phi \, \p^\beta \phi - g^{\alpha \beta} Q(\phi, \phi)\big) \p_\alpha \theta_{k,l} \, \p_\beta \theta_{k,l} \nnb \\
  & \,  + \cos^2 \theta \, \Big( \Lb \p^\alpha \phi \, \p^\beta \phi -  \frac{1}{2} g^{\alpha \beta} \Lb Q(\phi, \phi) \Big) \p_\alpha \theta_{k,l} \, \p_\beta \theta_{k,l} \nnb \\
 = & \, -  \sin \theta \cos \theta \Lb \theta \Big( \big( Q( \phi, \theta_{k,l} ) \big)^2 -  Q(\phi, \phi) Q( \theta_{k,l}, \theta_{k,l} ) \Big) \nnb \\
 & \,  + \cos^2 \theta \, \Big( \big(- \frac{1}{2}  \Lb^2 \phi \, L \theta_{k,l} - \frac{1}{2} \Lb L \phi \, \Lb \theta_{k,l} + \Lb \po^i \phi \, \po_i \theta_{k,l} \big)  Q( \phi, \theta_{k,l} ) -  \frac{1}{2} \Lb Q(\phi, \phi)  Q( \theta_{k,l}, \theta_{k,l} ) \Big). \label{Qb11}
\end{align}
The worst decay terms in $\big( Q( \phi, \theta_{k,l} ) \big)^2$ and $ Q(\phi, \phi) Q( \theta_{k,l}, \theta_{k,l} ) $ are $\Lb \phi \,  L \phi \, L \theta_{k,l} \, \Lb \theta_{k,l}$ and $\frac{1}{4}  (\Lb \phi \, L \theta_{k,l})^2 $. According to Lemma \ref{Sobolev}, we can control them by
\begin{align}
 & \iint_{D_{u, \ub}} \Big| \big( Q( \phi, \theta_{k,l} ) \big)^2 -  Q(\phi, \phi) Q( \theta_{k,l}, \theta_{k,l} )  \Big| \nnb \\
\lesssim & \iint_{D_{u, \ub}} \frac{1}{4}  (\Lb \phi \, L \theta_{k,l})^2 + \Big| \Lb \phi \,  L \phi \, L \theta_{k,l} \, \Lb \theta_{k,l} \Big| \nnb \\
\lesssim & \, \iint_{D_{u, \ub}} |\ub|^{-2} M^2 (L \theta_{k,l})^2 + \delta |\ub|^{-\frac{5}{2}} M^2 |L \theta_{k,l} \, \Lb \theta_{k,l}| \nnb \\
\lesssim & \, M^2 \int_0^{u} \Vert L \theta_{k,l} \Vert^2_{L^2(C_{u'})} \di u' +  \delta M^2 \Big( \int_0^{u} \Vert L \theta_{k,l} \Vert^2_{L^2(C_{u'})} \di u'\Big)^{\frac{1}{2}} \Big( \int_{\ub_0}^{\ub} |\ub'|^{-5}\Vert \Lb \theta_{k,l} \Vert^2_{L^2(\Cb_{\ub'})} \di \ub'\Big)^{\frac{1}{2}} \nnb \\
\lesssim & \, M^2 \, (\delta M)^2 \, \delta
+ \delta M^2 \, \big( (\delta M)^2 \, \delta \big)^{\frac{1}{2}} \Big( \delta M^2 \, \big| \, |\ub|^{-4} - |\ub_0|^{-4} \big| \Big)^{\frac{1}{2}}
\lesssim \delta^3 M^4, \label{Qb111}
\end{align}
here we have used \eqref{E} -- \eqref{sum} and the bootstrap assumption \eqref{assump}.
Noting that $\sin \theta \cos \theta \Lb \theta$ provides an extra decay $\delta |\ub|^{-2} M^2$, the first term in \eqref{Qb11} can be controlled by $\delta^4 M^6$.
Similarly, according to Lemma \ref{decay}, in the second term of \eqref{Qb11}, we estimate the worst decay terms as
\begin{align}
&  \, \iint_{D_{u, \ub}} \Big| \Lb^2 \phi \,\Lb \phi  \, (L \theta_{k,l} )^2 \Big| + \Big|  L \phi \Lb^2 \phi \, L \theta_{k,l} \Lb \theta_{k,l} \Big|
 \nnb \\
 \lesssim & \, \iint_{D_{u, \ub}} \delta^{-1} |\ub|^{-2} M^2 \, (L \theta_{k,l} )^2 + |\ub|^{-\frac{5}{2} } M^2 \,\big| L \theta_{k,l} \Lb \theta_{k,l} \big| \nnb \\
 \lesssim & \, \delta^{-1}  M^2 \int_0^{u} \Vert L \theta_{k,l} \Vert^2_{L^2(C_{u'})} \di u'
  +  M^2 \Big( \int_0^{u} \Vert L \theta_{k,l} \Vert^2_{L^2(C_{u'})} \di u'\Big)^{\frac{1}{2}} \Big( \int_{\ub_0}^{\ub} |\ub'|^{-5}\Vert \Lb \theta_{k,l} \Vert^2_{L^2(\Cb_{\ub'})} \di \ub'\Big)^{\frac{1}{2}} \nnb\\
 \lesssim & \, \delta^{-1}  M^2 \, \delta^2  M^2 \, \delta +  M^2 ( \delta^2  M^2 \, \delta)^{\frac{1}{2}}  ( \delta M^2 )^{\frac{1}{2}}
 \lesssim  \delta^2 M^4.
\end{align}

 Thus, the term \eqref{Qb11} can be controlled by
\be\label{Qb11'}
 \iint_{D_{u, \ub}} \Big| \frac{1}{2} \Lb g_1^{\alpha\beta} \, \p_\alpha \theta_{k,l} \, \p_\beta \theta_{k,l} \Big| \lesssim \delta^2 M^4.
\ee

The second term in \eqref{QB1} can be expanded as
\begin{align}\label{Qb12}
    & \, \frac{1}{2} L \big( \fb \, g_1^{\alpha\beta} \big) \p_\alpha \theta_{k,l} \, \p_\beta \theta_{k,l} \nnb \\
  = & \,  \frac{1}{2} L  \fb \, g_1^{\alpha\beta} \, \p_\alpha \theta_{k,l} \, \p_\beta \theta_{k,l} + \frac{1}{2} \fb  L  g_1^{\alpha\beta} \p_\alpha \theta_{k,l} \, \p_\beta \theta_{k,l}  \nnb \\
  = & \,  \frac{1}{2} \cos^2 \theta \, L  \fb \, \Big( \big( Q( \phi, \theta_{k,l} ) \big)^2 -  Q(\phi, \phi) Q( \theta_{k,l}, \theta_{k,l} ) \Big) + \frac{1}{2} \fb  L  g_1^{\alpha\beta} \p_\alpha \theta_{k,l} \, \p_\beta \theta_{k,l} .
\end{align}

According to \eqref{Qb111}, the first term in \eqref{Qb12} can be controlled by $\delta^3 M^6$ as $| L  \fb| = |\Lb \theta \, L \Lb \theta + \Lb \phi \, L \Lb \phi | \lesssim |u|^{-3} M^2$. For the second term in \eqref{Qb12}, the derivative $L g_1^{\alpha\beta}$ provides one more $\delta$ than $\Lb g_1^{\alpha\beta}$ for the worst terms and there is also an additional decay $|u|^{-2} M^2$ in $ \fb $, then according to \eqref{Qb11'},
\bes
\iint_{D_{u, \ub}} \Big| \frac{1}{2} \fb  L  g_1^{\alpha\beta} \p_\alpha \theta_{k,l} \, \p_\beta \theta_{k,l}\Big| \lesssim \delta^3 M^6.
\ees
Thus,  the term \eqref{Qb12} can be controlled by
\bes
\iint_{D_{u, \ub}} \Big| \frac{1}{2} L \big( \fb \, g_1^{\alpha\beta} \big) \p_\alpha \theta_{k,l} \, \p_\beta \theta_{k,l} \Big| \lesssim \delta^3 M^6.
\ees

For the third term in \eqref{QB1}, we remind that the term $\fb L \theta_{k,l}$ in  $\tilde{\Lb} \theta_{k,l}$  can be controlled by $\Lb \theta_{k,l}$, as $\Vert \fb L \theta_{k,l} \Vert_{L^2(D_{u, \ub})} \lesssim \delta^{\frac{3}{2}} M^3 $ while  $\Vert \Lb \theta_{k,l} \Vert_{L^2(D_{u, \ub})} \lesssim \delta^{\frac{1}{2}} M$, and there is also an additional factor $|u|^{-2}$  in  $\fb L \theta_{k,l}$ which will provide us more decay. We only need to deal with $\p_\alpha g_1^{\alpha\beta} \, \p_\beta \theta_{k,l} \,\Lb \theta_{k,l}$, which is
\begin{align}\label{Qb13}
   & \, \p_\alpha g_1^{\alpha\beta} \, \p_\beta \theta_{k,l} \,\Lb \theta_{k,l} \nnb \\
  = & \, \p_\alpha \cos^2 \theta \big(\p^\alpha \phi \, \p^\beta \phi - g^{\alpha \beta} Q(\phi, \phi)\big) \, \p_\beta \theta_{k,l} \, \Lb \theta_{k,l} \nnb \\
   & + \cos^2 \theta  \Big( \p_\alpha \p^\alpha \phi \, \p^\beta \phi + \p^\alpha \phi \,  \p_\alpha \p^\beta \phi  - g^{\alpha \beta} \p_\alpha Q(\phi, \phi) \Big) \, \p_\beta \theta_{k,l} \, \Lb \theta_{k,l}  \nnb \\
  = & \, -\sin 2\theta \Big( Q(\theta, \phi) Q(\phi, \theta_{k,l} ) - Q(\phi, \phi) Q(\theta, \theta_{k,l} ) \Big) \Lb \theta_{k,l} \nnb \\
   & + \cos^2 \theta \Big( \p_\alpha \p^\alpha \phi \, Q(\phi, \theta_{k,l} ) + \p^\alpha \phi \,  \p_\alpha \p^\beta \phi \, \p_\beta \theta_{k,l} - g^{\alpha \beta} \p_\alpha Q(\phi, \phi) \, \p_\beta \theta_{k,l} \Big) \Lb \theta_{k,l}.
\end{align}
The first term of \eqref{Qb13} has one of the worst terms as
\begin{align*}
   \iint_{D_{u, \ub}} \Big| \sin 2\theta  \, L \theta \, \Lb \phi \, \Lb \phi \, L \theta_{k,l} \, \Lb \theta_{k,l} \Big|
  \lesssim  \iint_{D_{u, \ub}} \delta^2 |u|^{-\frac{9}{2}} M^4 \Big|\, L \theta_{k,l} \, \Lb \theta_{k,l} \Big|
  \lesssim  \delta^4 M^6.
\end{align*}
For the first term in the last line of \eqref{Qb13}, we have
\be\label{laplace}
| \p_\alpha \p^\alpha \phi | = \Big|-L \Lb \phi +\po^i \po_i \phi +\frac{1}{r}(L-\Lb) \phi \Big| \lesssim |u|^{-2} M ,
\ee
 then
\begin{align*}
    & \, \iint_{D_{u, \ub}} \Big| \p_\alpha \p^\alpha \phi \, Q(\phi, \theta_{k,l} ) \Lb \theta_{k,l} \Big| \\
   \lesssim & \, \iint_{D_{u, \ub}} \Big| |u|^{-2} M \, \Lb \theta \,  L \theta_{k,l} \Lb \theta_{k,l} \Big|
    \lesssim \iint_{D_{u, \ub}} |u|^{-3} M^2 \,  \Big| L \theta_{k,l} \Lb \theta_{k,l} \Big|  \lesssim  \delta^2 M^4.
\end{align*}
According to Lemma \ref{pp}, the second term in the last line of \eqref{Qb13} has the worst decay term as
\begin{align*}
   \iint_{D_{u, \ub}} \Big| L \phi \, \Lb^2 \phi L \theta_{k,l} \,  \Lb \theta_{k,l} \Big| \lesssim \iint_{D_{u, \ub}} |u|^{-\frac{5}{2}} M^2 \,  \Big|  L \theta_{k,l} \,  \Lb \theta_{k,l} \Big| \lesssim \delta^2 M^4.
\end{align*}
Similarly, the third term in the last line of \eqref{Qb13} also has the worst term $L \phi \, \Lb^2 \phi \, L \theta_{k,l}  \Lb \theta_{k,l}$, which can be controlled by $\delta^2 M^4$.
Thus, the last line of \eqref{Qb13} can be controlled by $\delta^2 M^4$ and the third term in \eqref{QB1} also can be controlled by $\delta^2 M^4$.

The fourth term in \eqref{QB1} can be controlled by the worst term as
\begin{align*}
   \iint_{D_{u, \ub}} \Big| g_1^{\alpha\beta} \p_\alpha \fb \, \p_\beta \theta_{k,l} \, L \theta_{k,l} \Big|
  \lesssim   \iint_{D_{u, \ub}} \Big| L \phi \, \Lb \phi \cdot  \Lb \phi \, \Lb^2 \phi \cdot ( L \theta_{k,l} )^2 \Big|
  \lesssim \delta^3 M^6.
\end{align*}

For the last term in \eqref{QB1}, we have
\begin{align*}
   & \, g_1^{i \beta} \, \po_i \theta_{k,l} \,\p_\beta \theta_{k,l} \\
  = & \, \cos^2 \theta  \big( \p^i \phi \, \p^\beta \phi - g^{i\beta} \,  Q( \phi, \phi) \big) \,  \po_i \theta_{k,l} \,\p_\beta \theta_{k,l} \\
  = & \, \cos^2 \theta \Big((\po^i + \omega^i \p_r) \phi \, \p^\beta \phi \, \po_i \theta_{k,l} \,\p_\beta \theta_{k,l} - Q( \phi, \phi) \,  \po^i \theta_{k,l} \,(\po_i + \omega_i \p_r) \theta_{k,l} \Big) \\
  = & \, \cos^2 \theta \Big( \po^i \phi \,  \po_i \theta_{k,l} Q(\phi , \theta_{k,l} ) -  Q( \phi, \phi) \po^i \theta_{k,l} \,  \po_i \theta_{k,l} \Big),
\end{align*}
then
\begin{align*}
  & \, \iint_{D_{u, \ub}} \Big| \frac{1}{r} (1-\fb)  g_1^{i \beta} \, \po_i \theta_{k,l} \,\p_\beta \theta_{k,l} \Big| \\
  \lesssim & \,  \iint_{D_{u, \ub}} |u|^{-1} \Big( \big| \po^i \phi \,  \po_i \theta_{k,l} \, \Lb \phi \, L \theta_{k,l} \big| + \big| L \phi \, \Lb \phi \, \po^i \theta_{k,l} \,  \po_i \theta_{k,l}  \big| \Big)  \\
  \lesssim & \, \delta^{\frac{3}{4}} M^2  \Big( \int_0^{u} \Vert L \theta_{k,l} \Vert^2_{L^2(C_{u'})} \di u'\Big)^{\frac{1}{2}} \Big( \int_0^{u} \Vert \slashed{\nabla}  \theta_{k,l} \Vert^2_{L^2(C_{u'})} \di u'\Big)^{\frac{1}{2}} + \delta M^2  \int_0^{u} \Vert \slashed{\nabla}  \theta_{k,l} \Vert^2_{L^2(C_{u'})} \di u' \\
  \lesssim & \, \delta^{\frac{3}{4}} M^2 \cdot \delta^{\frac{3}{2}} M \cdot \delta M + \delta M^2 \cdot \delta^2 M^2  \lesssim \delta^3 M^4.
\end{align*}

In summary, the term $\Qb^{(1)} $ in \eqref{First-E} can be controlled by $\delta^2 M^4$. By an analogous process, $\Qb^{(2)}$, $\Qb^{(3)}$, $\Qb^{(4)}$ are also bounded by $\delta^2 M^4$.

Then we deal with $K$ in \eqref{First-E}.

For $K^{\Lb} [\theta_{k,l}]$, we have
\bes
\iint_{D_{u, \ub}} \big| -\frac{1}{r} L \theta_{k,l} \Lb \theta_{k,l} \big|
\lesssim  \Big( \int_0^{u} \Vert L \theta_{k,l} \Vert^2_{L^2(C_{u'})} \di u'\Big)^{\frac{1}{2}} \Big( \int_{\ub_0}^{\ub} |\ub'|^{-2}\Vert \Lb \theta_{k,l} \Vert^2_{L^2(\Cb_{\ub'})} \di \ub'\Big)^{\frac{1}{2}} \lesssim \delta^2 M^2.
\ees
The term with the worst decay in $K^{\fb L} [\theta_{k,l}]$ can be handled as
\bes
\iint_{D_{u, \ub}} \big| - \frac{1}{2}  \Lb \fb  \, (L \theta_{k,l})^2 \big| \lesssim \iint_{D_{u, \ub}} \delta^{-1} |\ub|^{-2}  M^2 \, |L \theta_{k,l} |^2 \lesssim \delta^2 M^4.
\ees
The worst term in $K^{\cos^2 \theta \Lb} [\phi_{k,l}]$ is $- \frac{1}{r} \cos^2 \theta L \phi_{k,l}\, \Lb \phi_{k,l}$, which can be dealt with the similar way as $K^{\Lb} [\theta_{k,l}]$ and can be bounded by $\delta^2 M^4$.

For
$ K^{\fb \cos^2 \theta L} [\phi_{k,l}]$, the worst term is $- \frac{1}{2}  (\Lb \fb) \cos^2 \theta |L \phi_{k,l}|^2$ , which is similar to that in $K^{\fb L} [\theta_{k,l}]$ and also can be bounded by $\delta^2 M^4$.

We turn to the term of $\big(\Box \theta_{k,l} - g_1^{\alpha \beta} \p_\alpha  \p_\beta \theta_{k,l}  -h_1^{\alpha \beta} \p_\alpha \p_\beta \phi_{k,l} \big) \tilde{\Lb} \theta_{k,l}  $ in \eqref{First-E} and
refer back to equation \eqref{Eq2-1}. For convenience, we derive some estimates for null and double null forms to prepare for the subsequent analysis.
\begin{lemma}
Under the assumption \eqref{assump}, for $0 \leq k \leq N$, it holds that
\be\label{i1i2}
\sum_{\substack{|i_1| + |i_2| \leq k \\ |i_1| ,|i_2| \leq |k/2|} } \big| Q( \phi_{i_1}, \phi_{i_2}) \big| \lesssim \delta |\ub|^{-\frac{5}{2}} M^2,
\ee
\be\label{i2}
\sum_{\substack{|i_1| + |i_2| \leq k \\ |i_1| \leq |k|, |i_2| \leq |k/2|} } \big| Q( \phi_{i_1}, \phi_{i_2}) \big| \lesssim \sum_{ |i_1| \leq |k|} \Big(|\ub|^{-1} | L \phi_{i_1} | + \delta |\ub|^{-\frac{3}{2}}  | \Lb \phi_{i_1} | +  \delta^{\frac{3}{4}} |\ub|^{-\frac{3}{2}}   |\po \phi_{i_1} |\Big) M ,
\ee
\be\label{i1i2i3}
\sum_{\substack{|i_1| + |i_2| + |i_3| \leq |k| \\ |i_1|\leq |k/2|, \,  |i_2|,  |i_3| < |k/2|} } \big | \, Q^{\alpha \beta} (\phi_{i_1}, Q_{\alpha \beta} (\theta_{i_2},   \phi_{i_3})) \big| \lesssim \delta |\ub|^{-4} M^3,
\ee
\begin{align}\label{i2i3}
& \, \sum_{\substack{|i_1| + |i_2| + |i_3| \leq |k| \\ |i_1| \leq |k|, |i_2|,  |i_3| < |k/2| } } \big | \, Q^{\alpha \beta} (\phi_{i_1}, Q_{\alpha \beta} (\theta_{i_2},   \phi_{i_3})) \big| \nnb \\
\lesssim  & \, \sum_{ |i_1| \leq |k|}  \Big( |\ub|^{-\frac{5}{2}}  | L \phi_{i_1} | + \delta |\ub|^{-\frac{7}{2}} | \Lb \phi_{i_1} | +  \delta^{\frac{3}{4}} |\ub|^{-\frac{7}{2}} |\po \phi_{i_1} | \Big) M^2,
\end{align}
and
\begin{align}\label{i1i3}
& \, \sum_{\substack{|i_1| + |i_2| + |i_3| \leq |k| \\ |i_1|,  |i_3| \leq |k/2| , \, |i_2| < |k|  } } \Big | \, Q^{\alpha \beta} (\phi_{i_1}, Q_{\alpha \beta} (\theta_{i_2},   \phi_{i_3})) \Big| \nnb \\
\lesssim & \, \sum_{|i_2| < |k|} \Big(\delta |\ub|^{-3} | \Lb \theta_{i_2+1}| + |\ub|^{-\frac{5}{2}} | L \theta_{i_2+1}| + \delta^{\frac{3}{4}} |\ub|^{-\frac{7}{2}} | \po \theta_{i_2+1}| \Big) M^2.
\end{align}

\end{lemma}

\begin{proof}
Estimates of \eqref{i1i2} and \eqref{i2} can be directly obtained from Lemma \ref{null form} and Lemma \ref{Sobolev}. Estimates of \eqref{i1i2i3} and \eqref{i2i3} can be obtained from Lemma \ref{double null}, Lemma \ref{Sobolev} and Lemma \ref{Sob-L}.

For \eqref{i1i3}, according to \eqref{LL}, it holds that
\bes
| L^2 \theta_{i_2} | \lesssim |\ub|^{-1} | L \theta_{i_2} | + |\ub|^{-1} | L \theta_{i_2+1} | \lesssim |\ub|^{-1} | L \theta_{i_2+1} |.
\ees

Similarly we can obtain
\bes
| L \Lb \theta_{i_2} | = | \Lb L \theta_{i_2} | \lesssim |\ub|^{-1} | \Lb \theta_{i_2+1} |, \quad
| \Lb^2 \theta_{i_2} |  \lesssim \delta^{-1} | \Lb \theta_{i_2+1} |, \quad
| \p \po \theta_{i_2} |  \lesssim |\ub|^{-1} | \Lb \theta_{i_2+1} |,
\ees
and
\bes
| \po^2 \theta_{i_2} |  \lesssim |\ub|^{-1} | \po \theta_{i_2+1} |.
\ees

Therefore, combining Lemma \ref{double null}, Lemma \ref{Sobolev} and Lemma \ref{Sob-L}, we have
\begin{align*}
   & \, \sum_{\substack{|i_1| + |i_2| + |i_3| \leq |k| \\ |i_1|,  |i_3| \leq |k/2| , \, |i_2| < |k|  } } \Big | \, Q^{\alpha \beta} (\phi_{i_1}, Q_{\alpha \beta} (\theta_{i_2},   \phi_{i_3})) \Big| \\
  \lesssim & \, \sum_{|i_2| < |k|} \Big( \delta^2 |\ub|^{-3} |\Lb^2 \theta_{i_2}| + \delta |\ub|^{-\frac{5}{2}} |L \Lb \theta_{i_2}| + |\ub|^{-2} |L^2 \theta_{i_2}| + \delta^{\frac{3}{2}} |\ub|^{-3} | \p \po \theta_{i_2}|  + \delta^{\frac{3}{4}} |\ub|^{-\frac{5}{2}} |\po^2 \theta_{i_2}| \\
  & \, + \delta |\ub|^{-\frac{7}{2}} | \Lb \theta_{i_2}| + |\ub|^{-\frac{5}{2}} | L \theta_{i_2}| + \delta^{\frac{3}{4}} |\ub|^{-\frac{7}{2}} | \po \theta_{i_2}| \Big) M^2\\
  \lesssim & \, \sum_{|i_2| < |k|} \Big(\delta |\ub|^{-3} | \Lb \theta_{i_2+1}| + |\ub|^{-\frac{5}{2}} | L \theta_{i_2+1}| + \delta^{\frac{3}{4}} |\ub|^{-\frac{7}{2}} | \po \theta_{i_2+1}| \Big) M^2.
\end{align*}

\end{proof}

Now let us commence with the analysis of the first term on the right-hand side of \eqref{Eq2-1}, noting that we have multiplied both sides of \eqref{Eq2-1} by $\delta^l \tilde{\Lb} \theta_{k,l}$.
We get that
\begin{align*}
 & \, \iint_{D_{u, \ub}} \Big | \,\delta^l \sum_{\substack{|i_1| + \cdots + |i_m| \leq |k| \\ 2 \leq m \leq |k|+2}} a_{i_1, \cdots, i_m} \, (\sin (2\theta))^{(m-2)} Q(Z^{i_1} \phi, Z^{i_2} \phi) Z^{i_3} \theta \cdots Z^{i_m} \theta \cdot \tilde{\Lb} \theta_{k,l}  \Big| \\
\lesssim & \, R_1 + R_2 + R_3,
\end{align*}
where
\bes
R_1 = \iint_{D_{u, \ub}}\sum_{\substack{|i_1| + |i_2| \leq |k| \\ |i_1| \leq |k|,|i_2| \leq |k/2| \\ m=2}}  \Big| \,\delta^l \sin (2\theta) Q(Z^{i_1} \phi, Z^{i_2} \phi) \cdot \tilde{\Lb} \theta_{k,l} \Big|,
\ees
\bes
R_2 = \iint_{D_{u, \ub}} \sum_{\substack{ |i_1| + \cdots + |i_m| \leq |k| \\ |i_1| \leq |k|,|i_2| \leq |k/2| \\ \max\{|i_3|, \cdots, |i_m|\} \leq |k/2| \\ 3 \leq m \leq |k|+2}} \Big| \,\delta^l Q(Z^{i_1} \phi, Z^{i_2} \phi) Z^{i_3} \theta \cdots Z^{i_m} \theta \cdot \tilde{\Lb} \theta_{k,l} \Big|,
\ees
and
\bes
R_3 = \, \iint_{D_{u, \ub}} \sum_{\substack{|i_1| + \cdots + |i_m| \leq |k| \\ |i_1|, |i_2| \leq |k/2| \\ \max\{|i_3|, \cdots, |i_m|\} \leq k \\ 3 \leq m \leq |k|+2}} \Big| \,\delta^l Q(Z^{i_1} \phi, Z^{i_2} \phi) Z^{i_3} \theta \cdots Z^{i_m} \theta \cdot \tilde{\Lb} \theta_{k,l} \Big|.
\ees

For $R_1$, combining Lemma \ref{Sob-L} and \eqref{i2}, we have
\begin{align*}
R_1
\lesssim & \iint_{D_{u, \ub}} \delta |\ub|^{-1}  M \cdot \sum_{ |i_1| \leq |k|} \Big(|\ub|^{-1} | L \phi_{i_1} | + \delta |\ub|^{-\frac{3}{2}}  | \Lb \phi_{i_1} | +  \delta^{\frac{3}{4}} |\ub|^{-\frac{3}{2}}   |\po \phi_{i_1} |\Big) M  \cdot | \tilde{\Lb} \theta_{k,l}  |\\
\lesssim & \, \delta M^2 \sum_{ |i_1| \leq |k|}  \Big( \int_0^{u} \Vert L \phi_{i_1} \Vert^2_{L^2(C_{u'})} \di u'\Big)^{\frac{1}{2}} \Big( \int_{\ub_0}^{\ub} |\ub'|^{-4}\Vert \Lb \theta_{k,l} \Vert^2_{L^2(\Cb_{\ub'})} \di \ub'\Big)^{\frac{1}{2}} \\
& + \delta^{2} M^2 \sum_{ |i_1| \leq |k|} \Big( \int_{\ub_0}^{\ub} |\ub'|^{-3}\Vert \Lb \phi_{i_1} \Vert^2_{L^2(\Cb_{\ub'})} \di \ub'\Big)^{\frac{1}{2}}  \Big( \int_{\ub_0}^{\ub} |\ub'|^{-2}\Vert \Lb \theta_{k,l} \Vert^2_{L^2(\Cb_{\ub'})} \di \ub'\Big)^{\frac{1}{2}} \\
& + \delta^{\frac{7}{4}} M^2 \sum_{ |i_1| \leq |k|}  \Big( \int_0^{u} \Vert \slashed{\nabla}  \phi_{i_1} \Vert^2_{L^2(C_{u'})} \di u'\Big)^{\frac{1}{2}} \Big( \int_{\ub_0}^{\ub} |\ub'|^{-5}\Vert \Lb \theta_{k,l} \Vert^2_{L^2(\Cb_{\ub'})} \di \ub'\Big)^{\frac{1}{2}} \\
\lesssim & \,  \delta M^2 \cdot \delta^{\frac{3}{2}} M  \cdot \delta^{\frac{1}{2}} M + \delta^2 M^2 \cdot \delta^{\frac{1}{2}} M \cdot \delta^{\frac{1}{2}} M + \delta^{\frac{7}{4}} M^2 \cdot \delta M \cdot  \delta^{\frac{1}{2}} M \lesssim \delta^3 M^4.
\end{align*}
Reminding that the term $\fb L \theta_{k,l}$ in  $\tilde{\Lb} \theta_{k,l}$  can be controlled by $\Lb \theta_{k,l}$, since $\Vert \fb L \theta_{k,l} \Vert_{L^2(D_{u, \ub})} \lesssim \delta^{\frac{3}{2}} M^3 $ and  $\Vert \Lb \theta_{k,l} \Vert_{L^2(D_{u, \ub})} \lesssim \delta^{\frac{1}{2}} M$.

Similar calculation for $R_2$ gives that
\begin{align*}
R_2
\lesssim & \iint_{D_{u, \ub}} \sum_{\substack{ |i_1| + \cdots + |i_m| \leq |k| \\ |i_1| \leq |k| \\ \max\{|i_3|, \cdots, |i_m|\} \leq |k/2| \\ 3 \leq m \leq |k|+2}} \Big|  \Big(  |\ub|^{-1} | L \phi_{i_1} | + \delta |\ub|^{-\frac{3}{2}}  | \Lb \phi_{i_1} | +  \delta^{\frac{3}{4}} |\ub|^{-\frac{3}{2}}   |\po \phi_{i_1} |\Big) M  \cdot \theta_{i_3} \cdots  \theta_{i_m} \cdot \tilde{\Lb} \theta_{k,l} \Big|\\
\lesssim & \iint_{D_{u, \ub}}\sum_{\substack{ |i_1| \leq |k|\\ 3 \leq m \leq |k|+2}} \big( \delta |\ub|^{-1}  M \big)^{m-2} \Big(  |\ub|^{-1} | L \phi_{i_1} | + \delta |\ub|^{-\frac{3}{2}}  | \Lb \phi_{i_1} | +  \delta^{\frac{3}{4}} |\ub|^{-\frac{3}{2}}   |\po \phi_{i_1} |\Big) M \cdot | \tilde{\Lb} \theta_{k,l} | \\
\lesssim & \,\delta^{3} M^4.
\end{align*}

For $R_3$, without loss of generality, we assume that $\max\{|i_3|, \cdots, |i_m|\} = |i_3|$.  In the short pulse
region II, notice that for any Lorentz vector field,
\be\label{Z}
|Z \theta| \leq (1+|t+r|) \, |\p \theta | \lesssim |\ub| \, |\p \theta |.
\ee
Combining Lemma \ref{Sob-L}, \eqref{i1i2} and \eqref{Z}, we have
\begin{align*}
R_3
\lesssim & \, \iint_{D_{u, \ub}} \sum_{\substack{|i_1| + \cdots + |i_m| \leq |k| \\ |i_1|, |i_2| \leq |k/2| \\ \max\{|i_3|, \cdots, |i_m|\} \leq k \\ 3 \leq m \leq |k|+2}}  \Big| \,\delta^l  Q(Z^{i_1} \phi, Z^{i_2} \phi) \big(|\ub| \, \p Z^{i_3-1} \theta \big) \cdots Z^{i_m} \theta \cdot \tilde{\Lb} \theta_{k,l} \Big| \\
\lesssim & \, \iint_{D_{u, \ub}} \sum_{\substack{ |i_3| \leq k \\ 3 \leq m \leq |k|+2}} \big( \delta |\ub|^{-\frac{5}{2}}  M^2 \big) \big( \delta |\ub|^{-1}  M \big)^{m-3} \Big|  \big(|\ub| \, \p \theta_{i_3-1} \big)  \cdot \tilde{\Lb} \theta_{k,l} \Big|\\
\lesssim & \, \delta M^2 \big( \int_{\ub_0}^{\ub} |\ub'|^{-\frac{3}{2}}\Vert \Lb \phi_{k-1} \Vert^2_{L^2(\Cb_{\ub'})} \di \ub'\big)^{\frac{1}{2}}  \big( \int_{\ub_0}^{\ub} |\ub'|^{-\frac{3}{2}}\Vert \Lb \theta_{k,l} \Vert^2_{L^2(\Cb_{\ub'})} \di \ub'\big)^{\frac{1}{2}}
 \lesssim  \delta^2 M^4.
\end{align*}

Similarly, the second term of the right-hand side of \eqref{Eq2-1} can also be bounded by $\delta^2 M^4$, because the product of two null forms $Q^{\alpha \beta } Q_{\alpha \beta}$ enjoys better decay than $Q$.

For the third term of \eqref{Eq2-1}, we first handle the double null term.
We decompose the term as
\begin{align*}
& \iint_{D_{u, \ub}} \Big | \,\delta^l \sum_{\substack{|i_1| + \cdots + |i_m| \leq |k| \\ |i_2|, |i_3| < |k| \\ 3 \leq m \leq |k|+3} } c_{i_1, \cdots, i_m} \, (\cos^2 \theta)^{(m-3)}  Q^{\alpha \beta} (Z^{i_1}\phi, Q_{\alpha \beta} (Z^{i_2}\theta,  Z^{i_3} \phi)) \\
 & \qquad \qquad \cdot Z^{i_4} \theta \cdots Z^{i_m} \theta \cdot \tilde{\Lb} \theta_{k,l} \, \Big|  \\
\lesssim & \, H_1 + H_2 + H_3,
\end{align*}
where
\bes
H_1 = \iint_{D_{u, \ub}} \sum_{\substack{|i_1| + \cdots + |i_m| \leq |k| \\ |i_1| \leq |k|, |i_2|,  |i_3| < |k/2| \\ 3 \leq m \leq |k|+3} } \Big | \,\delta^l \, Q^{\alpha \beta} (Z^{i_1}\phi, Q_{\alpha \beta} (Z^{i_2}\theta,  Z^{i_3} \phi))Z^{i_4} \theta \cdots Z^{i_m} \theta \cdot \tilde{\Lb} \theta_{k,l} \Big|,
\ees
\bes
H_2 = \iint_{D_{u, \ub}} \sum_{\substack{ |i_1| + \cdots + |i_m| \leq |k| \\ |i_2| < |k|, |i_1|, |i_3| \leq |k/2| \\ 3 \leq m \leq |k|+3} } \Big | \,\delta^l \, Q^{\alpha \beta} (Z^{i_1}\phi, Q_{\alpha \beta} (Z^{i_2}\theta,  Z^{i_3} \phi))Z^{i_4} \theta \cdots Z^{i_m} \theta \cdot \tilde{\Lb} \theta_{k,l} \Big|,
\ees
and
\bes
H_3 = \iint_{D_{u, \ub}} \sum_{\substack{ |i_1| + \cdots + |i_m| \leq |k| \\ |i_1| \leq |k/2| , |i_2|, |i_3| < |k/2| \\ \max\{|i_4|, \cdots, |i_m|\} \leq |k|  \\ 4 \leq m \leq |k|+3} } \Big | \,\delta^l \,  (\cos^2 \theta)^{(m-3)} Q^{\alpha \beta} (Z^{i_1}\phi, Q_{\alpha \beta} (Z^{i_2}\theta,  Z^{i_3} \phi))Z^{i_4} \theta \cdots Z^{i_m} \theta \cdot \tilde{\Lb} \theta_{k,l} \Big|.
\ees

For $H_1$, according to Lemma \ref{Sob-L} and \eqref{i2i3}, we obtain
\begin{align*}
H_1
\lesssim & \iint_{D_{u, \ub}}   \sum_{\substack{ |i_1| \leq |k| \\ 3 \leq m \leq |k|+3}} \big( \delta |\ub|^{-1}  M \big)^{m-3} \Big( |\ub|^{-\frac{5}{2}}  | L \phi_{i_1} | + \delta |\ub|^{-\frac{7}{2}} | \Lb \phi_{i_1} | +  \delta^{\frac{3}{4}} |\ub|^{-\frac{7}{2}} |\po \phi_{i_1} | \Big) M^2  \cdot \Big| \tilde{\Lb} \theta_{k,l} \Big| \\
\lesssim & \, M^2 \cdot \delta^{\frac{3}{2}} M \cdot  \delta^{\frac{1}{2}} M + \delta  M^2 \cdot \delta^{\frac{1}{2}} M \cdot  \delta^{\frac{1}{2}} M + \delta^{\frac{3}{4}} M^2 \cdot \delta M \cdot  \delta^{\frac{1}{2}} M \lesssim \delta^2  M^4.
\end{align*}

For $H_2$, according to \eqref{i1i3}, it holds that
\begin{align*}
H_2
\lesssim & \, \iint_{D_{u, \ub}}   \sum_{\substack{|i_2| < |k|\\ 3 \leq m \leq |k|+3} }
\big( \delta |\ub|^{-1}  M \big)^{m-3} \Big(\delta |\ub|^{-3} | \Lb \theta_{i_2+1}| + |\ub|^{-\frac{5}{2}} | L \theta_{i_2+1}| \\
& \qquad \qquad\qquad\qquad \qquad \qquad \qquad \quad + \delta^{\frac{3}{4}} |\ub|^{-\frac{7}{2}} | \po \theta_{i_2+1}| \Big) M^2 \cdot \big| \tilde{\Lb} \theta_{k,l} \big| \\
\lesssim & \, M^2 \Big( \delta \cdot \delta^{\frac{1}{2}} M \cdot \delta^{\frac{1}{2}} M + \delta^{\frac{3}{2}} M \cdot \delta^{\frac{1}{2}} M + \delta^{\frac{3}{4}} \cdot  \delta M \cdot \delta^{\frac{1}{2}} M \Big) \lesssim \delta^2  M^4.
\end{align*}

 For $H_3$, without loss of generality, we assume that $\max\{|i_4|, \cdots, |i_m|\} = |i_4|$, then according to \eqref{i1i2i3}, we have
\begin{align*}
H_3
\lesssim & \iint_{D_{u, \ub}} \Big( \sum_{\substack{|i_4| \leq |k| \\  m = 4 }}  | \sin 2 \theta | +  \sum_{\substack{|i_4| \leq |k| \\ 4 < m \leq |k|+3} } \big( \delta |\ub|^{-1}  M \big)^{m-4} \Big) \cdot \delta |\ub|^{-4}  M^3 \cdot |\ub| \, |\p \theta_{i_4-1}| \cdot | \tilde{\Lb} \theta_{k,l} | \\
\lesssim & \iint_{D_{u, \ub}}  \delta |\ub|^{-1}  M  \cdot \delta |\ub|^{-3}  M^3 \cdot  |\p \theta_{k-1}| \cdot | \tilde{\Lb} \theta_{k,l} | \lesssim \delta^2 M^4 \cdot \delta^{\frac{1}{2}} M \cdot \delta^{\frac{1}{2}} M \lesssim \delta^3 M^6.
\end{align*}

Up to this point, we have established that the third term in equation \eqref{Eq2-1} is bounded by $ \delta^2 M^4$.
Now, we proceed to examine the last term of \eqref{Eq2-1}, which is
\begin{align}
& \iint_{D_{u, \ub}} \Big| \cos^2 \theta \Big( Q^{\alpha \beta} (\phi , \p_\beta \phi) \p_\alpha \theta_{k,l} - Q^{\alpha \beta} (\phi , \p_\beta \theta) \p_\alpha \phi_{k,l} \Big) \cdot \tilde{\Lb} \theta_{k,l}  \Big|. \label{fourth}
\end{align}
According to \eqref{p1} and \eqref{laplace}, the first term of equation \eqref{fourth} can be estimated to be
\begin{align*}
 & \, \iint_{D_{u, \ub}} \Big| \cos^2 \theta \, Q^{\alpha \beta} (\phi , \p_\beta \phi) \p_\alpha \theta_{k,l}  \cdot \tilde{\Lb} \theta_{k,l} \Big| \\
  \lesssim  & \, \iint_{D_{u, \ub}} \Big| \big( \p^\alpha \phi \, \p^\beta \p_\beta \phi - \p^\beta \phi \, \p^\alpha \p_\beta \phi \big)  \, \p_\alpha \theta_{k,l} \cdot \tilde{\Lb} \theta_{k,l} \Big| \\
  \lesssim & \iint_{D_{u, \ub}} \Big( |\ub|^{-2} M \cdot |\Lb \phi | \, |L \theta_{k,l}| + \big( |\ub|^{-\frac{5}{2}} M^2 |L \theta_{k,l}| + \delta |\ub|^{-\frac{7}{2}} M^2   \cdot |\Lb \theta_{k,l} | \big) \Big)  \, |\tilde{\Lb} \theta_{k,l} |
   \lesssim  \delta^2 M^4.
\end{align*}
The second term in equation \eqref{fourth} mirrors the first term and is similarly bounded by $\delta^2 M^4$.

At last, the estimate of $\big(\cos^2 \theta \, \Box \phi_{k,l} - g_2^{\alpha \beta} \p_\alpha  \p_\beta \phi_{k,l}  -h_2^{\alpha \beta} \p_\alpha \p_\beta \theta_{k,l} \big) \tilde{\Lb} \phi_{k,l} $ in \eqref{First-E} is completely similar to $\big(\Box \theta_{k,l} - g_1^{\alpha \beta} \p_\alpha  \p_\beta \theta_{k,l}  -h_1^{\alpha \beta} \p_\alpha \p_\beta \phi_{k,l} \big) \tilde{\Lb} \theta_{k,l} $.

In summary, we have derived the following conclusions
\be
\delta \Eb^2_{\leq N}  ( u, \ub)  \lesssim \delta I^2_N(\theta_0, \theta_1, \phi_0, \phi_1) + \delta^2 M^4,
\ee
which is equivalent to
\be\label{con-1}
\Eb^2_{\leq N}  ( u, \ub) \lesssim I^2_N(\theta_0, \theta_1, \phi_0, \phi_1) + \delta M^4.
\ee

\subsubsection{Estimates for $E_{\leq N}$}

Recall the energy identity \eqref{energy1}. According to Lemma \ref{lem-en1}, along with equations \eqref{E} -- \eqref{sum}, the sum of the energy terms from the aforementioned identity for
$0 \leq l \leq k \leq N$ successfully bounds
$\delta^2 E_{\leq N} (u, \ub)$. Consequently, we must now address the energy inequality.
\begin{align}
\delta^2 E^2_{\leq N}  ( u, \ub)\lesssim & \,\delta^2 I^2_N(\theta_0, \theta_1, \phi_0, \phi_1) + \sum_{0\leq l \leq k \leq N} \iint_{D_{u, \ub}} \Big| Q^{(1)} + Q^{(2)} + Q^{(3)} + Q^{(4)} \Big| \nnb \\
&   + \Big|  K^L [\theta_{k,l}]+ K^{f \Lb} [\theta_{k,l}]+ K^{\cos^2 \theta L} [\phi_{k,l}] + K^{f \cos^2 \theta \Lb}[\phi_{k,l}] \Big| \nnb \\
&  + \Big| \big(\Box \theta_{k,l} - g_1^{\alpha \beta} \p_\alpha  \p_\beta \theta_{k,l}  -h_1^{\alpha \beta} \p_\alpha \p_\beta \phi_{k,l} \big) \tilde{L} \theta_{k,l} \, \Big|  \nnb \\
&  +\Big| \big(\cos^2 \theta \, \Box \phi_{k,l} - g_2^{\alpha \beta} \p_\alpha \p_\beta \phi_{k,l}  -h_2^{\alpha \beta} \p_\alpha \p_\beta \theta_{k,l} \big) \tilde{L} \phi_{k,l}  \, \Big|, \label{First-Eb}
\end{align}
where $Q^{(1)}$, $Q^{(2)}$, $Q^{(3)}$, $Q^{(4)}$ are defined in \eqref{Q1}, \eqref{Q2}, \eqref{Q3}, \eqref{Q4} and $K^{L} [\theta_{k,l}]$, $K^{f \Lb} [\theta_{k,l}]$, $K^{\cos^2 \theta L} [\phi_{k,l}]$, $K^{f \cos^2 \theta \Lb} [\phi_{k,l}]$ are defined in \eqref{KL}, \eqref{fLb}, \eqref{cosL}, \eqref{fcosLb}.

Similar to $\Eb_{\leq N} (u, \ub)$, we can control the first three integral terms of \eqref{First-Eb} by $ \delta^3 M^4$.

For the terms of $K$, all of them can be bounded by $ \delta^3 M^4$ except $ K^L [\theta_{k,l}]$ and $\frac{1}{r} \cos^2 \theta \, L \phi_{k,l}\, \Lb \phi_{k,l} $ in $K^{\cos^2 \theta L} [\phi_{k,l}]$. We will use the previous result for $\Eb_{\leq N} (u, \ub)$ to deduce
\begin{align*}
&\iint_{D_{u, \ub}} \Big| \frac{1}{r}  L \phi_{k,l}\, \Lb \phi_{k,l} \Big| + \Big| \frac{1}{r} \cos^2 \theta \, L \phi_{k,l}\, \Lb \phi_{k,l} \Big| \\
\lesssim & \iint_{D_{u, \ub}} \delta^{-1}  | L \phi_{k,l}|^2 + \delta |\Lb \phi_{k,l}|^2 \\
\lesssim &  \int_0^{u} \delta^{-1} \Vert L \phi_{k,l} \Vert^2_{L^2(C_{u'})} \di u' + \int_{\ub_0}^{\ub} \delta |\ub'|^{-2}\Vert \Lb \phi_{k,l} \Vert^2_{L^2(\Cb_{\ub'})} \di \ub' \\
\lesssim & \int_0^{u} \delta^{-1} E^2_{\leq k}(u')\di u' + \delta \big( \delta I^2_N + \delta^2 M^4 \big),
\end{align*}
where $E^2_{\leq k}(u') = \sup\limits_{\ub} E^2_{\leq k}(u', \ub)$.
\begin{remark}
	If we estimate the above term directly as before, we can only get
	\bes
	\iint_{D_{u, \ub}} \Big| \frac{1}{r}  L \phi_{k,l}\, \Lb \phi_{k,l} \Big| + \Big| \frac{1}{r} \cos^2 \theta \, L \phi_{k,l}\, \Lb \phi_{k,l} \Big|  \lesssim \delta^2 M^2,
	\ees
	which is not good enough to close the bootstrap argument. That is the main reason for us to estimate $\underline{E}$ first.

\end{remark}

In summary, we conclude
\bes
\delta^2 E^2_{\leq N}  ( u, \ub)  \lesssim \delta^2 I^2_N(\theta_0, \theta_1, \phi_0, \phi_1) + \delta^3 M^4 +  \int_0^{u} \delta E^2_{\leq N}(u')\di u' .
\ees
Taking the supremum on $\ub$, we get
\bes
E^2_{\leq N}  ( u)  \lesssim  I^2_N(\theta_0, \theta_1, \phi_0, \phi_1) + \delta M^4 +  \int_0^{u} \delta^{-1} E^2_{\leq N}(u')\di u' .
\ees
Applying the Gr{\"o}nwall inequality yields
\be\label{con-2}
E^2_{\leq N}  ( u, \ub)  \leq E^2_{\leq N}  ( u)  \lesssim  I^2_N(\theta_0, \theta_1, \phi_0, \phi_1) + \delta M^4 .
\ee

\subsection{The continuity argument}\label{3.5}

Now we can prove the global existence of smooth solution in region II by the
continuity argument. Combining \eqref{con-1} and \eqref{con-2},  it holds that
\be\label{con-3}
E_{\leq N}  ( u, \ub) + \Eb_{\leq N}  ( u, \ub) \leq  C \big(I_N(\theta_0, \theta_1, \phi_0, \phi_1) + \delta^{\frac{1}{2}} M^2 \big),
\ee
where $C$ is a positive constant. Let $[0, T]$ be the largest time interval on which the bootstrap assumption \eqref{boostrap} holds. By the local well-posedness for wave equations, we know that $T > 0$. If we choose $M > 4CI_N(\theta_0, \theta_1, \phi_0, \phi_1)$ and $\delta < (4CM)^{-2}$, the inequality \eqref{con-3} becomes
\bes
E_{\leq N}  ( u, \ub) + \Eb_{\leq N}  ( u, \ub) \leq  \frac{1}{2} M.
\ees
Now the solution and the energy estimates can be extended to a larger time
interval $[0, T']$, thus contradicting the maximality of $T$. This implies $T = \infty$ and the solution is global in region II.

\section{Small solution on the slice $C_{\delta}$}\label{sec4}
Recalling the short pulse region II, we possess the a priori
$L^{\infty}$
estimates as detailed in Lemma \ref{Sobolev}. In this section, we aim to refine these estimations. More specifically, we will demonstrate that\begin{proposition}
  Suppose we have bounded $E_N(\delta, \ub)$ and $\Eb_N(\delta, \ub)$ for $N \geq 6$. Then, for all $0 \leq l \leq k \leq N-3$, it holds that
  \begin{align*}
  \big|\po (\p^l \Gamma^{k-l} \theta)  \big|_{L^{\infty}(C_\delta)} \lesssim \delta^{\frac{3}{4}} |\ub|^{-2} , \qquad
  & \big|\po (\p^l \Gamma^{k-l} \phi)  \big|_{L^{\infty}(C_\delta)} \lesssim \delta^{\frac{3}{4}} |\ub|^{-2} , \\
  \big|\Lb (\p^l \Gamma^{k-l} \theta)  \big|_{L^{\infty}(C_\delta)} \lesssim \delta^{\frac{3}{4}} |\ub|^{-1} , \qquad
  & \big|\Lb (\p^l \Gamma^{k-l} \phi)  \big|_{L^{\infty}(C_\delta)} \lesssim \delta^{\frac{3}{4}} |\ub|^{-1} ,
  \end{align*}
  or equivalently,
  \begin{align}\label{improve}
  \big|\po Z^k \theta  \big|_{L^{\infty}(C_\delta)} \lesssim \delta^{\frac{3}{4}} |\ub|^{-2} , \qquad
  & \big|\po Z^k \phi  \big|_{L^{\infty}(C_\delta)} \lesssim \delta^{\frac{3}{4}} |\ub|^{-2} , \nnb\\
  \big|\Lb Z^k \theta  \big|_{L^{\infty}(C_\delta)} \lesssim \delta^{\frac{3}{4}} |\ub|^{-1} , \qquad
  & \big|\Lb Z^k \phi  \big|_{L^{\infty}(C_\delta)} \lesssim \delta^{\frac{3}{4}} |\ub|^{-1}.
  \end{align}
\end{proposition}

\begin{proof}
 We will establish this proposition through induction, focusing solely on
 $\theta$. For the base case where
$l=0$, the conclusion regarding the good derivatives is immediately derived from the a priori estimates provided in Lemma \ref{Sob-L} and \eqref{po-Omega}, \eqref{L},
  \bes
  |\po \Gamma^k \theta| \lesssim \frac{|\Gamma^{k+1} \theta|}{|\ub|} \lesssim \delta |\ub|^{-2}   \lesssim \delta^{\frac{3}{4}} |\ub|^{-2} .
  \ees

  To demonstrate the improvement for
 $\Lb \Gamma^k \theta$, it is sufficient to focus on this term. To achieve this, we will utilize an ordinary differential equation (ODE) form of equation \eqref{Eq2-1}. The left-hand side of this ODE can be expanded as follows
   \bes
   - L (\Lb \Gamma^k \theta) - \frac{1}{r} \Lb \Gamma^k \theta + \Big(\po^i \po_i  + \frac{1}{r} L \Big) \Gamma^k \theta -  g_1^{\alpha \beta} \p_\alpha \p_\beta \Gamma^k \theta  -h_1^{\alpha \beta} \p_\alpha \p_\beta \Gamma^k \phi.
  \ees

 First, we handle the term $g_1^{\alpha \beta} \p_\alpha \p_\beta \Gamma^k \theta$. According to Lemma \ref{symmetric}, \eqref{metric} and noting that $g_1^{\alpha \beta}$ is symmetric, we have
 \begin{align*}
    & \, g_1^{\alpha \beta} \p_\alpha \p_\beta \Gamma^k \theta \\
   = & \,  \frac{1}{4} (\Lb \phi)^2 \, L^2 \Gamma^k \theta + \frac{1}{2} L \phi \, \Lb \phi \, L \Lb \Gamma^k \theta + \frac{1}{4} (L \phi)^2 \, \Lb^2 \Gamma^k \theta + \po^i \phi \, \po^j \phi \, \po_i \po_j  \Gamma^k \theta - \Lb \phi \, \po^i \phi \, L \po_i \Gamma^k \theta \\
   & \, \, - L \phi \, \po^i \phi \, \Lb \po_i \Gamma^k \theta + \frac{1}{2r} \po^i \phi \, \po_i \phi \, (L-\Lb)\Gamma^k \theta \\
   & \, \, - Q(\phi ,\phi) \big( -L \Lb  \Gamma^k \theta +  \po^i  \po_i \Gamma^k \theta + \frac{1}{r} (L- \Lb) \Gamma^k \theta \big).
 \end{align*}
By combining Lemma \ref{Sobolev} and Lemma \ref{decay}, we identify the most significant term in $g_1^{\alpha \beta} \p_\alpha \p_\beta \Gamma^k \theta$ is $\frac{1}{4} (L \phi)^2 \, \Lb^2 \Gamma^k \theta $, which can be controlled by
\begin{align*}
    \big| \frac{1}{4} (L \phi)^2 \, \Lb^2 \Gamma^k \theta \big|
  \lesssim  \delta |\ub|^{-4} .
\end{align*}
 For the term $h_1^{\alpha \beta} \p_\alpha \p_\beta \Gamma^k \phi$, we define the symmetrization
 \bes
 h^{\alpha \beta} = \frac{1}{2} \big( h_1^{\alpha \beta} + h_1^{\beta \alpha } \big),
 \ees
 then
 \bes
 h_1^{\alpha \beta} \p_\alpha \p_\beta \Gamma^k \phi = h^{\alpha \beta} \p_\alpha \p_\beta \Gamma^k \phi.
 \ees

 Now, the term
  $h_1^{\alpha \beta} \p_\alpha \p_\beta \Gamma^k \phi$ can be bounded similarly as $g_1^{\alpha \beta} \p_\alpha \p_\beta \Gamma^k \theta$. We can also get that
  $\big(\po^i \po_i  + \frac{1}{r} L \big) \Gamma^k \theta$ is bounded by
  $\delta^{\frac{3}{4}} |\ub|^{-\frac{5}{2}} $. By exploiting the favorable structures of the null forms and double null forms, the right-hand side of equation \eqref{Eq2-1} can be bounded by
  $\delta |\ub|^{-\frac{7}{2}} $. Given that
  $\delta$ is chosen to be sufficiently small, equation \eqref{Eq2-1} simplifies to
  \be\label{ODE}
  \Big|  L (\Lb \Gamma^k \theta) + \frac{1}{r} \Lb \Gamma^k \theta \Big| \lesssim \delta^{\frac{3}{4}} |\ub|^{-\frac{5}{2}} .
  \ee
  Multiplying $2\ub^2 \Lb \Gamma^k \theta$ on both sides of \eqref{ODE}, we obtain
  \begin{align*}
  L \big(\ub^2 (\Lb \Gamma^k \theta)^2 \big) \lesssim & \, 2 \Big|\frac{1}{\ub} - \frac{1}{r} \Big| \, \big(\ub^2 (\Lb \Gamma^k \theta)^2 \big) + 2 \delta^{\frac{3}{4}} |\ub|^{-\frac{5}{2}} M \, \ub^2 |\Lb \Gamma^k \theta | \\
  \lesssim &\, \delta |\ub|^{-2} \, \big(\ub^2 (\Lb \Gamma^k \theta)^2 \big) + \delta^{\frac{3}{4}} |\ub|^{-\frac{3}{2}} M \, |\ub \, \Lb \Gamma^k \theta|.
  \end{align*}
  Dividing both sides of the above equation by $|\ub \, \Lb \Gamma^k \theta|$, we have
  \bes
  L |\ub \, \Lb \Gamma^k \theta|
  \lesssim \delta |\ub|^{-2} \, |\ub \, \Lb \Gamma^k \theta|+ \delta^{\frac{3}{4}} |\ub|^{-\frac{3}{2}} .
  \ees
  Integrating in the $\ub= \frac{1}{2} (t+r)$ direction on  $C_{\delta}$ and noting that $\Lb \Gamma^k \theta$ vanishes on $S_{\delta, 1-\delta}$, we have
  \bes
   |\ub \, \Lb \Gamma^k \theta|  \lesssim \int_{1-\delta}^{\ub}  \delta |\ub'|^{-2} \, |\ub' \, \Lb \Gamma^k \theta| \, \di \ub' + \delta^{\frac{3}{4}} .
  \ees
  An application of the Gr{\"o}nwall inequality yields
  \bes
  |\Lb \Gamma^k \theta| \lesssim \delta^{\frac{3}{4}} |\ub|^{-1} .
  \ees

  When $l = 1$, taking $Z^k = \p \Gamma^{k-1}$ in \eqref{Eq2-1}. Combining the situation $l = 0$ and Lemma \ref{decay}, we have
  \bes
  |\po \p \Gamma^{k-1} \theta| \lesssim \delta^{\frac{3}{4}} |\ub|^{-2},
  \ees
  and
  \be\label{Lb2}
  |\Lb \p \Gamma^{k-1} \theta| \lesssim \delta^{-\frac{1}{4}} |\ub|^{-1}.
  \ee
  From \eqref{Lb2} and Lemma \ref{decay},  we can further deduce $| \Lb^2 \p \Gamma^{k-1} \theta | \lesssim \delta^{-\frac{5}{4}} |\ub|^{-1} $.

  To get the smallness of $\Lb \p \Gamma^{k-1} \theta$, we  will still use the ODE form of \eqref{Eq2-1}. The worst term
  in the right-hand side of \eqref{Eq2-1} comes from the double null form, which can be controlled as
  \begin{align*}
      & \, \sum_{\substack{|i_1| + |i_2| + |i_3| \leq |k| \\ |i_1| \geq 1 } } \Big| Q^{\alpha \beta} \big(\p \Gamma^{i_1-1}\phi, Q_{\alpha \beta} ( \Gamma^{i_2}\theta,  \Gamma^{i_3} \phi) \big) \Big| \\
      & \, + \sum_{\substack{|i_1| + |i_2| + |i_3| \leq |k| \\ 1 \leq |i_2| < |k| } } \Big| Q^{\alpha \beta} \big(\Gamma^{i_1}\phi, Q_{\alpha \beta} ( \p \Gamma^{i_2-1}\theta,  \Gamma^{i_3} \phi) \big) \Big| \\
     \lesssim & \, \sum_{\substack{|i_1| + |i_2| + |i_3| \leq |k| \\ |i_1| \geq 1 } } \big| L \p \Gamma^{i_1-1}\phi \big| \, \big| \Lb^2  \Gamma^{i_2} \theta  \big| \, \big| L \Gamma^{i_3}\phi \big|
      + \sum_{\substack{|i_1| + |i_2| + |i_3| \leq |k| \\ 1 \leq |i_2| < |k| } }  \big| L \Gamma^{i_1}\phi \big| \, \big| \Lb^2 \p \Gamma^{i_2-1} \theta  \big| \, \big| L \Gamma^{i_3}\phi \big| \\
      \lesssim & \, \delta^{\frac{3}{4}} |\ub|^{-2}  \cdot \delta^{-\frac{1}{4}} |\ub|^{-1} \cdot \delta^{\frac{3}{4}} |\ub|^{-2} +  \delta |\ub|^{-\frac{3}{2}}  \cdot \delta^{-\frac{5}{4}} |\ub|^{-1}   \cdot \delta |\ub|^{-\frac{3}{2}}
      \lesssim  \delta^{\frac{3}{4}} |\ub|^{-4}.
  \end{align*}

  The term $\big(\po^i \po_i  + \frac{1}{r} L \big) \p \Gamma^{k-1} \theta$ contains the good derivatives and can be controlled by $\delta^{\frac{3}{4}} |\ub|^{-3} $. As above, one of the worst term in  $g_1^{\alpha \beta} \p_\alpha \p_\beta \Gamma^k \theta$ is still $\frac{1}{4} (L \phi)^2 \, \Lb^2 \p \Gamma^{k-1} \theta $, we can control it by
  \bes
  \Big|\frac{1}{4} (L \phi)^2 \, \Lb^2 \p \Gamma^{k-1} \theta \Big| \lesssim  \delta^2 |\ub|^{-3}  \cdot  \delta^{-\frac{5}{4}} |\ub|^{-1}  \lesssim \delta^{\frac{3}{4}} |\ub|^{-4}.
  \ees
  Then the ODE becomes
  \bes
  \Big|  L (\Lb \p \Gamma^{k-1} \theta) + \frac{1}{r} (\Lb \p \Gamma^{k-1} \theta) \Big| \lesssim \delta^{\frac{3}{4}} |\ub|^{-3} .
  \ees
  The same derivation as above gives that
  \bes
  |\Lb \p \Gamma^{k-1} \theta| \lesssim \delta^{\frac{3}{4}} |\ub|^{-1}.
  \ees

  The conclusions for higher-order derivatives on the null cone
 $C_{\delta}$ are derived through induction.
\end{proof}

\section{Globally smooth solution in the small data region I}\label{sec5}

In this section, we will explore the behavior of the solution in region I, which is bounded by
 $C_\delta$ and $\Sigma_1$.

Denote
\bes
D_I := \big\{ (t,x) \big| u \geq \delta, t \geq 1 \big\}, \qquad D_{I,t} := \big\{ (t',x) \big| u \geq \delta, 1 \leq t' \leq t \big\},
\ees
and
\bes
\Sigma_t = \big\{ (t,x) \big| u \geq \delta, t \geq 1 \big\} \cap D_I.
\ees
Given a point $(t,x) \in \Sigma_t$, we use $(t, B(t, x)) $ to denote the points on $C_\delta$.

The following lemma is the Klainerman-Sobolev inequality \cite{W-W}.

\begin{lemma}
  (Klainerman-Sobolev inequality) For any $\psi \in C^\infty( \mathbb R^{1+n} )$ and $t>1$ with $(t, x) \in D_I$, we have
  \bes
  |\psi (t, x) | \lesssim \frac{1}{(1+ |u|)^{\frac{1}{2}} }| \psi (t, B(t,x))| + \frac{1}{(1+ |\ub|) (1+ |u|)^{\frac{1}{2}}}\, \sum_{m \leq 3} \Vert Z^m \psi \Vert_{L^2(\Sigma_t)}.
  \ees
\end{lemma}

Combining \eqref{improve}, we can get the pointwise decay estimates in region I
  \begin{align}\label{pointdecay}
  \big|\p Z^k \theta \big| \lesssim & \frac{1}{(1+ |\ub|)(1+|u|)^{\frac{1}{2}} } \, \delta^{\frac{3}{4}} +   \frac{1}{(1+ |\ub|) (1+ |u|)^{\frac{1}{2}}}\,  \sum_{m \leq 3} \Vert \p Z^{k+m}  \theta \Vert_{L^2(\Sigma_t)} \nnb \\
   \lesssim & (1+ |\ub|)^{-1} \, (1+|u|)^{-\frac{1}{2}} \, \delta^{\frac{3}{4}}  M , \qquad k \leq N-3,
  \end{align}
  here $M$ is a large constant, which is to be determined later and only depends on the initial data.

Next lemma is on the Hardy inequality in region I, whose proof is similar to Lemma \ref{Sob-L}.
\begin{lemma}\label{Sob-L-I}
	For $0  \leq k \leq N - 3$, it holds that
	\be
\big|Z^k \theta \big| \lesssim (1+ |\ub|)^{-1} \, (1+|u|)^{\frac{1}{2}} \, \delta^{\frac{3}{4}}  M,
 \quad \big|Z^k\phi\big|\lesssim (1+ |\ub|)^{-1} \, (1+|u|)^{\frac{1}{2}} \, \delta^{\frac{3}{4}}  M,
	\ee
	where $M$ is a positive constant.
\end{lemma}

\begin{remark}
	The Hardy inequality presented in Lemma \ref{Sob-L-I} within region I differs from the one stated in Lemma \ref{Sob-L}, primarily because the variable $u$
	is bounded in region II, which affect the form of the inequality.
	\end{remark}

We define the homogeneous and non-homogeneous energies on the constant $t$ slice $\Sigma_t$ as
\begin{gather*}
E_k(t) = \Big( \int_{\Sigma_t} (\p_t Z^k \theta )^2 + |\nabla  Z^k \theta |^2 \di x \Big)^{\frac{1}{2}} + \Big( \int_{\Sigma_t} (\p_t Z^k \phi )^2 + |\nabla  Z^k \phi |^2 \di x \Big)^{\frac{1}{2}}, \\   E_{\leq k} (t) = \sum_{0 \leq j \leq k} E_j(t).
\end{gather*}
Noting that
\bes
(\p_t \psi )^2 + |\nabla  \psi |^2  = \frac{1}{2} \big( ( L \psi )^2 + ( \Lb \psi )^2 \big) + |\slashed{\nabla} \psi |^2 = \frac{1}{2} \big( ( L \psi )^2 + ( \Lb \psi )^2 \big) + \sum_{1\leq i \leq 3} |\po_i \psi|^2.
\ees
The above defined energies are equivalent to
\begin{gather*}
	E_k(t) = \Big( \int_{\Sigma_t} \frac{1}{2}\left((L Z^k \theta )^2+(\Lb Z^k \theta)^2\right) + |\slashed{\nabla}  Z^k \theta |^2 \di x \Big)^{\frac{1}{2}} \\\quad\quad\qquad+ \Big(\int_{\Sigma_t} \frac{1}{2}\left((L Z^k \phi )^2+(\Lb   Z^k\phi)^2\right) + |\slashed{\nabla}  Z^k \phi |^2  \di x \Big)^{\frac{1}{2}}, \\   E_{\leq k} (t) = \sum_{0 \leq j \leq k} E_j(t).
\end{gather*}

We now formulate the global existence of small solution in region I.

\begin{theorem}
  Assume the solution to the equation \eqref{Eq} satisfies \eqref{improve} on $C_\delta$, then there exists a uniquely global solution to the equation \eqref{Eq}  in region I. Moreover, the solution satisfies the following energy estimate
  \be
  E_{\leq N-3} (t) \lesssim \delta^{\frac{3}{4}}  I_N(\theta_0, \theta_1, \phi_0, \phi_1),
  \ee
  for $t \geq 1$, $N \geq 6$ and $I_N(\theta_0, \theta_1, \phi_0, \phi_1)  $ depends on the initial data up to $N+1$ order of derivatives.
\end{theorem}

\begin{remark}
	Since the smallness of the solution in the sense of $L^{\infty}$ defined on the slice $C_\delta$ depends on the energies obtained in region II. As the new data, we consider the small solution in region I, we will lose three derivatives in the energy $E_{N-3}$.
\end{remark}

\begin{proof}
  The proof is also based on the bootstrap argument. Assume the solution exists up to time $t$ and for all $1 \leq t' \leq t$, and there is a large constant $M$, which is to be determined later and only depends on the initial data , such that
  \bes
  E_{\leq N-3} (t') \leq  \delta^{\frac{3}{4}} M.
  \ees
  Recall the pointwise decay estimates in region I is \eqref{pointdecay},
  \begin{align*}
  \big|\p Z^k \theta \big| \lesssim  (1+ |\ub|)^{-1} \, (1+|u|)^{-\frac{1}{2}} \, \delta^{\frac{3}{4}} M, \qquad k \leq N-3.
  \end{align*}
  By Lemma \ref{Sob-L-I}, we have  \bes
  \big|Z^k \theta \big| \lesssim (1+ |\ub|)^{-1} \, (1+|u|)^{\frac{1}{2}} \, \delta^{\frac{3}{4}}  M, \qquad k \leq N-3.
  \ees

   For good derivatives, we have
  \bes
  |\po Z^k \theta| \lesssim \frac{1}{1+t+r} |Z^{k+1} \theta|  \lesssim (1+ |\ub|)^{-2} \, (1+|u|)^{\frac{1}{2}} \, \delta^{\frac{3}{4}}  M, \qquad k \leq N-4.
  \ees
  The same estimates hold for $\phi$ as well.

  We take $\p_t = \frac{1}{2} (L+\Lb) $ as the multiplier,
  and multiplying $\p_t Z^k \theta$ and $\p_t Z^k \phi$ on both sides of \eqref{Eq2-1} and \eqref{Eq2-2} respectively, and integrating by parts over the domain $D_{I,t}$. By analogous calculations to those in Section \ref{3.3}, we get
  \begin{align}
  &\iint_{D_{I, t}} \big(\Box Z^k \theta - g_1^{\alpha \beta} \p_\alpha  \p_\beta Z^k \theta  -h_1^{\alpha \beta} \p_\alpha \p_\beta Z^k \phi \big) \cdot \p_t Z^k \theta  \nnb \\
  & \qquad \qquad + \big(\cos^2 \theta \Box Z^k \phi - g_2^{\alpha \beta} \p_\alpha \p_\beta Z^k \phi -h_2^{\alpha \beta} \p_\alpha \p_\beta Z^k \theta \big) \cdot \p_t Z^k \phi \nnb \\
  = & \, \Eo^2_{k}(\theta, \phi, \p_t) -  \Eo^{2}_{k, u= \delta}(\theta, \phi, \p_t) \nnb \\
   &  - \iint_{D_{I, t}} \frac{1}{2} \big( K^{ L} [Z^k \theta]+ K^{ \Lb}[Z^k \theta] + K^{\cos^2 \theta L} [Z^k \phi]+ K^{ \cos^2 \theta \Lb}[Z^k \phi] \big) \nnb \\
   & \qquad \qquad - \big( Q_k^{(1)} + Q_k^{(2)} + Q_k^{(3)} +Q_k^{(4)}  \big), \label{int}
  \end{align}
  where
  \begin{align*}
  Q_{k}^{(1)} = & \, \frac{1}{2} \p_t g_1^{\alpha \beta} \, \p_\alpha Z^k \theta \, \p_\beta Z^k \theta - \p_\alpha g_1^{\alpha \beta} \, \p_\beta Z^k \theta \, \p_t Z^k \theta ,\\
  Q_{k}^{(2)} = & \, \p_t h_1^{\alpha \beta} \, \p_\alpha Z^k \theta \, \p_\beta Z^k \phi - \p_\alpha h_1^{\alpha \beta} \, \p_\beta Z^k \phi \, \p_t Z^k \theta ,
   \\
  Q_{k}^{(3)} = & \, \frac{1}{2} \p_t g_2^{\alpha \beta} \, \p_\alpha Z^k \phi \, \p_\beta Z^k \phi - \p_\alpha g_2^{\alpha \beta} \, \p_\beta Z^k \phi \, \p_t Z^k \phi , \\
  Q_{k}^{(4)} = & \,- \p_\alpha h_2^{\alpha \beta} \, \p_\beta Z^k \theta \, \p_t Z^k \phi.
  \end{align*}
  Furthermore,
  \begin{align*}
  \Eo^2_{k}(\theta, \phi, \p_t) \sim & \, \frac{1}{4}\int_{\Sigma_{t}} \big( ( L Z^k \theta)^2 + ( \Lb Z^k \theta)^2 + 2 |\slashed{\nabla}  Z^k \theta|^2 \big) + \big( ( L Z^k \phi)^2 + ( \Lb Z^k \phi)^2 + 2 |\slashed{\nabla}  Z^k \phi|^2 \big) \\
  = & \, \frac{1}{2}\int_{\Sigma_{t}} \big( (\p_t  Z^k \theta )^2 + |\nabla  Z^k \theta |^2 \big) + \big( (\p_t  Z^k \phi )^2 + |\nabla  Z^k \phi |^2 \big),
  \end{align*}
  \begin{align*}
  \Eo^{2}_{k, u= \delta}(\theta, \phi, \p_t) \sim  \int_{C_{\delta}} \big( |L  Z^k \theta|^2 +  |\slashed{\nabla}  Z^k \theta|^2 \big) + \big( |L  Z^k \phi|^2 +  |\slashed{\nabla}  Z^k \phi|^2 \big) + ( L\phi \Lb Z^k \theta - L \theta \Lb Z^k \phi )^2.
  \end{align*}

 We now focus on the first term on the left-hand side of equation \eqref{int}. In line with \eqref{Eq2-1}, we start by examining the first term on the left-hand side of \eqref{Eq2-1}, which is
  \be\label{worst}
  \iint_{D_{I, t}} \Big| \sum_{\substack{|i_1| + \cdots + |i_m| \leq |k| \\ 2 \leq m \leq |k|+2}} a_{i_1, \cdots, i_m} \, (\sin (2\theta))^{(m-2)} Q(Z^{i_1} \phi, Z^{i_2} \phi) Z^{i_3} \theta \cdots Z^{i_m} \theta \cdot \p_t Z^k \theta \Big|.
  \ee
  Since for any null form, we have (see Lemma \ref{null form})
  \bes
   | \mathcal{Q} (\xi, \chi) | \lesssim | \po \xi| \, | \p \chi| + | \p \xi| \, | \po \chi|,
  \ees
  then we can control the worst term in \eqref{worst} as
  \begin{align*}
  & \iint_{D_{I, t}}\sum_{\substack{|i_1|, |i_2| < |k/2| \\  |i_3| \leq |k| \\ m=3 }} \Big| Q(Z^{i_1} \phi, Z^{i_2} \phi) Z^{i_3} \theta  \cdot \p_t Z^k \theta \Big| + \sum_{\substack{ |k/2| \leq |i_1| \leq |k|\\  |i_2|, |i_3| \leq |k/2|\\ m=3 }} \Big| Q(Z^{i_1} \phi, Z^{i_2} \phi) Z^{i_3} \theta  \cdot \p_t Z^k \theta \Big| \\
  \lesssim & \iint_{D_{I, t}}\sum_{\substack{|i_1|, |i_2| < |k/2| \\  |i_3| \leq |k| }}  \big( | \po Z^{i_1}\phi| \, | \p  Z^{i_2} \phi| + | \p Z^{i_1}\phi| \, | \po  Z^{i_2} \phi| \big) (1+ |u|) |\p Z^{i_3-1} \theta | \, |  \p_t Z^k \theta | \\
  & \qquad + \sum_{\substack{ |k/2| \leq |i_1| \leq |k|\\  |i_2|, |i_3| \leq |k/2|}}  \big( | \po Z^{i_1}\phi| \, | \p  Z^{i_2} \phi| + | \p Z^{i_1}\phi| \, | \po  Z^{i_2} \phi| \big)  | Z^{i_3} \theta | \, |  \p_t Z^k \theta | \\
   \lesssim &  \iint_{D_{I, t}} (1+ |\ub|)^{-2}   (1+|u|)^{\frac{1}{2}}   (1+ |\ub|)^{-1}   (1+|u|)^{-\frac{1}{2}} \big( \delta^{\frac{3}{4}}  M' \big)^2 (1+ |u|) |\p Z^{k-1} \theta | \, |  \p_t Z^k \theta | \\
   & \quad + (1+ |\ub|)^{-1}   (1+|u|)^{-\frac{1}{2}}   (1+ |\ub|)^{-1}   (1+|u|)^{\frac{1}{2}} \big( \delta^{\frac{3}{4}}  M' \big)^2  | \p Z^{i_1}\phi| \, |  \p_t Z^k \theta | \\
   \lesssim & \, \big( \delta^{\frac{3}{4}}  M' \big)^2  \Big( \int_{1}^{t}  (1+ \tau)^{-2} \Vert \p Z^{k-1} \theta  \Vert_{L^2(\Sigma_\tau)}^2 \di \tau \Big)^{\frac{1}{2}}  \Big( \int_{1}^{t}  (1+ \tau)^{-2} \Vert \p_t Z^{k-1} \theta  \Vert_{L^2(\Sigma_\tau)}^2 \di \tau \Big)^{\frac{1}{2}} \\
   \lesssim & \, \big( \delta^{\frac{3}{4}}  M\big)^2 \cdot \delta^{\frac{3}{4}}  M \cdot \delta^{\frac{3}{4}}  M  \lesssim \delta^3 M^4.
  \end{align*}

  The remaining terms on the left-hand side of equation \eqref{Eq2-1} are of higher order and contribute additional decay. Consequently, the first term on the left-hand side of equation \eqref{int} is dominated and can be bounded by $\delta^3 M^4$.

  An analogous calculation shows the second term in the left-hand side of \eqref{int} is also bounded by $\delta^3 M^4$.

  For the terms $K$, we have
  \begin{align*}
     & \, \iint_{D_{I, t}} \frac{1}{2} \Big| K^{ L} [Z^k \theta]+ K^{ \Lb}[Z^k \theta] + K^{\cos^2 \theta L} [Z^k \phi]+ K^{ \cos^2 \theta \Lb}[Z^k \phi] \Big| \\
    = & \, \iint_{D_{I, t}} \Big| Q( \cos^2 \theta, Z^k \phi) \, \p_t Z^k \phi - \frac{1}{2} \p_t (\cos^2 \theta) \, |\nabla  Z^k \phi |^2  \Big| \\
    \lesssim & \,  \iint_{D_{I, t}} \big|\sin(2 \theta) \big| \, \big|\p \theta \big| \, \big| \p Z^k \phi \big|^2  \\
    \lesssim & \, \iint_{D_{I, t}} (1+ |\ub|)^{-1} \, (1+|u|)^{\frac{1}{2}} \, \delta^{\frac{3}{4}}  M \cdot (1+ |\ub|)^{-1} \, (1+|u|)^{-\frac{1}{2}} \, \delta^{\frac{3}{4}}  M \cdot | \p Z^k \phi |^2
     \lesssim  \delta^3 M^4.
  \end{align*}

  For the terms $Q$, we only need to handle the term $\frac{1}{2} \p_t g_1^{\alpha \beta} \, \p_\alpha Z^k \theta \, \p_\beta Z^k \theta $ in $Q_k^{(1)}$, as the remaining terms can be controlled in the same way.
  As
  \begin{align*}
     & \, \iint_{D_{I, t}} \Big| \frac{1}{2} \p_t g_1^{\alpha \beta} \, \p_\alpha Z^k \theta \, \p_\beta Z^k \theta \Big| \\
    \lesssim & \, \iint_{D_{I, t}} \big| \p^2 \phi \big| \, \big| \p \phi \big| \,  \big|\p Z^k \theta \big|^2 \\
    \lesssim & \, \iint_{D_{I, t}} \big((1+ |\ub|)^{-1} \, (1+|u|)^{-\frac{1}{2}} \, \delta^{\frac{3}{4}}  M\big)^2 \big|\p Z^k \theta \big|^2
     \lesssim \delta^3 M^4.
  \end{align*}

  For the boundary term $\Eo^{2}_{k, u= \delta}(\theta, \phi, \p_t)$, according to \eqref{improve}, we have
  \begin{align*}
  \Eo^{2}_{k, u= \delta}(\theta, \phi, \p_t) \lesssim & \,  \int_{1-\delta}^{\ub} \int_{S_{\delta, \ub'}} (\delta^{\frac{3}{4}} |\ub'|^{-2} M)^2 \, r^2 \, \di \ub' \, \di \sigma_{S^2} \\
   & \, \lesssim  \int_{1-\delta}^{\ub} \big( \int_{S_{\delta, \ub'}} \delta^{\frac{3}{2}} |\ub'|^{-4} M^2 \, |\ub'|^2 \, \di \sigma_{S^2} \big) \, \di \ub'
    \lesssim
   \delta^{\frac{3}{2}} M^2, \qquad k \leq N-3.
  \end{align*}

  Combining the above estimates, we deduce that
   \bes
   \Eo^2_{k}(\theta, \phi, \p_t) \lesssim \delta^{\frac{3}{2}} ( I_N(\theta_0, \theta_1, \phi_0, \phi_1))^2 + \delta^3 M^4
   \leq  C_1^2 \Big( \delta^{\frac{3}{2}} ( I_N(\theta_0, \theta_1, \phi_0, \phi_1))^2 + \delta^3 M^4 \Big).
   \ees

  Choosing $M > 2 C_1 I_N(\theta_0, \theta_1, \phi_0, \phi_1)$ and $\delta < (2 C_1 M)^{-\frac{4}{3}}$, we can obtain
  \bes
   \Eo^2_{k}(\theta, \phi, \p_t) \leq \frac{1}{2} \delta^{\frac{3}{2}} M^2.
  \ees
  At last, utilizing the standard continuity argument, the solution exists globally in the region I.

\end{proof}

\begin{remark}
The set of initial data satisfying conditions \eqref{Con1}-\eqref{Con3} are indeed non-empty and the proof is employing a method analogous to that in \cite{W-W}, we neglect the details here.
\end{remark}

\vskip 4mm

\noindent{\Large {\bf Acknowledgements.}} J.W. is supported NSFC (Grant No. 12271450). C.W. is by the Outstanding Youth Fund of Zhejiang Province (Grant No. LR22A010004), the NSFC (Grant No. 12071435).

	\vskip 4mm
\noindent{\Large {\bf Declaration.}}The authors have no relevant financial interests to disclose.
\vskip 4mm
\noindent{\Large {\bf Data availability statement.}} No new data were created or analysed in this study.
\vskip 4mm
\noindent{\Large {\bf Ethical approval.}} Not applicable.
\vskip 4mm
\noindent{\Large {\bf Consent to participate.}} Not applicable.

\end{document}